\pdfoutput=1
\documentclass[letterpaper, 11pt,  reqno]{amsart}

\usepackage[margin=1.2in,marginparwidth=1.5cm, marginparsep=0.5cm]{geometry}

\usepackage{amsmath,amssymb,amscd,amsthm,amsxtra, esint, xcolor}

\usepackage[implicit=true]{hyperref}

\allowdisplaybreaks[2]

\usepackage{enumitem}
\usepackage{cancel}

\usepackage[mathscr]{euscript}

\usepackage{color}

\newtheorem{theorem}{Theorem} [section]

\newtheorem{lemma}[theorem]{Lemma}
\newtheorem{proposition}[theorem]{Proposition}

\newtheorem{corollary}[theorem]{Corollary}

\theoremstyle{definition}
\newtheorem{definition}[theorem]{Definition}
\newtheorem{remark}[theorem]{Remark}

\DeclareMathOperator*{\intt}{\int}

\DeclareMathOperator*{\supp}{supp}

\newcommand{\noi}{\noindent}

\newcommand{\R}{\mathbb{R}}

\newcommand{\T}{\mathbb{T}}
\newcommand{\N}{\mathbb{N}}

\newcommand{\ph}{\varphi}

\newcommand{\Law}{\mathrm{Law}}

\let\P= \undefined
\newcommand{\P}{\mathbf{P}}

\newcommand{\Q}{\mathbf{Q}}

\newcommand{\U}{\Upsilon}

\newcommand{\E}{\mathbf{E}}

\DeclareMathOperator{\rad}{rad}

\newcommand{\M}{\mathcal{M}}

\newcommand{\F}{\mathcal{F}}

\newcommand{\al}{\alpha}
\newcommand{\be}{\beta}

\newcommand{\dl}{\delta}

\newcommand{\eps}{\varepsilon}

\newcommand{\g}{\gamma}
\newcommand{\G}{\Gamma}
\newcommand{\ld}{\lambda}
\newcommand{\Ld}{\Lambda}
\newcommand{\s}{\sigma}

\newcommand{\ft}{\widehat}

\newcommand{\wt}{\widetilde}
\newcommand{\cj}{\overline}

\newcommand{\ta}{\theta}

\renewcommand{\o}{\omega}
\renewcommand{\O}{\Omega}

\newcommand{\les}{\lesssim}
\newcommand{\ges}{\gtrsim}

\newcommand{\jb}[1]
{\langle #1 \rangle}

\newcommand{\ind}{\mathbf{1}}

\newcommand{\NN}{\mathcal{N}}

\numberwithin{equation}{section}
\numberwithin{theorem}{section}

\makeatletter
\@namedef{subjclassname@2020}{\textup{2020} Mathematics Subject Classification}
\makeatother

\begin{document}
\baselineskip = 14pt

\title[Radial $\Phi^4$ models in supercritical regimes]
{Radial $\Phi^4$ models in supercritical regimes}

\author[T.~Oh, L.~Tolomeo, N.~Tzvetkov, Y.~Wang]
{Tadahiro Oh, Leonardo Tolomeo, Nikolay Tzvetkov, and Yuzhao Wang}

\address{
Tadahiro Oh, School of Mathematics\\
The University of Edinburgh\\
and The Maxwell Institute for the Mathematical Sciences\\
James Clerk Maxwell Building\\
The King's Buildings\\
Peter Guthrie Tait Road\\
Edinburgh\\
EH9 3FD\\
United Kingdom, and School of Mathematics and Statistics, Beijing Institute of Technology, Beijing 100081, China}

\email{hiro.oh@ed.ac.uk}

\address{
Leonardo Tolomeo,  School of Mathematics\\
The University of Edinburgh\\
and The Maxwell Institute for the Mathematical Sciences\\
James Clerk Maxwell Building\\
The King's Buildings\\
Peter Guthrie Tait Road\\
Edinburgh\\
EH9 3FD\\
United Kingdom}

\email{l.tolomeo@ed.ac.uk}

\address{Nikolay Tzvetkov, Ecole Normale Sup\'erieure de Lyon, UMPA, UMR CNRS-ENSL 5669, 46, all\'ee d'Italie, 69364-Lyon Cedex 07, France}
\email{nikolay.tzvetkov@ens-lyon.fr}

\address{
Yuzhao Wang\\
School of Mathematics\\
Watson Building\\
University of Birmingham\\
Edgbaston\\
Birmingham\\
B15 2TT\\ United Kingdom}

\email{y.wang.14@bham.ac.uk}

\begin{abstract}
We consider the $\Phi^4_d$ measure under radial symmetry on the unit ball, formally given by
$$d\rho(u) = \frac 1 Z \exp\Big( - \frac 14 \int_{B_1} |u|^4 dx - \frac 12 \int_{B_1} |\nabla u|^2 dx\Big) d u, $$
where $u$ is restricted to radially symmetric functions with 0 boundary conditions.
We prove that for every dimension $d$, there exists such a measure, which is concentrated on $L^p(B_1)$ for every $p < \frac{2d}{d-2}$ (for every $p < \infty$ if $d \le 2$).
Furthermore, we show that these measures are not Gaussian, and when $d \ge 4$, they are singular with respect to the underlying Gaussian measure.

We achieve these results by first defining the measure $\rho^+_\dl[\lambda]$ on the shell $\{\dl \le |x| \le 1\}$, $0 < \dl < 1$, with a boundary condition on the sphere of radius $\dl$ of the form $u|_{\{|x| = \delta\}} = \lambda$. Then we prove that this measure satisfies a variational formula, and use it to show moment bounds which are uniform in the boundary value $\lambda$.  A crucial ingredient in this argument is the fact that every solution of the elliptic equation $-\Delta u + u^3 = 0$ on the annulus $\{ \delta \le |x| \le 1 \}$ satisfies uniform bounds inside the domain, independently of the boundary conditions.
\end{abstract}



\maketitle

\tableofcontents

\newpage

\section{Introduction}

We consider the $\Phi^4_d$ measure defined on distributions over the ball {$B_1 = \cj{B(0,1)} = \{x \in \R^d: |x| \le 1\} \subset \R^d$} of radius $1$ centered at $0$ with radial symmetry, formally given by
\begin{align}
\label{Phi44}d \rho = \frac 1 Z \exp\Big( - \frac 14 \int_{B_1} |u|^4 dx - \frac 12 \int_{B_1} |\nabla u|^2 dx\Big) d u,
\end{align}

\noi
where $u$ is (formally) a radially symmetric function with zero Dirichlet boundary condition on {$\partial B_1 = \{x\in \R^d: |x| = 1\}$}.
Here, $\mu$ is the radial Gaussian free field on $B_1$, formally given by
$$d \mu = ``\frac 1 {Z_\mu}  \exp\bigg(- \frac12 \int_{B_1} |\nabla u|^2 dx\bigg) d u" .$$
See \eqref{GM_B} and \eqref{law}.

\subsection{Background}

The $\Phi_d^4$ measures play important roles in constructive quantum field theory, singular SPDEs, and dispersive equations with random data.
We consider the Hamiltonian
\begin{align}
\label{Hamiltonian}
H(u) = \int_{\mathcal M}  \bigg( \frac{1}{2} |\nabla u (x)|^2 + \frac{1}{4} |u (x)|^4\bigg) dx ,
\end{align}

\noi
where $\mathcal M$ is a $d$-dimensional manifold (possibly with boundary), such as the torus $\T^d$, the Euclidean space $\R^d$, or the unit ball $B_1 \subset \R^d$.
In this paper, we focus on the case $\M = B_1$.
We formally define the $\Phi_d^4$ measures as
\begin{align}
\label{Phi4d_f}
d \Phi_d^4 (u) : = \frac1Z e^{-H(u)} du = \frac1Z \exp \bigg( - \int_{\M} \bigg( \frac{1}{2} |\nabla u (x)|^2 + \frac{1}{4} |u (x)|^4\bigg) dx \bigg) d u.
\end{align}

\noi
The construction and properties of $\Phi_d^4$ measures have been extensively studied since the 1970s by the constructive quantum field theory community.
The construction of the $\Phi_1^4$ measure is straightforward; see \cite{Bo94}.
The construction of $\Phi_d^4$ measures for $d = 2,3$ is more involved, and renormalisations are needed; see \cite{Nel66} for $d = 2$, and \cite{GJ73, Fe74, FO76, MS76} for $d = 3$ (see also \cite{BarGub20, GH21} for more recent constructions).
However, when $d \ge 4$, the $\Phi_d^4$ measure \eqref{Phi4d_f} is expected to be trivial: for $d \ge 5$, the triviality of the scaling limits of lattice $\Phi^4_d$ models was proved by Aizenman \cite{Aiz} and Fr\"ohlich \cite{Fro}, and for $d = 4$, their marginal triviality was established in the recent breakthrough of Aizenman--Duminil-Copin \cite{ADC21}.
(Marginal) triviality is a strong non-existence result for the object \eqref{Phi4d_f}.

In this paper, we focus on the construction of the $\Phi^4_d$ measure \eqref{Phi44},
which is supported on distributions on the $d$-dimensional ball $B_1$, i.e., \eqref{Phi4d_f} with $\M = B_1$ and $d \ge 4$.
In view of the (marginal) triviality results in \cite{Aiz,Fro,ADC21} when $d \ge 4$,
we propose to study the $\Phi^4_d$ measure supported on distributions under extra symmetry, as the original measure is expected to display (marginal) triviality.
In the current paper,
we are interested in the $\Phi^4_d$ measure \eqref{Phi44} supported on {\it radially symmetric} distributions on the ball $B_1$.

This question can be seen as a one-dimensional case with non-flat geometry $r^{d-1} dr$, exhibiting distinct difficulties from the flat cases.
Thanks to the radial symmetry, locally outside the origin, the Gaussian free field $\mu$ is supported on functions of positive regularity (see Lemma~\ref{LEM:Xal}), which makes the radial $\Phi^4_d$ measure easier to study. We point out, however, that this property does not extend all the way to the origin, and the scaling of the radial $\Phi^4_d$ measure is identical to that of the non-radial one. In the non-radial case, this scaling makes the problem (super-)critical for $d \ge 4$, leading to the (marginal) triviality results \cite{Aiz,Fro,ADC21}. The main point of this paper is to study the behaviour of these measures when the singularity concentrates on a low-dimensional submanifold, which in this case is $0$-dimensional.

We now state a formal version of the main theorem. See Theorem \ref{THM:main} below for a rigorous statement.
\begin{theorem}[Existence and non-triviality of $\Phi^4_d$ measure]
For every $d \ge 1$, there exists a non-Gaussian $\Phi^4_d$ measure \eqref{Phi4d_f}, supported on $L^2_{\rm rad}(B_1)$, the {\it radially symmetric functions} in $L^2 (B_1)$.
\end{theorem}

When $1 \le d \le 3$, this result follows from the discussion in \cite{AT08}; see Proposition~\ref{PROP:conv} below.

\subsection{Gaussian free field on general compact domains}
\label{SUB:GFFg}
This subsection reviews the construction of the Gaussian measures $\mu$ defined on distributions on a general compact set $D$.
We start by recalling some basic properties of the Laplacian $-\Delta$ defined on a compact radial domain $D$ with a regular boundary.
In this paper, we mainly consider two types of radial domains:
\begin{itemize}
\item centred ball with radius $\dl$, $D = \cj{B_\dl} = \{x \in \R^d: |x| \le \dl\}$;

\item centred shell, $D = \cj{B_1 \backslash B_\dl} = \{x \in \R^d: \dl \le |x| \le 1\}$.
\end{itemize}

\noi
These domains are suitable for imposing radial assumptions.
Under the radial assumption, i.e., all functions under consideration are radial functions,
the operator $-\Delta$ with Dirichlet boundary conditions on $D$ has a sequence of eigenvalues $\ld_n^2$ with
\[
0 < \ld_0 \le \ld_1 \le \cdots \le \ld_n \le \cdots \longrightarrow \infty
\]
and the corresponding  normalised eigenfunctions $e_n$, satisfying
\begin{align}
\label{eigenD}
-\Delta e_n = \ld_n^2 e_n, \quad e_n(x) = 0 \textup{ for } x \in \partial D,
\end{align}
form an orthonormal basis of $L^2_{\rm rad} (D)$.
Throughout the paper, for radial $f$, we identify $\int_D f(x)\, dx$ with $\int f(r)\, r^{d-1} dr$, that is, we drop the surface area $\o_d = |\mathbb S^{d-1}|$ of the unit sphere, which only affects harmless multiplicative constants.
This convention also applies to the inner product of $L^2_{\rm rad} (D)$.
Where $\o_d$ appears explicitly, as in Remark~\ref{RMK:radial}, $dx$ denotes the Lebesgue measure.
Here, eigenvalues and eigenfunctions are dependent on the domain.
We will not explicitly present this dependence if it is clear from the context.
With the above notation, we define the free energy as
\begin{align}
\label{freeE}
H_0 (u) = \frac12 \int_D |\nabla u (x)|^2 dx.
\end{align}

\noi
Using the eigenbasis $\{e_n\}_{n \ge 0}$,
we can decompose any $u \in \mathcal D'_{\rm rad} (D)$ as
\[
u = \sum_{n=0}^\infty u_n e_n,
\]

\noi
where
\[
u_n = \jb{u,e_n} = \int_D u(x) e_n(x) dx.
\]

\noi
Then, under the coordinates $u = (u_n)_{n \ge 0}$, the energy functional \eqref{freeE} has the form
\[
H_0 (u) = H_0 \bigg(\sum_{n=0}^\infty u_n e_n \bigg) = \frac12 \sum_{n=0}^\infty \ld_n^2 |u_n|^2.
\]
In particular, $\|u\|_{\dot H^1(D)}^2 := \int_D |\nabla u|^2 dx = \sum_{n=0}^\infty \ld_n^2 |u_n|^2$. We denote by $H^1_0(D)$ the completion of the space of smooth functions with compact support in $D$ with respect to the norm $\|\cdot\|_{H^1(D)}$.

\noi
With the above notations, we may define the Gaussian measure with Cameron--Martin space $H^1_0 (D)$ formally given by
\begin{align}
d \mu = Z^{-1} e^{-\frac12 \|u\|^2_{\dot H^1}} du = Z^{-1} \prod_{n=0}^\infty e^{-\frac12 \ld_n^2 |u_n|^2} du_n,
\label{GM}
\end{align}

\noi
where $du_n$ is the Lebesgue measure on $\R$.
The Gaussian measure $\mu$ defined in \eqref{GM} is the induced probability measure under the map
\begin{align}
\label{GMap}
\o \in \O \longmapsto u^\o = \sum_{n=0}^\infty \frac{g_n (\o)}{\ld_n} e_n,
\end{align}

\noi
where $\{g_n\}_{n \ge 0}$ is a sequence of independent standard real-valued Gaussian random variables, i.e.\ $g_n \sim \NN (0,1)$, on a probability space $(\O,\mathcal F, \mathbb P)$.
From \eqref{GMap}, we see that
\begin{align}
\label{GMap2}
\E \big[ \|u^\o\|_{L^2 (D)}^2 \big] = \sum_{n=0}^\infty \ld_n^{-2} < \infty,
\end{align}

\noi
where in the last step, we used the radial Weyl law, i.e. $\ld_n \sim n$. From \cite[Lemma 2.2]{AT08}, we have that
\begin{align}
\label{GMap3}
\E \big[ \|u^\o\|_{L^p (D)}^p \big] = C_p \bigg\| \sum_{n \ge 0} \frac{e_n^2(x)}{\ld_n^2} \bigg\|_{L^{p/2} (D)}^{\frac{p}2} ,
\end{align}

\noi
where $C_p$ is a constant that only depends on finite $p \ge 2$.
The convergence property of \eqref{GMap3} depends on the choice of $D$, which will be made clear in the following subsections.

The formal definition \eqref{GM} can be made rigorous through a finite-dimensional approximation.
For this purpose, we define the spectral projection $P_N$ by
\begin{align}
\label{PN}
P_N u = \sum_{n=0}^N u_n e_n.
\end{align}

\noi
The image of $P_N$ is the finite-dimensional space
\[
E_N = {\rm span} \{e_0, \cdots, e_N\}.
\]

\noi
Consider the Gaussian measure on $\R^{N+1}$ given by
\[
d \mu_N = \prod_{n=0}^N \frac{\ld_n}{\sqrt{2\pi}} e^{-\frac12 \ld_n^2 |u_n|^2} du_n.
\]

\noi
By identifying $E_N$ with $\R^{N+1}$, this Gaussian measure also defines a probability measure on the finite-dimensional space $E_N$, which will be also denoted by $\mu_N$.
Similarly to \eqref{GMap},
the measure $\mu_N$ can also be viewed as the induced probability measure under the map
\begin{align}
\label{GMapN}
\o \in \O \longmapsto u^\o_N = \sum_{n=0}^N \frac{g_n (\o)}{\ld_n} e_n.
\end{align}

\noi
Given $s < \frac12$, the sequence $\{u_N^\o\}$ is a Cauchy sequence in $L^2 (\O; H^s_0(D))$ converging to $u^\o$ given in \eqref{GMap}.
The distribution of the random variable $u^\o \in H^s_0 (D)$ is the Gaussian measure $\mu$, which can be decomposed as
\[
\mu = \mu_N \otimes \mu_N^{\perp}.
\]

\noi
Here $\mu_N^\perp$ is the distribution of the random variable given by
\[
u^{\o,\perp}_N (x) = \sum_{n=N+1}^\infty \frac{g_n (\o)}{\ld_n} e_n.
\]

\subsection{Measure construction on unit balls}
In this subsection, we consider a particular bounded radial domain $D = B_1$, where $B_1 = \{x \in \R^d: |x| \le 1\}$ is the unit ball of $\R^d$.
A radial function on $B_1$ has values that depend only on $r = |x|$.
For $p \ge 1$ we write $L^p_{\rm rad}(B_1)$ for the closed subspace of radial functions in $L^p(B_1)$.
Under radial coordinates, $L^p_{\rm rad}(B_1)$ is the same as $L^p ([0,1], r^{d-1}dr)$.
We recall the radial eigenfunctions of $-\Delta$ on the unit ball from \cite{AT08}. We define radial functions
\begin{align}
\label{ef_n}
e_{n,d} (r) = \be_{n,d}^{-1} r^{-\frac{d-2}2} J_{\frac{d-2}2} (z_{n,d} r),
\end{align}

\noi
where $J_{\frac{d-2}2}$ is the Bessel function of order $\frac{d-2}2$, $z_{n,d}$ is the increasing sequence of its simple zeros on $(0, \infty)$, and
\begin{align}
\label{be_n}
\beta_{n,d} = \bigg( \int_0^1 \big|J_{\frac{d-2}2} (z_{n,d} r) \big|^2 r dr\bigg)^{\frac12}.
\end{align}

\noi
Then, $e_{n,d}$ for $n \ge 0$ are the eigenfunctions of the Laplace operator, corresponding to eigenvalues $z_{n,d}^2$, which form an orthonormal basis of $L^2 ([0,1], r^{d-1}dr)$.

With the above notations, we may define the Gaussian measure with Cameron--Martin space $H^1_0 (B_1)$ formally given by
\begin{align}
d \mu = Z^{-1} \prod_{n=0}^\infty e^{-\frac12 z_{n,d}^2 |u_n|^2} du_n,
\label{GM_B}
\end{align}

\noi
where $du_n$ is the Lebesgue measure on $\R$. See \eqref{GM}.
As in \eqref{GMap},
the Gaussian measure $\mu$ defined in \eqref{GM_B} is the induced probability measure under the map
\begin{align}
\label{GMap_B}
\o \in \O \longmapsto u^\o = \sum_{n=0}^\infty \frac{g_n (\o)}{z_{n,d}} e_{n,d}.
\end{align}

\noi
From \eqref{GMap_B}, we see that
\begin{align}
\label{GMap2_B}
\E \big[ \|u^\o\|_{L^2 (B_1)}^2 \big] = \sum_{n=0}^\infty z_{n,d}^{-2} < \infty,
\end{align}

\noi
where in the last step, we used again the radial Weyl law, i.e. $z_{n,d} \sim n$; see also \cite[(2.6)]{AT08}.
Define $p_c$ by setting
\begin{align}
\label{pc1}
p_c =
\begin{cases}
\frac{2d}{d-2}, & \text{if } d \ge 3,\\
\infty, & \text{if } d \le 2.
\end{cases}
\end{align}

\noi
Recall from \cite[Theorem 4]{AT08} that one has
\begin{align}
\label{GMap4_B}
\E \big[ \|u^\o\|_{L^p (B_1)}^p \big]
\begin{cases}
    < \infty, & 2 \le p < p_c,\\
    = \infty, & p > p_c.
\end{cases}
\end{align}

As to the critical case when $p = p_c$, we have the following result.
\begin{lemma}
\label{LEM:GMap_B}
Let $p_c$ be as in \eqref{pc1}. Then, with $u^\o$ defined in \eqref{GMap_B} and $d \ge 3$, we have
\begin{align}
\label{es:GMap_B}
\E \bigg[ \|u^\o\|_{L^{p_c} (B_1)}^{p_c} \bigg] = \infty.
\end{align}
As a consequence, we have
\begin{align}
\label{es:GMap_B2}
\|u^\o\|_{L^{p_c} (B_1)}=\infty,
\end{align}
almost surely.
\end{lemma}

\begin{proof}
The proof follows the same strategy as that of \cite[Theorem 3 and Theorem 4]{AT08}.
We establish only the estimate \eqref{es:GMap_B};
the bound \eqref{es:GMap_B2} then follows from Fernique's integrability theorem, as in the proof of \cite[Theorem 3]{AT08}.
From \eqref{GMap3}, we note that
\begin{align}
\label{GMap3_B}
\E \big[ \|u^\o\|_{L^p (B_1)}^p \big] = C_p \bigg\| \sum_{n \ge 0} \frac{e_{n,d}^2(x)}{z_{n,d}^2} \bigg\|_{L^{p/2} (B_1)}^{\frac{p}2} ,
\end{align}

\noi
where $C_p$ is a constant that only depends on finite $p \ge 2$.
It suffices to show that
\begin{align}
\label{es1}
\bigg\| \sum_{n \ge 0} \frac{e_{n,d}^2(x)}{z_{n,d}^2} \bigg\|_{L^{\frac{d}{d-2}} (B_1)} = \infty.
\end{align}

\noi
Recall from \cite[(2.6)]{AT08} that $z_{n,d} \sim n$, and from \cite[Lemma 2.6]{AT08} that
\[
|e_{n,d} (r) | \ges n^{\frac{d-1}2},
\]

\noi
provided  $n \ge 1$ and $rn \le \gamma$, for some constant $\gamma \in (0,1)$.
Therefore, given  $r \in (0,\gamma]$, we have
\[
\sum_{n =1}^\infty \frac{e_{n,d}^2(r)}{z_{n,d}^2} \ge \sum_{n =1}^{[\frac{\gamma}{r}]} \frac{e_{n,d}^2(r)}{z_{n,d}^2} \ges  \sum_{n =1}^{[\frac{\gamma}{r}]} n^{d-3} \ges r^{-d+2},
\]

\noi
provided $d \ge 3$.
It then follows that
\[
\bigg\| \sum_{n \ge 0} \frac{e_{n,d}^2(x)}{z_{n,d}^2} \bigg\|_{L^{\frac{d}{d-2}} (B_1)}^{\frac{d}{d-2}} \ges \int_0^{\gamma} r^{-d} r^{d-1} dr = \infty.
\]

\noi
Thus, we finish the proof.
\end{proof}

We now turn to the Gibbs measure defined on $B_1$.
Consider the nonlinear energy functional
\begin{align}
\label{energy}
H(u) = \frac12 \int_{B_1} |\nabla u|^2 dx + \frac1{p} \int_{B_1} |u|^{p} dx,
\end{align}

\noi
for $p > 2$.
To define the Gibbs measure
\begin{align}
\label{deGibbs}
d\rho (u) = \frac1Z \exp \bigg( - \frac12 \int_{B_1} |\nabla u|^2 dx - \frac1{p} \int_{B_1} |u|^{p} dx\bigg) du,
\end{align}

\noi
it suffices to define the weighted Gaussian measure
\begin{align}
\label{deGibbs2}
d\rho (u) = \frac1Z \exp \bigg( - \frac1{p} \int_{B_1} |u|^{p} dx\bigg) d\mu (u),
\end{align}

\noi
where $Z$ is a  normalising constant. Note that we abused the notation $Z$, which may differ from line to line.
We start with a frequency-truncated approximation of \eqref{deGibbs2}:
\begin{align}
\label{deGibbs2_N}
d\rho_N (u) = \frac1{Z_N} \exp \bigg( - \frac1{p} \int_{B_1} |u_N|^{p} dx\bigg) d\mu_N (u_N) \otimes d\mu_N^\perp (u_N^\perp),
\end{align}

\noi
where $u_N = P_N u$, $u_N^\perp = P^\perp_N u = u - u_N$, and the partition function $Z_N$ is given by
\begin{align}
\label{deGibbs3_N}
Z_N = \int \exp \bigg( - \frac1{p} \int_{B_1} |u_N|^{p} dx\bigg) d\mu (u) .
\end{align}

We show that $p_c$ is critical in the following sense.

\begin{proposition}\label{PROP:conv}
Let $d \ge 1$ and $p_c$ be as in \eqref{pc1}. Then the following statements hold:
\begin{enumerate}
    \item[{\rm (i)}]
    If $2 < p < p_c$ and $q \ge 1$, then we have
    \begin{align}
    \label{conv}
    \lim_{N \to \infty} \exp \bigg( - \frac1{p} \int_{B_1} |u_N|^{p} dx\bigg) = \exp \bigg( - \frac1{p} \int_{B_1} |u|^{p} dx\bigg)
    \end{align}

    \noi
    in $L^q (\mu)$.
    As a consequence, the Gibbs measure $\rho_N$ in \eqref{deGibbs2_N} converges, in total variation, to the defocusing Gibbs measure $\rho$ defined by \eqref{deGibbs2}. Furthermore, the limiting measure $\rho$ is absolutely continuous with respect to the base Gaussian free measure $\mu$ in \eqref{GM_B}.

    \item[{\rm (ii)}]
    If $d \ge 3$ and $p \ge p_c$, then we have
    \begin{align}
    \label{conv1}
    \lim_{N \to \infty} Z_N = \lim_{N \to \infty} \int \exp \bigg( - \frac1{p} \int_{B_1} |u_N|^{p} dx\bigg) d\mu (u) = 0,
    \end{align}
    where $Z_N$ is the partition function defined in \eqref{deGibbs3_N}. As a consequence, the defocusing Gibbs measure \eqref{deGibbs2} cannot be defined via the approximation \eqref{deGibbs2_N} as a weighted Gaussian measure, since the weights $\exp(- \frac1p \int_{B_1} |u_N|^p dx)$ converge to $0$ in $L^1(\mu)$.
\end{enumerate}
\end{proposition}

The proof of Proposition \ref{PROP:conv} follows from \eqref{GMap4_B} and \eqref{es:GMap_B} via standard arguments; see \cite[Section 3]{BTT10}.
In particular,
Proposition \ref{PROP:conv} states that,
when $d \ge 3$ and $p \ge p_c$,
the Gibbs measure \eqref{deGibbs2} cannot be constructed as a weighted Gaussian measure via the approximation \eqref{deGibbs2_N}.
One may wonder if a renormalisation, such as Wick ordering, may help us to obtain a non-trivial limit in \eqref{conv1}.
However, the answer is negative.
To explain this further,
given $x \in B_1$, let $u^\o_N = P_N u^\o$ as in \eqref{GMapN} with $D = B_1$.
Then $u^\o_N (x)$ is a mean-zero Gaussian random variable with variance
\[
\s_N (x) = \E \big[ |u_N^\o (x)|^2 \big] = \sum_{n=0}^N \frac{e_{n,d}^2 (x)}{z_{n,d}^2}.
\]

\noi
Therefore, from \eqref{GMap2_B}, we have
\[
\E \big[ \|u^\o_N\|^2_{L^2(B_1)}\big] = \int_{B_1} \s_N (x) dx = \sum_{n=0}^N \frac1{z_{n,d}^2} \le C < \infty,
\]

\noi
as $N \to \infty$.
This indicates that, unlike in the singular case, $u^\o_N$ does not diverge in $L^2(\O \times B_1)$, and its variance $\s_N$ does not diverge in $L^1(B_1)$. In fact, $\s_N (x)$ increases to $\E \big[ |u^\o (x)|^2 \big] < \infty$ for every $x \neq 0$; see \eqref{GFFX_e}.
This convergence does not provide uniform control near the origin, where the limiting variance is singular.
Thus, away from the origin there is no divergent constant for the Wick ordering to remove, and it may not help in removing the singularity at the origin from the nonlinear interaction.
To see this,
let us define the `Wick-ordered' polynomial $:(P_N u (x))^4:$, in analogy with the singular case \cite[Subsection 1.3]{OT18}, by
\[
:(P_N u (x))^4: = H_4 (P_N u (x); \s_N (x)) = (P_N u (x))^4 - 6 \s_N(x) (P_N u (x))^2 + 3 (\s_N(x))^2,
\]

\noi
where $H_k$ is the  Hermite polynomial of degree $k$.
For $N \in \mathbb N$, let
\begin{align}
\label{GNu}
G_N(u) = \frac14 \int_{B_1} :(P_N u (x))^4: dx  = \frac{1}4 \int_0^1  H_4 (P_N u (r); \s_N (r)) r^{d-1} dr,
\end{align}

\noi
where $u$ is a radial distribution.
If the sequence $G_N (u)$ converges in $L^2 (\mu)$ as $N \to \infty$, its limit would be a candidate for the Wick renormalisation of the Gibbs measure.
However, the following lemma shows that the functional $G_N(u)$ defined in \eqref{GNu} diverges in $L^2 (\mu)$.

\begin{lemma}
\label{LEM:GNu}
Let $d \ge 4$.
Then, $\{G_N\}_{N \in \mathbb N}$ diverges in $L^2 (\mu)$, where $\mu$ is the Gaussian measure given in \eqref{GM_B}.
\end{lemma}

\begin{proof}
From \cite[Lemma 3.1]{ORT23}, we note that
\[
\begin{split}
\liminf_{N \to \infty} \E & \bigg| \int_{B_1} H_4  (P_N u (x); \s_N (x)) dx \bigg|^2 \\
& = \liminf_{N \to \infty} \int_{B_1} \int_{B_1} \E \big[ H_4  (P_N u (x); \s_N (x)) H_4  (P_N u (y); \s_N (y)) \big] dxdy \\
& =  4! \liminf_{N \to \infty} \int_{B_1} \int_{B_1} \big(\E [P_N u (x)P_N u (y)] \big)^4 dxdy\\
& = 4! c \liminf_{N \to \infty} \int_{0}^1 \int_{0}^1 \bigg( \sum_{n=0}^N \frac{e_{n,d}(r) e_{n,d}(s)}{z_{n,d}^2} \bigg)^4 (rs)^{d-1} drds\\
& \ge 4! c \int_{0}^1 \int_{0}^1 K(r,s)^4 (rs)^{d-1} drds\\
& = \frac{4! c}{(d-2)^4} \int_0^1 \int_0^1 \bigg( \frac{1}{\max (r,s)^{d-2}} - 1 \bigg)^4 r^{d-1} s^{d-1} drds\\
& \ge \frac{4! c}{(d-2)^4} \int_0^1 \int_0^r \bigg( \frac{1}{r^{d-2}} - 1 \bigg)^4 r^{d-1} s^{d-1} \, ds dr \\
& \ges \int_0^{1/2} r^{-2d+ 7} dr = \infty,
\end{split}
\]

\noi
provided $d \ge 4$.
Here $K(r,s)$ denotes the radial Green's function defined in \eqref{Green_k}, and Lemma~\ref{LEM:Green_r} is applied in the fifth step.
\end{proof}

The above computation shows that the Wick ordering does not remove the singularity from the interaction,
and thus is not sufficient for renormalising the Gibbs measure \eqref{deGibbs2} with $p =  4$.
A similar argument applies in the critical and supercritical regimes $p \ge p_c$.

\medskip

We conclude this subsection by demonstrating a result that highlights how the singularity of the radial Gaussian free field arises exclusively from the origin of the ball,
 in sharp contrast with Lemma \ref{LEM:GMap_B}.

\begin{lemma}\label{LEM:GMap}
Let $\dl \in (0,1)$, $2 \le p < \infty$, and $u^\o$ be as in \eqref{GMap_B}.
Then we have
\begin{align}
\label{GMap5}
\E \big[ \|u^\o\|_{L^p (B_1 \backslash B_\dl)}^p \big] \les 1,
\end{align}

\noi
where the implicit constant depends on $\dl$ and $p$.
\end{lemma}

Lemma \ref{LEM:GMap} directly follows from \eqref{GMap3_B} and Lemma \ref{LEM:GFFX} with \eqref{GMap_B}.
We omit its proof.

\medskip

Inspired by Lemma~\ref{LEM:GMap},
we propose a novel approximation procedure to construct the Gibbs measure \eqref{deGibbs2}.
More precisely,
we approximate \eqref{deGibbs2} by removing the singularity concentrated at the origin, as opposed to the frequency truncation method employed in \eqref{deGibbs2_N}.
This suggests constructing \eqref{deGibbs2} by approximating the potential energy $\|u\|_{L^p(B_1)}^p$ in a non-perturbative way,
rather than treating it as a perturbation of the Gaussian measure as described in Proposition~\ref{PROP:conv}~(i) for $p < p_c$.

\subsection{Radial \texorpdfstring{$\Phi_d^4$}{Phi4d} measures and main results}

Let $\{e_n\}_{n\ge0}$ be an orthonormal basis of $L^2_{\mathrm{rad}} (B_1) $ made of eigenfunctions of $-\Delta$ (with Dirichlet boundary condition) with eigenvalues $0< \ld_0^2 \le \ld_1^2 \le \dotsb \le \ld_n^2 \le \dotsb$; namely, $e_n = e_{n,d}$ and $\ld_n = z_{n,d}$, see \eqref{ef_n}.
Then, define the random variable
\begin{align}
\label{gaussian} X := u^\o = \sum_{n \ge 0} \frac{g_n}{\ld_n} e_n,
\end{align}

\noi
where $\{g_n\}_{n \ge 0}$ is a sequence of independent standard real-valued Gaussian random variables on the probability space $(\O, \mathcal F, \mathbb P)$, as in \eqref{GMap} (see also \eqref{GMap_B}). Given any radial test function $\ph$ supported on $B_1$, it follows
\begin{align}
    \label{test}
    \E \big[ \jb{X,\ph}_{\dot H^1}^2 \big] = \|\ph\|_{\dot H^1}^2 = \sum_{n \ge 0} \ld_n^2 |\jb{\ph, e_n}|^2.
\end{align}
We define
\begin{align} \label{law}
\mu:= \Law(X).
\end{align}
See \eqref{GM} and \eqref{GM_B}.
Given $0 < \eps < 1$, we then consider the following weighted Gaussian measure:
\begin{equation}\label{rho_eps}
d \rho_\eps(u) = \frac 1 {Z_\eps} \exp\Big(- \frac 14 \int_{B_1\backslash B_\eps} |u|^4 dx \Big)d\mu(u),
\end{equation}
where $B_\eps = \{x \in \R^d: |x| < \eps\}$ denotes the open ball of radius $\eps$.
Here $Z_\eps$ is the partition function defined by
\begin{align}
\label{partition_eps}
Z_\eps = \int \exp\Big(- \frac 14 \int_{B_1\backslash B_\eps} |u|^4 dx \Big) d\mu(u).
\end{align}
From Lemma \ref{LEM:GMap},
we see that $Z_\eps$ is non-zero and thus the measure \eqref{rho_eps} is well defined for any $\eps \in (0,1)$.
Furthermore, at least formally, we expect that
\begin{align} \label{G_limit} \lim_{\eps \to 0^+} \rho_{\eps} = \rho_0,
\end{align}
for some well-defined probability measure $\rho_0$.
The main purpose of this paper is to justify the above limit \eqref{G_limit} (up to a subsequence) and study the properties of the limiting measure $\rho_0$.

Before stating our results, we specify the space in which the limit \eqref{G_limit} is taken.
We write $C^\al (E)$, $\al \in (0,1)$, for the H\"older space on a compact set $E$, and $c^\al (E)$ for the associated {\it little H\"older space}, namely the closure of $C^\infty (E)$ in $C^\al (E)$.
Given now $\al \in (0,\frac12)$, we let
\begin{align}
\label{Xp}
\mathbb X^{\al} := c^\al_{\rm loc, \rad} (B_1 \backslash \{0\})
\end{align}

\noi
be the space of radial functions $u$ on $B_1 \backslash \{0\}$ such that $u \in c^\al (B_1 \backslash B_\dl)$ for every $\dl \in (0,1)$, endowed with the topology induced by the family of seminorms
\begin{align}
\label{Xpsemi}
u \longmapsto \|u\|_{C^\al (B_1 \backslash B_{1/k})}, \qquad k \ge 2.
\end{align}

\noi
It is fairly standard to check that $\mathbb X^\al$ is a separable Fr\'echet space, and thus a Polish space.
Furthermore, when $d \ge 3$, the radial Gaussian free field $X$ in \eqref{gaussian} has a version with values in $\mathbb X^\al$, $0<\al<1/2$; see  Lemma \ref{LEM:Xal}.

We have the following results.

\begin{theorem}[Existence of the limiting measure] \label{THM:main}
Let $d \ge 1$ and $\rho_\eps$ be the Gibbs measures defined in \eqref{rho_eps}.
Then we have
\begin{enumerate}
    \item[{\rm (i)}] When $1 \le d \le 3$, the measures $\rho_\eps$ converge in total variation, as $\eps \to 0$, to the measure $\rho_0$ given by \eqref{deGibbs2} with $p = 4$, which is constructed in Proposition \ref{PROP:conv}~(i).
    \item[{\rm (ii)}] When $d \ge 4$,
    the sequence $\rho_\eps$ has a weak limit point $\rho_0$ in the space of probability measures over $\mathbb X^{\al}$. Moreover, if $\rho_0$ is any such limit point, then $\rho_0$ is concentrated in the space $L^{\frac{2d}{d-2}-}(B_1)$\footnote{Hereafter, we use $a-$ to denote $a - \varepsilon$ for arbitrarily small $\varepsilon > 0$.}.
\end{enumerate}
\end{theorem}

See also Lemma \ref{LEM:limit} for further details on the convergence in {\rm (ii)}.

The proof of Theorem \ref{THM:main}~(i) is similar to that of Proposition \ref{PROP:conv}~(i), as we have
\begin{align}\label{d123}
\lim_{\eps \to 0}  \int_{B_1\backslash B_\eps} |u|^{4} dx =  \int_{B_1} |u|^{4} dx , \,\, a.e.
\end{align}
with respect to the Gaussian measure $\mu$, and the right-hand side of \eqref{d123} is finite $\mu$-almost surely when $d = 1,2,3$, by \eqref{GMap4_B}.
However, when $d \ge 4$, the right-hand side of \eqref{d123} is infinite $\mu$-almost surely: for $d = 4$, this is \eqref{es:GMap_B2}, and for $d \ge 5$, it follows from \eqref{es:GMap_B2} and H\"older's inequality, since then $p_c < 4$. Thus, Theorem \ref{THM:main}~(ii) is non-trivial.
In what follows, we shall focus on Theorem \ref{THM:main}~(ii).

 \begin{remark}\rm
 In higher dimensions,
 the quartic interaction becomes increasingly energy-supercritical.
 This  suggests that, as the dimension grows, typical functions sampled from the limiting measure exhibit stronger singularities at zero.
 Nevertheless, independently of the dimension, the limiting measure is supported on
 $L^{2} (B_1)$,
 and thus on distributions over the ball.
 \end{remark}

\begin{remark}\rm
We expect that a similar argument should apply for general interactions $\int |u|^p\,dx$, for any $p > 2$. We expect the limiting measure to be always supported on the space
$L^{\frac{2d}{d-2}-}$.
However, the argument would be much more involved due to the non-algebraic
nonlinearity when $p \not\in 2\N$. Even when $p \in 2\N$, the high
power in the nonlinearity introduces many additional difficulties when characterising
the measure $\rho_{\delta}^+ [\lambda]$ in Section~\ref{SEC:rho+}. For simplicity, we
focus on the quartic interaction $\int |u|^4 dx$ in dimensions $d \geq 4$
throughout this work.
\end{remark}

We also have the following non-triviality result for the limiting measure constructed in Theorem \ref{THM:main}.

\begin{theorem}
\label{THM:main_sub}
If $\rho_0$ is any weak limit point of $\rho_\eps$ obtained in Theorem \ref{THM:main},
then $\rho_0$ has sub-Gaussian tails, in the following sense: for every radial $\ph \in C^\infty_c (B_1 \backslash B_{3/4})$, there exists $\eta > 0$ such that
\begin{align*}
\int \exp \big( \eta \jb{u,\ph}^4 \big) \, d \rho_0 (u) < \infty.
\end{align*}
Moreover, if $\ph \not\equiv 0$, then $\jb{u,\ph}$ is not $\rho_0$-almost surely constant. In particular, $\rho_0$ is not Gaussian.
\end{theorem}

In the case of the $\Phi^4_3$ measure on $\T^3$, a similar sub-Gaussian tail estimate was obtained by Hairer and Steele \cite{HS22}.

\begin{remark}\rm
The analogue of Theorem \ref{THM:main_sub} is known to fail without symmetry assumptions: Aizenman and Duminil-Copin \cite{ADC21} established the marginal triviality of the $\Phi^4_4$ model, namely that its scaling limits on $\R^4$ are Gaussian (for $d \ge 5$, this goes back to \cite{Aiz, Fro}).
A similar failure is expected on $B_1$ without the radial assumption for $d \ge 4$.
\end{remark}

In low dimensions, the Gibbs measure is mutually absolutely continuous with respect to the base Gaussian measure: this is well known for $d \le 2$, and under radial symmetry it also holds for $d = 3$, by Proposition \ref{PROP:conv}~(i).
However, although our measure $\rho_0$ is supported on radial functions (which can be viewed as one-dimensional objects),
it exhibits stronger singularities.
More precisely, we have:

\begin{theorem}
\label{THM:main_singular}
Let $d \ge 4$.
If $\rho_0$ is any weak limit point of $\rho_\eps$ obtained in Theorem \ref{THM:main},
then $\rho_0$ is singular with respect to the Gaussian measure $\mu$.
\end{theorem}

\begin{remark}\rm
(i) Theorems~\ref{THM:main}~(ii), \ref{THM:main_sub}, and~\ref{THM:main_singular} hold for every weak limit point $\rho_0$ of $\{\rho_\eps\}_{0 < \eps < 1}$.
We do not address the uniqueness of the limit in this paper, but we expect it to hold.
Indeed, by the Markov property, the restriction of $\rho_0$ to $B_1 \setminus B_\dl$ is a superposition of the measures $\rho^+_\dl[\ld]$ (see Subsection~\ref{SUB:strategy} below), and the bounds of Section~\ref{SEC:proof}, being uniform in the boundary value $\ld$, suggest that the dependence on $\ld$, except for its sign, disappears when taking limits as $\delta \to 0$. More specifically, we expect that
$$ \rho_0 = \frac12 \rho_0^+[+\infty] + \frac12 \rho_0^+[-\infty], $$
where $\rho_0^+[\pm \infty]$ is defined analogously to $\rho^+_\dl[\ld]$, by taking $\ld = \pm \infty$ and $\dl = 0$ in \eqref{z}.

\smallskip
\noi
(ii) One may also replace the potential $\frac14 \int |u|^4 dx$ by $\frac\gamma4 \int |u|^4 dx$ for a coupling constant $\gamma \in \R \setminus \{0\}$.
Since our arguments are non-perturbative and only use the defocusing nature of the interaction, the same results as in this paper hold for every $\gamma > 0$, and we expect the uniqueness discussed in (i) to hold for every $\gamma > 0$ as well.
This is in sharp contrast with renormalised models in infinite volume, where the uniqueness of the limiting measure depends on the size of the coupling constant (see, for instance, \cite{GJS75, CGW22}), and with critical models in the focusing case, where the construction depends on the size of the coupling constant; see \cite{OOT24, OOT25}.
For $\gamma < 0$, the measure is focusing, and it is not normalisable even in the presence of the cut-off $\eps > 0$, since $\int \exp\big(\frac{|\gamma|}{4} \int_{B_1 \setminus B_\eps} |u|^4 dx\big) d\mu(u) = \infty$. Following \cite{LRS88}, one may then add a cut-off on the $L^2$-norm; even so, when $d \ge 3$, we expect the resulting radial measure to be non-normalisable, by adapting the arguments in \cite[Section~4]{LRS88} and \cite[Section~5]{OST22}.
\end{remark}

\subsection{Strategy}

\label{SUB:strategy}

The result of Theorem \ref{THM:main} is strongly non-perturbative: in the energy-critical case $d = 4$ the potential energy has exactly the strength of the kinetic energy, and for $d > 4$ it is the stronger of the two, so that in neither case can the quartic interaction be regarded as a perturbation of the Gaussian measure $\mu$.
As far as we are aware, this is the first positive result on the $\Phi^4_d$ measure in this (super-)critical regime.
The proof rests on three main ingredients, which we now describe.

The first is the Markov property of the radial Gaussian free field \eqref{gaussian}. The Markov property plays a fundamental structural role in constructive quantum field theory, see for instance \cite{Nel73, Simon05}, and has recently been exploited in the context of $\Phi^4$ measures by Barashkov and Gunaratnam \cite{BG25}, who made use of it in order to construct a $\Phi^4_3$ measure on cylinders that satisfies Segal's axioms of quantum field theory. In our setting, the Markov property implies that outside of a small ball of radius $\dl > 0$, one can describe the measure only via the law of the boundary value $u(\dl)$ on the sphere $\partial B_\dl$. More precisely, we show in Section \ref{SEC:Markov} that, for $0 < \eps < \dl < 1$,
\begin{equation*}
(R_\dl)_{\#} \rho_\eps = \int \rho^+_\dl [\ld] \, \nu_{\eps,\dl} (d\ld) ,
\end{equation*}
where $R_\dl$ is the restriction to the annulus $B_1 \setminus B_\dl$, and $ \rho^+_\dl [\ld]$ denotes the radial $\Phi^4$ measure with boundary data $\ld$, formally given by
\begin{equation*}
d\rho^+_\dl [\ld] = \frac1{Z^+_\dl[\ld]} \exp \Big( - \frac14 \int_{B_1 \backslash B_\dl} |u|^4 dx -\frac12 \int_{B_1 \backslash B_\dl} |\nabla u (x)|^2 dx\Big) \ind_{\{u(\dl) = \ld, u(1) = 0\}}\, du.
\end{equation*}
Here, $\nu_{\eps,\dl}$ is the law of the boundary value $u(\dl)$ when $\Law(u) = \rho_\eps$; see \eqref{nued}.
Therefore, in order to prove tightness of the sequence $\rho_\eps$, it is enough to prove that the measure $\rho^+_\dl [\ld]$ enjoys uniform bounds that are independent of the boundary value $\ld$.

This is what brings us to the second ingredient. Heuristically, one would expect that the measure $\rho^+_\dl [\ld]$ concentrates around the minimisers of the expression
\begin{equation}
E_\dl(u) = \frac14 \int_{B_1 \backslash B_\dl} |u|^4 dx +  \frac12 \int_{B_1 \backslash B_\dl} |\nabla u (x)|^2 dx \label{energy_intro}
\end{equation}
with boundary values $u(\dl) = \lambda$, $u(1) = 0$.
By taking variations in the variable $u$, such a minimiser will be the solution $z_{\dl,\ld}$ of
\begin{equation}
\label{zeqintro}
\begin{cases}
\Delta z (x) = z^3 (x), & x \in B_1\backslash B_\dl, \\
z(x) = \ld, & x \in \partial B_\dl, \\
z(x) = 0, & x \in \partial B_1.
\end{cases}
\end{equation}
It turns out that $z_{\dl,\ld}$ itself satisfies a uniform bound which is independent of $\ld$. Indeed, in Lemma \ref{LEM:z} we prove that
\begin{equation}
\label{zestintro}
|z_{\dl,\ld}(r)| \le \frac{\sqrt2}{r - \dl}, \qquad \dl < r \le 1.
\end{equation}
This can be seen as an elliptic instance of the ``coming down from infinity'' property of the dynamic $\Phi^4$ model, which was first observed by Mourrat and Weber in \cite{MW17}, and then further exploited in a series of works \cite{CMW23, GH19, GH21, MoW20a, MoW20b, OTWZ26, TW18}.

However, we point out that the estimate \eqref{zestintro} is a special case of the theory of maximal solutions in general domains for elliptic equations of the form
$$ \Delta u +hu - ku^p =0, $$
where $h,k$ are appropriate continuous functions. We refer to the works by Bandle and Marcus \cite{BM92, BM95}, V\'eron \cite{Ve92}, and Marcus and V\'eron \cite{MV97}.
We nevertheless give a self-contained proof of Lemma \ref{LEM:z}, which in the radial case is a short ODE argument.

The third ingredient is the variational formula of Bou\'e and Dupuis \cite{BD}, which allows us to rigorously connect the measure $\rho^+_\dl [\ld]$ with variational properties of the energy \eqref{energy_intro}. Informally speaking, this says that
\begin{align*}
&\log \int \exp\big(-F(u)\big) d\rho^+_\dl [\ld](u) \\
&= \inf_{V \in {\mathbb H}_\dl} \E\Big[\frac14 \int_{B_1 \backslash B_\dl} |X^+_\dl + \ld h_\dl + V|^4 dx +  \frac12 \int_{B_1 \backslash B_\dl} |\nabla V|^2 dx\Big] \\
&\phantom{=\ }- \inf_{V \in {\mathbb H}_\dl} \E\Big[F(X^+_\dl + \ld h_\dl + V) + \frac14 \int_{B_1 \backslash B_\dl} |X^+_\dl + \ld h_\dl + V|^4 dx \\
&\phantom{=\ - \inf_{V \in {\mathbb H}_\dl} \E\Big[} +  \frac12 \int_{B_1 \backslash B_\dl} |\nabla V|^2 dx\Big],
\end{align*}
where $X^+_\dl$ is the radial Gaussian free field on the shell $B_1 \backslash B_\dl$ with zero boundary values, $h_\dl$ is the radial harmonic function on the shell with $h_\dl = 1$ on $\partial B_\dl$ and $h_\dl = 0$ on $\partial B_1$ (see \eqref{hdelta}), and $\mathbb H_\dl$ is an appropriate space of random processes with $0$ boundary condition on $\partial (B_1 \setminus B_\dl)$. See Proposition~\ref{PROP:rho} for a precise statement.
This approach to constructing $\Phi^4$ measures was introduced by Barashkov and Gubinelli in \cite{BarGub20}, who already noticed that one can perform a change of variable in order to remove diverging terms in the expansion. In our setting, we perform the change of variable
$$ V = (z_{\dl,\ld} - \ld h_\dl) + \Xi + \Upsilon + W,$$
see \eqref{CoV} (with $V = \Theta(1)$) and Proposition \ref{prop:rhoVar}. Here the corrections are added in succession according to the following scheme:
\begin{itemize}
\item First we write $V = (z_{\dl,\ld} - \ld h_\dl) + V_1$, so that $X^+_\dl + \ld h_\dl + V = X^+_\dl + z_{\dl,\ld} + V_1$, to isolate the fact that the measure should concentrate around minimisers of the energy \eqref{energy_intro}.
\item Because of the divergence rate $|z_{\dl,\ld}| \les \frac{1}{r-\dl}$, this creates some diverging terms in the expression of the form $\int V_1 X^+_\dl z_{\dl,\ld}^2$, which we further remove by performing the change of variable $V_1 = \Xi + V_2$, where $\E \big[ |\Xi(r)|^2 \big] \les \frac{r - \dl}{r^{d-1}}$ by Lemma \ref{LEM:Xi}.
\item Finally, there is one last (possibly) diverging term, of the form $\int V_2 \Xi z_{\dl,\ld}^2$, and we tailor the final  $\Upsilon$ so that this term disappears after the change of variable $V_2 = \Upsilon + W$.
\end{itemize}
This allows us to show a variational formula for the measure $\rho^+_\dl [\ld]$, given by Proposition \ref{prop:rhoVar}. We also remark that in order to achieve this, we need to show an infinite-dimensional extension of the Bou\'e--Dupuis variational formula, that works for lower  semicontinuous functionals, see Corollary \ref{COR:BD}. In Section~\ref{SEC:proof}, we then exploit this to prove that, for every $\dl \in (0,\frac12]$ and $\al \in (0,\frac12)$,
\begin{align*}
\int \exp \Big( \eta \|u\|^2_{C^\al (B_1 \backslash B_{2\dl})} \Big) \, d\rho^+_\dl [\ld](u) \le C,
\end{align*}
for some constants $\eta, C > 0$ depending only on $\dl$ and $\al$ (see Lemma~\ref{LEM:EB}), together with exponential moment bounds, degenerating at an explicit rate as $r \to 0$, for $r^{d-2} u(r)^2$ and $\|u\|^4_{L^4 (B_1 \backslash B_r)}$ (see Proposition~\ref{PROP:PW}).
All these bounds are uniform in the boundary value $\ld \in \R$, and hence, by the Markov property, uniform in $\eps$ under $\rho_\eps$.
The first one allows us to build the limiting measure $\rho_0$ via a standard tightness argument, while the second one, after a dyadic summation, gives $\E_{\rho_0} \big[ \int_{B_1} |u|^p dx \big] < \infty$ for every $p < \frac{2d}{d-2}$ (see Lemma~\ref{LEM:L4-}).
Moreover, by showing an exponential moment bound for $\jb{u,\ph}^4$, where $\ph$ is a radial test function supported in $B_1 \backslash B_{3/4}$ (see Proposition~\ref{PROP:SG}), we prove that the tails of $\jb{u,\ph}$ under $\rho_0$ are sub-Gaussian.

Finally, in order to prove singularity, we base our argument on the fact that the candidate logarithm of the density
$$ f_\dl(u) = -\frac 14 \int_{B_1\setminus B_\dl} |u|^4 $$
in the limit satisfies a non-concentration inequality of the form
$$ \lim_{\dl \to 0} \sup_{Z \in \R} \mu(\{ |f_\dl - Z| \le 1 \}) = 0, $$
see also the discussion at the beginning of Subsection~\ref{SUB:sing}.

\medskip

The remainder of this paper is  organised as follows.
In Section~\ref{SEC:pre}, we recall several fundamental definitions and tools, including radial Green's functions, radial Gaussian free field estimates, comparison principles for elliptic operators, and the Bou\'e--Dupuis formula.
Section~\ref{SEC:Markov} is devoted to the Markov property of the radial Gaussian free field and the Gibbs measure $\rho_\varepsilon$.
In Section~\ref{SEC:rho+}, we  analyse the structure of $\rho_\delta^+[\lambda]$.
The proof of Theorem~\ref{THM:main} is presented in Section~\ref{SEC:proof}.
Finally, in Section~\ref{SEC:sing}, we establish Theorems~\ref{THM:main_sub} and~\ref{THM:main_singular}.

\begin{remark}\label{RMK:radial}
Because of the radial symmetry, one could rewrite the measure \eqref{Phi44} in radial coordinates as
$$ d\rho = \frac 1 Z \exp\Big( - \frac{ \o_d}{4} \int_0^1 |u(r)|^4 r^{d-1}dr - \frac {\o_d}{ 2} \int_0^1 |u'(r)|^2 r^{d-1} dr \Big) d u, $$
where $\o_d$ denotes the surface area of the unit sphere in $\R^d$,
and then use the Feynman--Kac formula to view this measure as the law of a one-dimensional diffusion in the variable $r$, weighted by an exponential factor.
Indeed, the Gaussian part of the right-hand side is the law of the time-changed Brownian motion $u(r) = B(\tau(r))$, where $B$ is a standard Brownian motion and $\tau(r) = \int_r^1 \frac{ds}{\o_d s^{d-1}}$ (so that $u(1) = 0$), whose squared Cameron--Martin norm is $\o_d \int_0^1 |u'(r)|^2 r^{d-1} dr$, while the quartic part is the Feynman--Kac weight $\exp \big( - \frac{\o_d}4 \int_0^1 u(r)^4 r^{d-1} dr \big)$.
In principle, one could then obtain the various results of this paper by a spectral analysis of the associated Schr\"odinger operators
$$ \mathscr L_r = - \frac1{2 \o_d\, r^{d-1}} \partial_u^2 + \frac{\o_d}4 r^{d-1} u^4 , $$
acting in the field variable $u$, with $r$ playing the role of time. While this approach is arguably easier than the one we decided to follow in this paper, it is extremely rigid, in the sense that it requires the measure under consideration to be a one-dimensional object. We believe that the approach that we chose to follow is more flexible, and could be adapted to related higher-dimensional problems.
\end{remark}

\subsection*{Acknowledgements}

The authors are grateful to Didier Smets and Denis Serre for bringing to their attention the literature on maximal solutions of semilinear elliptic equations.

T.O. was supported by the European Research Council (grant no.\ 864138 ``SingStochDispDyn'') and by the EPSRC Mathematical Sciences Small Grant (grant no.\ EP/Y033507/1), and also acknowledges support from the NSFC (grant no.\ W2531005).
L.T. is funded by the European Union, via the ERC Starting Grant ``CritPDEsRand'', grant agreement no.\ 101219486.
N.T. was partially supported by the ANR project Smooth ANR-22-CE40-0017.
Y.W. was supported by the EPSRC Mathematical Sciences Small Grant (grant no.\ UKRI1116).
Views and opinions expressed are however those of the authors only and do not necessarily reflect those of the European Union or the European Research Council Executive Agency.
Neither the European Union nor the granting authority can be held responsible for them.

\subsection*{Declaration on the use of AI}

A complete first draft of this paper, with full proofs, was written before any use of AI.
We then used the AI assistants Claude Opus~5 and Claude Opus~5.5 (Anthropic), with the ``Extra'' reasoning setting, to proofread the manuscript and to help us fix the errors and gaps that it found, including drafting the proofs of some auxiliary steps.
Every statement and every proof has been verified by us, and we take sole responsibility for any errors that remain.

\section{Preliminaries}
\label{SEC:pre}
In this section, we prepare some tools, such as radial Green function estimates, radial Gaussian field estimates, the comparison principle for Schr\"odinger operators, and the  Bou\'e--Dupuis formula, that will be used throughout the entire paper.

\subsection{Radial Green functions}

In this subsection, we recall some Green function formulas of the Laplacian on the radial balls $B_1 \subset \R^d$ and shells $B_1 \backslash B_\dl \subset B_1$.
For a detailed discussion on various Green functions, see \cite[Subsection 2.2.4]{EvansPDE} and references therein.

\subsubsection{Green functions on balls}
Let $u$ be a radial solution to the following boundary value problem
\begin{align}
\label{BVP}
\begin{cases}
- \Delta u (x) = f(x), & x \in B_1, \\
u(x) = 0, & x \in \partial B_1.
\end{cases}
\end{align}

\noi
We may abuse notation and write $u(r) = u(x)$ for $r = |x|$.
Then Poisson's equation \eqref{BVP} can be written as
\begin{align}
\label{BVP1}
\begin{cases}
- \partial_r^2 u (r) - \frac{d-1}r \partial_r u (r) = f(r), \quad  r \in (0,1), \\
u(1) = 0,\\
\partial_r u(0) = 0.
\end{cases}
\end{align}

\noi
Denote by $K(r,s)$ the radial Green function to the Laplacian operator.
Then, the solution to \eqref{BVP1} can be recast as
\begin{align}
\label{Green1}
u(r) = \int_0^1 K(r,s) f(s) s^{d-1} ds.
\end{align}

\noi
Here, $K(r,s)$ is the kernel of $(-\Delta)^{-1}$ on radial functions with respect to the measure $s^{d-1} ds$, with the convention on radial integrals of Subsection~\ref{SUB:GFFg}.
We also have
\begin{align}
\label{Green_k}
K(r,s) = \sum_{n=0}^\infty \frac{e_{n,d}(r) e_{n,d}(s)}{z_{n,d}^2},
\end{align}

\noi
for $r, s \in (0,1]$,
where $e_{n,d}$ are the radial eigenfunctions defined in \eqref{ef_n} and  $z_{n,d}^2$ are the corresponding eigenvalues.

\medskip

Then, we have the following formula.

\begin{lemma}
\label{LEM:Green_r}
With the above notation, we have
\begin{align}
\label{Green_r}
K(r,s) = \frac{1}{d-2} \bigg( \frac1{\max(r, s)^{d-2}} - 1 \bigg),
\end{align}

\noi
for $d \ge 3$.
\end{lemma}

\begin{proof}
Here, we provide a direct proof from \eqref{BVP1} for completeness.
Let $f$ be a smooth test function.
We first note from \eqref{BVP1} that
\[
- \partial_r (r^{d-1} \partial_r u (r)) = r^{d-1} f(r),
\]

\noi
for $r \in (0,1)$.
By integrating the above from $0$ to $r$, we get
\[
-r^{d-1} \partial_r u(r) = \int_0^r s^{d-1} f(s) ds,
\]

\noi
where we used the boundary condition $\partial_r u(0) = 0$.
We integrate the above again to get
\[
u(r) = \int_r^1 \bigg( \int_0^\tau \Big(\frac{s}{\tau} \Big)^{d-1} f(s) ds\bigg) d\tau = \int_0^1 \frac1{d-2} \bigg(\frac1{\max (s,r)^{d-2}} - 1 \bigg) f(s) s^{d-1} ds,
\]

\noi
where we used the boundary condition $u(1) = 0$.
Then, \eqref{Green_r} follows directly.
\end{proof}

\subsubsection{Green functions on shells}
Let $u (x)$ be the solution to the following equation posed on the shell $B_1\backslash B_\dl$ with Dirichlet boundary conditions,
\begin{equation}
\begin{cases}
- \Delta u (x) = f (x),  & \dl < |x|< 1 ,\\
u (x) = 0, & x \in \{|x| = \dl\} \cup \{|x| = 1\}.
\end{cases}
\label{BVP_dl}
\end{equation}

\noi
Denote by $G_{\dl} (x,y)$ the Green function associated with Poisson's equation  \eqref{BVP_dl}.
Then, the solution to \eqref{BVP_dl} can be written as
\begin{align}
\label{Green_dl}
u(x) = \int_{B_1 \backslash B_\dl} G_\dl (x,y) f(y) dy.
\end{align}

\noi
If $u$ is a radial solution to \eqref{BVP_dl},
then we can write \eqref{BVP_dl} as
\begin{align}
\label{BVP1_dl}
\begin{cases}
- \partial_r^2 u (r) - \frac{d-1}r \partial_r u (r) = f(r), \quad  r \in (\dl,1), \\
u(1) = 0,\\
u(\dl) = 0.
\end{cases}
\end{align}

\noi
Denote by $K_\dl (r,s)$ the radial Green function to the Laplacian operator on $B_1 \backslash B_\dl$.
Then, the solution to \eqref{BVP1_dl} can be recast as
\begin{align}
\label{Green1_dl}
u(r) = \int_\dl^1 K_\dl (r,s) f(s) s^{d-1} ds.
\end{align}

\noi
Therefore, it follows from \eqref{Green1_dl} that
\begin{align}
    \label{r_green3}
    \begin{split}
    - \bigg( \partial_r^2 + \frac{d-1}r \partial_r\bigg) K_\dl (r,s)  = \boldsymbol \dl(r-s) r^{-d+1},
    \end{split}
\end{align}

\noi
with boundary conditions $K_\dl(1,s) = K_\dl(\dl,s) = 0$,
where $\boldsymbol \dl$ is the delta function.
Then, we have the following formula for $K_\dl (r,s)$.
\begin{lemma}
\label{LEM:GFFX2}
Let $d \ge 3$ and $K_\dl (r,s)$ be as in \eqref{Green1_dl} and \eqref{r_green3}. Then we have
\begin{align}
    \label{kernel0}
    K_\dl (r,s) = \frac{\big( \min (r,s)^{d-2} - \dl^{d-2} \big)\big( 1 - \max (r,s)^{d-2}\big)}{(d-2) \, r^{d-2} s^{d-2} \big(1 - \dl^{d-2}\big)} .
\end{align}

\noi
In particular, $K_\dl$ is symmetric and non-negative on $(\dl,1)^2$.
\end{lemma}

\begin{proof}
By integrating \eqref{r_green3}, we get
\begin{align}\label{kernel2}
- r^{d-1} \partial_r K_\dl(r,s) =  \ind_{\ge 0} (r - s) +C_1,
\end{align}
where $C_1 $ is a function depending only on $s$ and $\dl$, and
\[
\ind_{\ge 0} (r-s) =
\begin{cases}
0, & r < s,\\
1, & r \ge s.
\end{cases}
\]
By rewriting \eqref{kernel2} as
\[
- \partial_r K_\dl(r,s) = \frac{\ind_{\ge 0} (r - s)}{r^{d-1}} + \frac{C_1 }{r^{d-1}}
\]
and integrating in $r$ from $\dl$, using the boundary condition $K_\dl (\dl,s) = 0$, we get
\begin{align}
\label{kernel3}
K_\dl (r,s) = - \frac1{d-2}\bigg( \frac1{s^{d-2}} - \frac1{r^{d-2}} \bigg)_+ + \frac{C_1}{d-2} \bigg( \frac1{ r^{d-2}} - \frac1{\dl^{d-2}}\bigg),
\end{align}

\noi
where $(a)_+ = 0$ for $a < 0$ and $(a)_+ = a$ for $a \ge 0$.
Note that the positive part is supported on $\{r > s\}$, since $s \mapsto s^{-d+2}$ is decreasing.
Imposing the remaining boundary condition $K_\dl (1, s) = 0$ for $s \in (\dl, 1)$ gives
\[
- \frac1{d-2} \bigg( \frac1{s^{d-2}} - 1 \bigg) + \frac{C_1}{d-2} \bigg( 1 - \frac1{\dl^{d-2}}\bigg) = 0,
\]

\noi
that is, $C_1 = - \frac{\dl^{d-2} (1 - s^{d-2})}{s^{d-2}(1-\dl^{d-2})}$. Substituting this into \eqref{kernel3} and
simplifying separately on $\{r \le s\}$ and $\{r > s\}$ leads to \eqref{kernel0}.
\end{proof}

\subsection{Gaussian free field estimates}

In this subsection, we derive some estimates for the radial Gaussian free fields $X$ and $X^+_\dl$, as given in \eqref{gaussian} and  \eqref{mu+}, respectively.
In particular, given  a compact domain $D$, we denote by
\begin{align}
\label{XD}
X^D = \sum_{n=0}^\infty \frac{g_n}{\ld_n} e_n,
\end{align}

\noi
where $\{e_n\}$ and  $\{\ld_n^2\}$ are eigenfunctions and eigenvalues, respectively, of $-\Delta$ over the domain $D$.
See Subsection \ref{SUB:GFFg} for more details.

\subsubsection{Radial Gaussian free field on ball}
This subsection is devoted to estimates of the radial Gaussian free field $X$ on the ball  $B_1$.
Let $\ph$ be a radial test function supported on the unit ball  $B_1$. From the definition of $X$ and \eqref{test},
it follows that
\begin{align}
\label{GFFX}
\begin{split}
\E \bigg[ \Big( \int_{B_1} X (x) \ph(x) dx\Big)^2 \bigg ]
= \iint_{(0,1]^2} K (r,s) \ph(r)\ph(s) (rs)^{d-1}drds ,
\end{split}
\end{align}

\noi
 where $K$ is the radial Green function given in \eqref{Green_r}.
Here and in what follows, we only consider radial functions $\ph$.
We do not distinguish  $\ph (r)$ from  $\ph (x)$ when $r = |x|$.
As a consequence of \eqref{Green_r}, we have
\begin{lemma}
\label{LEM:GFFX}
Let $X$ be the radial Gaussian free field defined in \eqref{gaussian}. Then we have the following pointwise estimate
\begin{align}
    \label{GFFX_e}
    \E \big[ |X(r) |^2\big] = K (r,r) \les \frac1{r^{d-2}},
\end{align}

\noi
provided $r \in (0,1)$, where $K (\cdot, \cdot)$ is the radial Green function given in \eqref{Green_k}.
\end{lemma}

\begin{proof}
From \eqref{Green_r} and \eqref{Green_k}, we see that
\begin{align}
\begin{split}
\E \big[ |X(r) |^2\big] & = \E \bigg[ \Big| \sum_{n\ge 0} \frac{e_{n,d} (r) g_n}{z_{n,d}}\Big|^2\bigg]  = \sum_{n \ge 0} \frac{e_{n,d}^2 (r)}{z_{n,d}^2} = K(r,r),
\end{split}
\end{align}

\noi
which finishes the proof in view of Lemma \ref{LEM:Green_r}.
\end{proof}

As a consequence of Lemma \ref{LEM:GFFX}, we can make sense of the measure \eqref{rho_eps}.

\begin{proposition}
\label{LEM:Ga}
Given  $0 < \eps <1$, the weighted Gaussian measure $\rho_\eps$ given in \eqref{rho_eps} is well defined.
\end{proposition}

\begin{proof}
Since $Z_\eps \le 1$, it suffices to show that the partition function $Z_\eps$ in \eqref{partition_eps} is positive.
By Jensen's inequality, Fubini's theorem, the Gaussianity of $X(r)$, and \eqref{GFFX_e}, we have
\begin{align}
\begin{split}
- \log Z_\eps & = - \log \E \bigg[ \exp \bigg( - \frac 14 \int_{B_1\backslash B_\eps} |X|^4 dx \bigg) \bigg]
\le \frac14 \E \bigg[ \int_{B_1\backslash B_\eps} |X|^4 dx \bigg] \\
& = \frac34 \int_\eps^1 \big( \E [|X (r)|^2] \big)^2 r^{d-1} dr
\les \int_\eps^1 \frac1{r^{d-3}} dr < \infty.
\end{split}
\label{Ga}
\end{align}

\noi
Hence, $Z_\eps > 0$, which finishes the proof.
\end{proof}

\subsubsection{Radial Gaussian free field on shell}

In this subsection, we consider the radial Gaussian free field $X^+_\dl$ on the shell $\dl < |x| < 1$, whose law is given by the Gaussian measure
\begin{align}\label{mu+}
d \mu^+_\dl (u) = \frac 1 {Z_+} \exp\Big( - \frac 12 \int_{B_1\backslash B_\dl} |\nabla u|^2 dx\Big) du,
\end{align}
supported on distributions with $0$ boundary condition on $\partial (B_1 \backslash B_\dl) = \{|x| = \dl\} \cup \{|x| = 1\}$, and extended to $0$ on $B_\dl = \{|x| < \dl\}$.

As a consequence of Lemma \ref{LEM:GFFX2}, we have the following estimate.

\begin{corollary}
\label{COR:Green}
Let $\dl \in (0,1)$ and let $K_\dl (r,s)$ be as in \eqref{Green1_dl}. Then we have
\begin{align}
    \label{KE}
    0 \le K_{\dl}(r,s) \le \frac{\min (r^{d-2}, s^{d-2}) - \dl^{d-2}}{(d-2)\, r^{d-2} s^{d-2}},
\end{align}

\noi
for $\dl < r, s < 1$.
\end{corollary}

\begin{proof}
Since $\max (r,s) > \dl$, we have $0 < 1 - \max (r,s)^{d-2} \le 1 - \dl^{d-2}$, and \eqref{KE} follows from \eqref{kernel0}.
\end{proof}

Corollary \ref{COR:Green} allows us to deduce the following estimate on the radial Gaussian free field $X^+_\dl$.

\begin{lemma} \label{LEM:X+}
Let $\dl \in (0,1)$ and $r, s \in (\dl, 1)$.
Then, we have
\begin{align} \label{covariance}
\E\big[ X^+_\dl(r)  X^+_\dl(s)\big ] = K_\dl (r,s),
\end{align}
where $K_\dl(r,s)$ is the Green function given in \eqref{kernel0}.
Furthermore, it follows that
\begin{equation} \label{X+estimate}
\E \big[ \big|X^+_\dl(r) \big|^2 \big] = K_\dl(r,r) \le \frac{r-\dl}{r^{d-1}} \le \frac{1}{r^{d-2}}.
\end{equation}
\end{lemma}

\begin{proof}
The proof is similar to that of Lemma \ref{LEM:GFFX} by observing
\begin{align}
\begin{split}
\E \big[  X^+_\dl(r)  X^+_\dl(s) \big] & = \E \bigg[ \bigg( \sum_{n\ge 0} \frac{e_n^\dl (r) g_n}{\ld_n} \bigg) \bigg( \sum_{n\ge 0} \frac{e_n^\dl (s) g_n}{\ld_n} \bigg) \bigg] \\
& = \sum_{n \ge 0} \frac{e_n^\dl (r) e_n^\dl (s)}{\ld_n^2} = K_\dl(r,s) ,
\end{split}
\label{variance}
\end{align}

\noi
where $\ld_n^2$ and $e_n^\dl$ are the eigenvalues and the radial eigenfunctions of $-\Delta$ on $B_1 \backslash B_\dl$ with Dirichlet boundary conditions, normalised in $L^2_{\rm rad} (B_1 \backslash B_\dl) = L^2 ((\dl,1), r^{d-1} dr)$.
Then \eqref{X+estimate} follows from \eqref{KE} and the bound $r^{d-2} - \dl^{d-2} \le (d-2) r^{d-3} (r - \dl)$.
\end{proof}

\begin{corollary}\label{COR:L4est} Let $d \ge 4$ and $\dl \in (0,1)$. We have that
\begin{equation}
\label{eq:X+L4}
\E \bigg[ \int_{B_1} |X^+_\dl|^4 dx \bigg] \les \Ld (\dl), \end{equation}
where the implicit constant is independent of $\dl$, and $\Ld (\cdot)$ is given by
\begin{align}
\label{Ld}
\Ld (\dl) =
\begin{cases}
|\ln \dl|, & d = 4; \\
\dl^{-d + 4}, & d \ge 5,
\end{cases}
\end{align}
for $\dl > 0$.
\end{corollary}
\begin{proof}
Since $X^+_\dl$ is Gaussian and $X^+_\dl|_{B_\dl} \equiv 0$, we have that
\begin{align*}
\E \bigg [ \int_{B_1} |X^+_\dl(x)|^4 dx \bigg]
& = \E \bigg[ \int_{B_1\backslash B_\dl} |X^+_\dl(x)|^4 dx \bigg] = \int_{B_1\backslash B_\dl} \E \big [ |X^+_\dl(x)|^4 \big] dx \\
& = 3 \int_{B_1\backslash B_\dl} \big( \E [|X^+_\dl (x)|^2]\big)^2 dx, \\
\intertext{by using the spherical coordinates and \eqref{X+estimate}, we continue with}
&= c \int_\dl^1 \big( \E [ |X^+_\dl (r)|^2] \big)^2 r^{d-1} dr  = c \int_\dl^1 \big( K_\dl(r,r) \big)^2 r^{d-1} dr\\
& \les \int_{\dl}^1 r^{-d +3} dr \le \Ld (\dl),
\end{align*}

\noi
which finishes the proof.
\end{proof}

\begin{lemma}\label{LEM:reg}
    Let  $\dl \in (0,1)$.
    For every $\ta >0$ and every $p>0$, there exists a constant $C_\dl = C_\dl(\ta,p)$ such that
    \begin{align}
    \label{reg_e}
    \E\Big [ \| X^+_\dl\|^p_{C^{\frac12-\ta}(B_1\backslash B_{\dl})} \Big] \le C_\dl.
    \end{align}
\end{lemma}
\begin{proof}
    By Kolmogorov's continuity theorem, since $X^+_\dl$ is Gaussian, it is enough to check that for every $\dl \le s \le t \le 1 $,
    \begin{equation}
    \label{regularityX+expectation}
    \E \big [ |X^+_\dl(t)-X^+_\dl(s)|^2 \big ] \les_\dl  |t-s|,
    \end{equation}
    where the implicit constant is allowed to depend on the parameter $\dl$.
    From \eqref{covariance}, we have that
    \begin{align*}
        \E\big [|X^+_\dl(t)-X^+_\dl(s)|^2 \big ] = K_\dl(t,t) -2K_{\dl}(t,s) + K_\dl(s,s).
    \end{align*}
    Recalling the formula \eqref{kernel0}, it is immediate that $K_\dl (t,s)$ is Lipschitz on  $[\dl,1]^2$, since $t,s \in (\dl, 1)$ for $\dl > 0$. Therefore, \eqref{regularityX+expectation} follows, and thus the lemma.
\end{proof}

\subsection{Comparison principle}
This section is devoted to the linear operator $(-\Delta + V)^{-1}$ appearing in \eqref{Upsilon}.
We shall derive a comparison principle that will be used in the proof of Lemma \ref{LEM:Upsilon}.
Before stating the comparison principle used in this paper, we recall the Lax--Milgram theorem from \cite{Br11},
which is the key ingredient in proving the existence and uniqueness of the solution to the weak formulation of the Dirichlet problem.
Let $D$ be a regular bounded domain in $\R^d$ with boundary $\partial D$, and consider the Dirichlet problem
\begin{align}
\label{Poisson}
\begin{cases}
-\Delta u + V u= f & \textup{ in } D;\\
u = 0& \textup{ on } \partial D,
\end{cases}
\end{align}

\noi
where $V \in L^\infty (D)$ with $V \ge 0$ on $D$, and $f \in L^2 (D)$.
Let $H_0^1 (D)$ be the completion of the space of all smooth functions with compact support in $D$ with respect to the norm $\| \cdot \|_{H^1 (D)}$.
The  Dirichlet problem \eqref{Poisson} can be written in weak formulation as:
to find $u \in H_0^1 (D)$ such that
\begin{align}
\label{Poisson_w}
\int_D \big( \nabla u \cdot \nabla v + V u v \big) dx = \int_D f v dx,
\end{align}

\noi
for all $v \in H_0^1 (D)$.
By defining the bilinear form on $H_0^1 (D)$ via
\begin{align}
\label{LM_b}
a(u,v) = \int_D \nabla u \cdot \nabla v dx  + \int_D V uv dx,
\end{align}

\noi
we conclude from the  Lax--Milgram theorem, \cite[Section 5.3]{Br11}, that there exists a unique solution $u \in H_0^1(D)$ to \eqref{Poisson_w}. Moreover, we have
\begin{align}
\label{LM_b1}
\|u\|_{H^1 (D)} \les \|f\|_{H^{-1} (D)}.
\end{align}

\noi
If, moreover, $\partial D$ is of class $C^2$, then $u \in H^2 (D)$.
For further discussion on the Lax--Milgram theorem, see  \cite[Chapter 9]{Br11}.
If in addition, for some $\al \in (0,1)$, $\partial D$ is of class $C^{2,\al}$ and $V, f \in C^{0,\al}(\cj{D})$, then from Schauder estimates \cite[Chapter 6]{GT01} (see also  \cite[Chapter 2]{FR22} and \cite[Chapter 5]{GM}) we have $u \in C^{2,\al} (\cj{D})$, with
\begin{align}
\label{Schauder}
\|u\|_{C^{2,\al}(\cj D)} \les \|u\|_{L^\infty (D)} + \|f\|_{C^{0,\al}(\cj D)},
\end{align}

\noi
where the implicit constant depends only on $D$, $\al$, and an upper bound for $\|V\|_{C^{0,\al}(\cj D)}$.
It is easy to check that all the assumptions on $\partial D$ are satisfied by the balls and shells considered in this paper.

The above discussion, in particular, implies that the operator
\begin{align}
\label{inver}
(-\Delta + V)^{-1}: L^2 (D) \to H^1_0 (D)
\end{align}

\noi
is well-defined provided $V \in L^\infty (D)$ and $V \ge 0$.
We recall a well-known comparison result for the heat kernels of Schr\"odinger operators from \cite[Sec. {\rm II}. 6]{Simon05}; also see \cite[Lemma 4.3]{LRN18}.

\begin{lemma}[Heat kernel estimates]
\label{LEM:HKE}
Consider two Schr\"odinger operators $- \Delta + V_j$ on $L^2 (D)$ with Dirichlet boundary conditions, where $V_j \in L^\infty (D)$, $j = 1,2$, and $V_1 (x) \ge V_2 (x)$.
Let $K_j^\tau $ be the integral kernel of $e^{ - \tau (- \Delta + V_j)}$ in the sense
\[
  e^{ - \tau (- \Delta + V_j)} f  (x) = \int_D K_j^\tau  (x,y) f(y) dy
\]

\noi
for all $f \in L^2 (D)$ and $j = 1,2$.
Then for all $\tau > 0$, we have the kernel estimate
\begin{align}
\label{HKE}
0 \le K^\tau_1 (x,y) \le K^\tau_2 (x,y)
\end{align}

\noi
for almost every $(x,y) \in D \times D$.
\end{lemma}

The proof of Lemma \ref{LEM:HKE} follows from the Trotter product formula (see e.g. \cite[Theorem 1.1]{Simon05}),
\[
\exp \big( -\tau ( - \Delta + V_j) \big) = \lim_{m \to \infty} \bigg( \exp \bigg( \frac{\tau \Delta}{m}\bigg) \cdot \exp \bigg( -\frac{\tau V_j}{m} \bigg)  \bigg)^m,
\]

\noi
together with the positivity of the heat kernel $ \exp ( \frac{\tau \Delta}{m} )  $ and the pointwise inequality
\[
0 \le \exp \bigg( -\frac{\tau V_1}{m} \bigg) \le \exp \bigg( -\frac{\tau V_2}{m} \bigg).
\]

\noi
For the detailed proof,
we refer readers to \cite[Sec. {\rm II}. 6]{Simon05} or \cite[Lemma 4.3]{LRN18}.

Now, we are ready to state the comparison principle.

\begin{lemma}
[Comparison principle]
\label{LEM:CP}
Let $V_1, V_2 \in L^\infty (D)$ with $V_1(x) \ge V_2(x) \ge 0$ for all $x \in D$.
Given $f \in C(\cj D)$,
we have
\begin{align}
\label{CP}
\big| (-\Delta + V_1 (x))^{-1} f (x) \big| \le
 (-\Delta + V_1 (x))^{-1} |f| (x) \le
(-\Delta + V_2 (x))^{-1} |f| (x),
\end{align}

\noi
where the inverse $(-\Delta + V_1 (x))^{-1}$ is given in \eqref{inver}.
\end{lemma}

\begin{proof}
Recall the integral representation formula
\[
(-\Delta + V_j )^{-1} = \int_0^\infty \exp \big( -\tau (-\Delta + V_j )\big) d \tau ,
\]

\noi
for $j = 1,2$.
Then by invoking Lemma \ref{LEM:HKE}, we obtain that for all $x \in D$,
\begin{align}
\label{CP-positive}
\begin{split}
\big|(-\Delta + V_1 )^{-1} f (x)\big| & \le \int_0^\infty \exp \big( -\tau (-\Delta + V_1 )\big) |f|(x) d \tau \\
& = \int_0^\infty \int_D K_1^\tau (x,y) |f|(y) dy d \tau\\
& \le \int_0^\infty \int_D K_2^\tau (x,y) |f|(y) dy d \tau\\
& = (-\Delta + V_2 )^{-1} |f| (x),
\end{split}
\end{align}

\noi
where we used the kernel comparison \eqref{HKE}.
Thus, we finish the proof.
\end{proof}

With the above preparations, we are ready to solve the following  integro-differential equation
\begin{align}
\label{IDe}
-\Delta u (t,x) + V(t,x) u(t,x) = a(x) \int_0^t u(s,x) ds + b(t,x),
\end{align}

\noi
where $V, b \in C([0,1]; C^{0,\al}(\cj D))$,  $a \in C^{0,\al}(\cj D)$, and $V \ge 0$ on $D$.
Note that \eqref{IDe} is a general form of \eqref{Upsilon}.
The following lemma confirms that \eqref{IDe} is well-posed.

\begin{lemma}
\label{LEM:IDe}
Let $\al \in (0,1)$, and assume that $\partial D$ is of class $C^{2,\al}$.
Let $V, b \in C ([0,1]; C^{0,\al}(\cj D))$ and $a \in C^{0,\al}(\cj D)$, with $V \ge 0$ on $[0,1] \times D$.
Then, there exists a unique solution to \eqref{IDe} in $C([0,1]; H^1_0 (D) \cap C^{2,\al}(\cj D) )$.
\end{lemma}

\begin{proof}

From the Lax--Milgram theory and the Schauder estimates above, we note that for fixed $t \in [0,1]$, the operator
\[
(-\Delta + V(t,x))^{-1} : C^{0,\al} (\cj D) \to H^1_0 (D) \cap C^{2,\al} (\cj D)
\]

\noi
is well-defined. Moreover, there exists a constant $C_1 > 0$, depending only on $D$, $\al$, and $\sup_{t \in [0,1]} \|V(t)\|_{C^{0,\al}(\cj D)}$, such that
\begin{align}
\label{IDe_S}
\|(-\Delta + V(t,x))^{-1} g\|_{C^{2,\al}(\cj D)} \le C_1 \|g\|_{C^{0,\al}(\cj D)}
\end{align}

\noi
for every $t \in [0,1]$ and $g \in C^{0,\al}(\cj D)$.
Indeed, by Lemma \ref{LEM:CP} with $V_1 = V(t)$ and $V_2 = 0$, we have $|(-\Delta + V(t,x))^{-1} g| \le (-\Delta)^{-1} |g| \le \|g\|_{L^\infty (D)} (-\Delta)^{-1} 1$, where $(-\Delta)^{-1} 1 \in C^{2,\al}(\cj D)$ is bounded, so that \eqref{IDe_S} follows from \eqref{Schauder}.
We define a map
\begin{equation} \label{Gamma_fixed_point}
\G (u) (t) = (-\Delta + V(t,x))^{-1} \bigg( a(x) \int_0^t u(s,x) ds + b(t,x)\bigg),
\end{equation}

\noi
whose fixed points are exactly the solutions to \eqref{IDe}.
It suffices to show the map $\G$ admits a fixed point in $C([0,1]; H^1_0 (D) \cap C^{2,\al}(\cj D) )$. We first show that the map $\G$ maps $C([0,1]; H^1_0 (D) \cap C^{2,\al}(\cj D) )$ to itself.
We note
\[
\begin{split}
\G(u) (t_1) & - \G(u) (t_2) =   (-\Delta + V(t_1,x))^{-1} \bigg( a(x) \int_{t_2}^{t_1} u(s,x) ds + b(t_1,x) - b(t_2,x)\bigg) \\
&   +  \Big[ (-\Delta + V(t_1,x))^{-1} - (-\Delta + V(t_2,x))^{-1} \Big] \bigg( a(x) \int_0^{t_2} u(s,x) ds + b(t_2,x)\bigg),
\end{split}
\]

\noi
whose first term tends to $0$ in $C^{2,\al}(\cj D)$ as $t_1 \to t_2$, by \eqref{IDe_S} and the continuity of $b$.
Then, we only need to show that the operator $(-\Delta + V(t,x))^{-1}$, from $C^{0,\al}(\cj D)$ to $C^{2,\al}(\cj D)$, is continuous in $t \in [0,1]$.
For $t_1, t_2 \in [0,1]$, we have the resolvent identity
\[
(-\Delta  + V(t_1,x))^{-1}  - (-\Delta + V(t_2,x))^{-1} = (-\Delta  + V(t_1,x))^{-1} \big( V(t_2) - V(t_1) \big) (-\Delta + V(t_2,x))^{-1}.
\]

Therefore, by \eqref{IDe_S} and the bound $\|V \psi\|_{C^{0,\al}(\cj D)} \le \|V\|_{C^{0,\al}(\cj D)} \|\psi\|_{C^{2,\al}(\cj D)}$, the norm of $(-\Delta + V(t_1,x))^{-1} - (-\Delta + V(t_2,x))^{-1}$, as an operator from $C^{0,\al}(\cj D)$ to $C^{2,\al}(\cj D)$, is $\les \|V(t_1) - V(t_2)\|_{C^{0,\al}(\cj D)}$, which tends to $0$ as $t_1 \to t_2$ by the continuity of $V$.

We then prove that $\G$ is a contraction. From \eqref{IDe_S}, we note
\begin{align}
\label{IDe_C}
\begin{split}
\| \G (u_1)(t) - \G (u_2)(t)\|_{C^{2,\al} (\cj D)} & = \bigg\| (-\Delta + V(t,x))^{-1} \bigg( a(x) \int_0^t \big( u_1(s,x) - u_2(s,x) \big) ds \bigg) \bigg\|_{C^{2,\al} (\cj D)} \\
& \le C_1 \|a\|_{C^{0,\al}(\cj D)} \int_0^t \| u_1(s) - u_2(s)\|_{C^{2,\al} (\cj D)} ds,
\end{split}
\end{align}

\noi
which implies that $\G$ is a contraction on  $C([0,t_0]; H^1_0 (D) \cap C^{2,\al} (\cj D) )$ with  $t_0 = \big( 2 C_1 \|a\|_{C^{0,\al}(\cj D)} + 1 \big)^{-1}$.

To extend the solution $u$ to \eqref{IDe} beyond $t_0$, we rewrite \eqref{IDe} as
\begin{align}
\label{IDe2}
-\Delta u (t,x) + V(t,x) u(t,x) = a(x) \int_{t_0}^t u(s,x) ds + b_1(t,x),
\end{align}

\noi
where $t\ge t_0$ and
\[
b_1 (t,x) = b(t,x) + a(x) \int_0^{t_0} u(s,x) ds.
\]

\noi
Since $b_1 \in C([t_0,1]; C^{0,\al}(\cj D))$ and $t_0$ depends only on $C_1$ and $\|a\|_{C^{0,\al}(\cj D)}$, by the same argument as above, we may solve \eqref{IDe2} on the time interval $[t_0,2t_0]$. By repeating this process, we can solve \eqref{IDe} in $C([0,1]; H^1_0 (D) \cap C^{2,\al} (\cj D) )$, and the solution is unique, since it is unique on each of these intervals.
We also note that, by iterating \eqref{IDe_C}, we have
\[
\sup_{0 \le t \le 1} \| \G^n (u_1)(t) - \G^n (u_2)(t)\|_{C^{2,\al} (\cj D)} \le \frac{\big( C_1 \|a\|_{C^{0,\al}(\cj D)} \big)^n}{n!} \sup_{0 \le t \le 1} \| u_1(t) - u_2(t)\|_{C^{2,\al} (\cj D)}
\]

\noi
for every $n \ge 1$, so that $\G^n$ is a contraction on $C([0,1]; H^1_0 (D) \cap C^{2,\al}(\cj D))$ for $n$ large enough.
\end{proof}

\subsection{The Bou\'e--Dupuis formula}
In this section,
we recall the  Bou\'e--Dupuis formula.
Let  $W_D(t)$ be a cylindrical Brownian motion in $L^2_{0,\rm rad} (D)$ with zero boundary condition defined by
\begin{align}
\label{BM}
W_D (t) = \sum_{n \ge 0} B_n(t) e_n^D,
\end{align}
\noi
where $D \subset \R^d$ is  a radially symmetric domain, $e_n^D$ are the radial eigenfunctions of $-\Delta$ on $D$ given in \eqref{eigenD}, and $\{B_n(t)\}_{n \ge 0}$ is a sequence of mutually independent Brownian motions. We  define the centered cylindrical Brownian process $X^D(t)$ by
\begin{align}
    \label{Yt}
    X^D (t) = (-\Delta)^{-\frac12} W_D (t) = \sum_{n \ge 0} \frac{B_n(t)}{\ld_n^D} e_n^D,
\end{align}

\noi
which is well-defined in $H^{\frac12 - \eps} (D)$.
We may suppress the superscript $D$ and write $X^D (t)$ as $X (t)$, $\ld_n^D$ as $\ld_n$, and $e_n^D$ as $e_n$ when the dependence on $D$ is clear from the context.

Let $\big(\F^{X^D}_t\big)_{t\in[0,1]}$ be the $\P$-augmented filtration generated by $X^D$ (equivalently, by $\{B_n\}_{n \ge 0}$), and let $\mathbb H$ be the space of drifts, which consists of $\big(\F^{X^D}_t\big)$-progressively measurable processes $V$ such that $V(0)=0$ and $\dot V \in L^2_t ([0,1]; H^1_{0,{\rm rad}} (D))$, $\P$-almost surely.
Here $H^s_{0,{\rm rad}} (D)$ denotes the completion of the space of smooth radial functions with compact support in $D$ with respect to the $H^s$ norm.
When $s=0$, we write $L^2_{0,\rm{rad}} = H^0_{0,\rm{rad}}$.
We now recall the Bou\'e--Dupuis formula from~\cite{BD, Ust}, specialising it to our setting. The following is a direct consequence of \cite[Theorem 3.2]{Zhang09}. See also \cite[Proposition {A.2}]{TolWeb} and \cite[Theorem 1.1]{HW22}.

\begin{proposition}[Bou\'e--Dupuis variational formula]\label{LEM:BD}
 Suppose that $F:L^2_{0, {\rm rad}} (D) \to \R$
	is a bounded measurable function.
	Then, we have
	\begin{align}
	- \log \E\Big[e^{-F( X^D(1))}\Big]
	= \inf_{V \in \mathbb H}
	\E\bigg[ F(X^D(1) + V(1)) + \frac{1}{2} \int_0^1 \|\dot V \|_{\dot H^1}^2 dt \bigg],
	\label{BDformula}
	\end{align}
\noi
and the expectation $\E = \E_\P$
is with respect to the underlying probability measure~$\P$.
\end{proposition}

The condition that $F$ is bounded in Proposition \ref{LEM:BD} is not necessarily satisfied in our later application.
It is known that this boundedness condition can be relaxed in the finite-dimensional setting, where $X^D$ is a  finite-dimensional process. In particular, in \cite[Theorem 1.1]{HW22} the authors showed that the boundedness condition on $F$ can be relaxed to the following ones,
\[
\E \Big[ e^{-F(X^D(1))} \Big] < \infty \textup{ and } \E \big[ \max (F(X^D(1)), 0) \big] < \infty,
\]

\noi
provided that $X^D$ is a finite-dimensional process.
But for the infinite-dimensional setting as in Proposition \ref{LEM:BD}, the situation is more involved, and we need  several lemmas.

\begin{lemma}\label{lem:weak_drift}
Let a \emph{weak drift} be a tuple $V^{\rm w} = (\wt \O, \wt \F, (\wt \G_t)_{t \in [0,1]}, \wt \P, \wt X^D, \wt V)$ consisting of a filtered probability space, a process $\wt X^D$ with the same law as $X^D$ in \eqref{Yt} which is a cylindrical Brownian motion of the form \eqref{Yt} with respect to $(\wt \G_t)$, and a $(\wt \G_t)$-progressively measurable process $\wt V$ with $\wt V(0) = 0$ and $\dot{\wt V} \in L^2_t([0,1]; H^1_{0,{\rm rad}} (D))$, $\wt \P$-almost surely.  Let $F : L^2_{0, {\rm rad}} (D) \to (-\infty, +\infty]$ be Borel measurable and bounded below. We write $\mathbb H^{\rm w}$ for the collection of all weak drifts.  Then, with the convention $-\log 0 = +\infty$,
\begin{equation}
\begin{aligned}
\label{HHw}
- \log \E \Big[ e^{-F(X^D(1))} \Big] &\le \inf_{V^{\rm w} \in \mathbb H^{\rm w}} \wt \E\bigg[ F (\wt X^D(1) + \wt V(1)) + \frac12 \int_0^1 \|\dot{\wt V}\|_{\dot H^1}^2 dt\bigg]\\
&\le \inf_{V \in \mathbb H} \E\bigg[ F (X^D(1) + V(1)) + \frac12 \int_0^1 \|\dot V\|_{\dot H^1}^2 dt\bigg] .
\end{aligned}
\end{equation}
\end{lemma}
\begin{proof}
We first show that
\begin{equation} \label{HHwle}
 \inf_{V^{\rm w} \in \mathbb H^{\rm w}} \wt \E\bigg[ F (\wt X^D(1) + \wt V(1)) + \frac12 \int_0^1 \|\dot{\wt V}\|_{\dot H^1}^2 dt\bigg]
\le \inf_{V \in \mathbb H} \E\bigg[ F (X^D(1) + V(1)) + \frac12 \int_0^1 \|\dot V\|_{\dot H^1}^2 dt\bigg].
\end{equation}

\noi
Let $V \in \mathbb H$.
Then $X^D$ is a cylindrical Brownian motion with respect to the augmented filtration $\big(\F^{X^D}_t\big)$, and $V$ is $\big(\F^{X^D}_t\big)$-progressively measurable by the definition of $\mathbb H$, so that
$$ V^{\rm w} = \big(\O, \F, \big(\F^{X^D}_t\big), \P, X^D, V\big)$$
is a weak drift, for which the two functionals in \eqref{HHwle} coincide.
Taking the infimum over $V \in \mathbb H$ gives \eqref{HHwle}.

We now move to showing
\begin{equation} \label{HHwge}
 \inf_{V^{\rm w} \in \mathbb H^{\rm w}} \wt \E\bigg[ F (\wt X^D(1) + \wt V(1)) + \frac12 \int_0^1 \|\dot{\wt V}\|_{\dot H^1}^2 dt\bigg]
\ge - \log \E \Big[ e^{-F(X^D(1))} \Big].
\end{equation}
Let $V^{\rm w}= (\wt \O, \wt \F, (\wt \G_t), \wt \P, \wt X^D, \wt V) \in \mathbb H^{\rm w}$.
It suffices to show that
\begin{equation} \label{HHwge1}
- \log \E \Big[ e^{-F(X^D(1))} \Big] \le \wt \E\bigg[ F (\wt X^D(1) + \wt V(1)) + \frac12 \int_0^1 \|\dot{\wt V}\|_{\dot H^1}^2 dt\bigg].
\end{equation}
We may assume that $\wt \E \big[ \int_0^1 \|\dot{\wt V}\|_{\dot H^1}^2 dt\big] < \infty$, since otherwise, $F$ being bounded below, the right-hand side of \eqref{HHwge1} is $+\infty$ and there is nothing to prove.
We first prove \eqref{HHwge1} under the additional assumption that $F$ is bounded from above as well.

Let $\chi: \R \to [0,1]$ be a smooth, compactly supported cutoff function with $\chi|_{[-1,1]} \equiv 1$, let $M \ge 1$ be such that $\chi$ is supported in $[-M,M]$, write
$$ E(t) := \int_0^t \| \dot {\wt V}(s)\|_{\dot H^1}^2 ds , $$
and define
$$ {\wt V}_n(t) = \int_0^t \chi\Big(\frac{E(s)}{n}\Big) \dot {\wt V}(s) \, ds. $$
Then $V^{\rm w}_n =  (\wt \O, \wt \F, (\wt \G_t), \wt \P, \wt X^D, {\wt V}_n)$ is also a weak drift, and $\dot {\wt V}_n (s) = \chi \big( E(s)/n\big) \dot{\wt V}(s)$.
Since $E$ is non-decreasing and $\chi$ vanishes outside $[-M,M]$, we have $\dot {\wt V}_n(s) = 0$ whenever $E(s) > Mn$, and therefore
\begin{align}
\label{Novikov}
\int_0^1 \| \dot {\wt V}_n(s)\|_{\dot H^1}^2 ds \le \int_0^1 \ind_{E(s) \le Mn} \| \dot {\wt V}(s)\|_{\dot H^1}^2 ds \le Mn
\end{align}

\noi
almost surely.
In particular, Novikov's condition is satisfied by $\dot{\wt V}_n$, so that
$$\Q_n = \exp\Big(- \int_0^1 \jb{ \dot {\wt V}_n(s) , d  \wt X^D(s)}_{\dot H^1} - \frac12 \int_0^1 \|\dot{{\wt V}_n}(s)\|_{\dot H^1}^2 ds\Big) \wt \P$$
is a probability measure, and by Girsanov's theorem $\wt X^D + {\wt V}_n$ is a cylindrical Brownian motion with the same law as $X^D$ under $\Q_n$.
Moreover, since $0 \le \chi \le 1$ and $E(1) < \infty$ almost surely, we have $\chi (E(s)/n) \to 1$ for every $s$, so that, by dominated convergence,
\begin{equation}\label{Vntrunc1}
\begin{gathered}
\int_0^1 \| \dot {\wt V}_n(s)\|_{\dot H^1}^2 ds \le \int_0^1 \| \dot {\wt V}(s)\|_{\dot H^1}^2 ds,\\
\wt \E \bigg[\int_0^1 \|\dot{{\wt V}_n}\|_{\dot H^1}^2 ds\bigg] \longrightarrow \wt \E \bigg[\int_0^1 \|\dot{{\wt V}}\|_{\dot H^1}^2 ds\bigg] .
\end{gathered}
\end{equation}

\noi
Finally, on the event $A_n := \{ E(1) \le n\}$ we have $\chi(E(s)/n) = 1$ for every $s \in [0,1]$, hence ${\wt V}_n \equiv {\wt V}$; since $E(1) < \infty$ almost surely, the events $A_n$ increase to a set of full measure.

By Girsanov's theorem, followed by Jensen's inequality, we have that
\begin{align*}
&- \log \E \Big[ e^{-F(X^D(1))} \Big] \\
&= - \log \E_{\Q_n}\Big[ \exp\Big( -F(\wt X^D(1) + {\wt V}_n(1))  \Big) \Big] \\
&= - \log \wt\E\Big[ \exp\Big( -F(\wt X^D(1) + {\wt V}_n(1)) - \int_0^1 \jb{\dot {\wt V}_n(s) , d  \wt X^D(s)}_{\dot H^1} - \frac12 \int_0^1 \|\dot{{\wt V}_n}(s)\|_{\dot H^1}^2 ds\Big)\Big]\\
&\le \wt\E\Big[ F(\wt X^D(1) + {\wt V}_n(1)) + \int_0^1 \jb{\dot {\wt V}_n(s) , d  \wt X^D(s)}_{\dot H^1} + \frac12 \int_0^1 \|\dot{{\wt V}_n}(s)\|_{\dot H^1}^2 ds\Big] \\
&= \wt\E\Big[ F(\wt X^D(1) + {\wt V}_n(1))+ \frac12 \int_0^1 \|\dot{{\wt V}_n}(s)\|_{\dot H^1}^2 ds\Big],
\end{align*}

\noi
where in the last step we used that, by \eqref{Novikov}, the stochastic integral is a martingale and hence has vanishing expectation.
We now let $n \to \infty$.
The energy term converges by \eqref{Vntrunc1}.
For the other term, we use that ${\wt V}_n \equiv {\wt V}$ on $A_n$ and that $F$ is bounded, to write
\begin{align*}
\wt\E\Big[ F(\wt X^D(1) + {\wt V}_n(1))\Big]
& = \wt\E\Big[ F(\wt X^D(1) + {\wt V}(1)) \ind_{A_n}\Big] + \wt\E\Big[ F(\wt X^D(1) + {\wt V}_n(1)) \ind_{A_n^c}\Big] \\
& \le \wt\E\Big[ F(\wt X^D(1) + {\wt V}(1)) \ind_{A_n}\Big] + \big( \sup F \big) \, \wt \P (A_n^c) .
\end{align*}

\noi
Since $F$ is bounded and $\wt \P (A_n) \to 1$, the first term converges to $\wt\E [ F(\wt X^D(1) + {\wt V}(1))]$ and the second one to $0$, so that
$$ - \log \E \Big[ e^{-F(X^D(1))} \Big] \le \wt\E\Big[ F(\wt X^D(1) + {\wt V}(1))+ \frac12 \int_0^1 \|\dot{{\wt V}}(s)\|_{\dot H^1}^2 ds\Big],$$
which is \eqref{HHwge1} for bounded $F$.

We now remove the assumption that $F$ is bounded from above.
For $m \in \N$, the functional $F \wedge m$ is bounded and measurable, so that, by the above,
$$ - \log \E \Big[ e^{-(F\wedge m)(X^D(1))} \Big] \le \wt\E\Big[ (F\wedge m)(\wt X^D(1) + {\wt V}(1))+ \frac12 \int_0^1 \|\dot{{\wt V}}(s)\|_{\dot H^1}^2 ds\Big] ,$$
and the right-hand side is bounded by $\wt\E[ F(\wt X^D(1) + {\wt V}(1))+ \frac12 \int_0^1 \|\dot{{\wt V}}\|_{\dot H^1}^2 ds]$, since $F \wedge m \le F$.
As $m \to \infty$, we have $e^{-(F\wedge m)} \downarrow e^{-F}$ with $e^{-(F\wedge m)} \le e^{-\inf F}$, so that, by dominated convergence,
$- \log \E [ e^{-(F\wedge m)(X^D(1))} ] \to - \log \E [ e^{-F(X^D(1))} ]$.
This gives \eqref{HHwge1}, and taking the infimum over $V^{\rm w} \in \mathbb H^{\rm w}$ yields \eqref{HHwge}, which finishes the proof.
\end{proof}

\begin{lemma} \label{lem:var_limit}
    Let $F : L^2_{0, {\rm rad}} (D) \to (-\infty, +\infty]$ be lower semicontinuous and bounded below.
    Let $X^D_n(t)$ be a sequence of cylindrical Brownian motions with the same law as $X^D$ in \eqref{Yt}, and let $V_n(t)$ be a sequence of processes satisfying the conditions in the definition of $\mathbb H$ with $X^D$ replaced by $X^D_n$, namely, $V_n$ is progressively measurable with respect to the augmented filtration $\big(\F^{X^D_n}_t\big)$ generated by $X^D_n$, $V_n(0) = 0$, and $\dot V_n \in L^2_t([0,1]; H^1_{0,{\rm rad}} (D))$ almost surely, such that
    \begin{align}\label{var_e}
    \sup_{n} \E\bigg[ \int_0^1 \|\dot V_n\|_{\dot H^1}^2 dt \bigg] < \infty.
    \end{align}
    Then, there exists a weak drift $(\wt \O, \wt \F, (\wt \G_t), \wt \P, \wt X^D, \wt V) \in \mathbb H^{\rm w}$ such that
    \begin{equation}
    \label{var_limit}
    \begin{split}
        \wt \E\bigg[ & F (\wt {X}^D(1) + \wt V(1)) + \frac{1}{2} \int_0^1 \| \dot{\wt V} \|_{\dot H^1_{\rm{rad}}}^2 dt \bigg] \\
        & \le \liminf_{n\to \infty} \E\bigg[ (F\wedge n) (X_n^D(1) + V_n(1)) + \frac{1}{2} \int_0^1 \| \dot V_n \|_{\dot H^1_{\rm{rad}}}^2 dt \bigg].
    \end{split}
    \end{equation}

\end{lemma}

\begin{proof}

Consider the couple
$$ (X^D_n(t), V_n(t)) \in C_t([0,1], L^2_{0, {\rm rad}} (D)) \times C_t([0,1], L^2_{0, {\rm rad}} (D)). $$
We pick a subsequence such that the liminf in \eqref{var_limit} is actually a limit. Abusing notation, we still denote this subsequence $(X^D_n(t), V_n(t))$.
Because of our hypotheses on $X^D_n, V_n$, the laws of these couples are tight on the space $C_t([0,1], L^2_{0, {\rm rad}} (D)) \times C_t([0,1], L^2_{0, {\rm rad}} (D))$. Therefore, by Prokhorov's theorem (see, e.g., \cite[Theorem 23.2]{K21}), up to taking a further subsequence, $\Law(X^D_n(t), V_n(t))$ admits a weak limit. We can then invoke Skorokhod's representation theorem, and we have that, up to taking another subsequence, there exists a probability space $(\wt \Omega, \wt \F, \wt \P)$ and random variables $(\wt{X^D_n}(t), \wt{V_n}(t))$ with the same law as $(X^D_n(t), V_n(t))$, such that
\begin{align}
\label{Skorokhod}
\lim_{n \to \infty} (\wt{X^D_n}(t), \wt{V_n}(t)) = (\wt{X^D}(t), \wt{V}(t))
\end{align}
in the space $C_t([0,1], L^2_{0, {\rm rad}} (D)) \times C_t([0,1], L^2_{0, {\rm rad}} (D))$, $\wt \P$-almost surely.
In particular $\wt{X^D}$ has the same law as $X^D$, and $\wt V(0) = 0$ with $\dot{\wt V} \in L^2_t([0,1]; H^1_{0,{\rm rad}}(D))$ almost surely, by \eqref{var_e} and Fatou's lemma.
We now claim that
\begin{align}
\label{var_step2}
\begin{split}
\wt\E\bigg[ & F (\wt{X}^D(1) + \wt{V}(1)) + \frac{1}{2} \int_0^1 \| \dot{\wt{V}} \|_{\dot H^1_{\rm{rad}}}^2 dt \bigg] \\
& \le \liminf_{n\to \infty} \E\bigg[ (F\wedge n) (X_n^D(1) + V_n(1)) + \frac{1}{2} \int_0^1 \| \dot V_n \|_{\dot H^1_{\rm{rad}}}^2 dt \bigg].
\end{split}
\end{align}

\noi
Indeed, lower semicontinuity of $F$ and of $V \mapsto \int_0^1 \|\dot V\|^2_{\dot H^1_{\rm rad}} dt$ along \eqref{Skorokhod}, followed by Fatou's lemma, give
\begin{align*}
&\wt\E\bigg[ F (\wt{X}^D(1) + \wt{V}(1)) + \frac{1}{2} \int_0^1 \| \dot{\wt{V}} \|_{\dot H^1_{\rm{rad}}}^2 dt \bigg] \\
\le&\ \wt \E\bigg[ \liminf_{n\to \infty}\bigg( F (\wt{X_n^D}(1) + \wt{V_n}(1)) + \frac{1}{2} \int_0^1 \| \dot{\wt{V_n}} \|_{\dot H^1_{\rm{rad}}}^2 dt\bigg) \bigg] \\
=&\  \wt\E\bigg[ \liminf_{n\to \infty}\bigg( (F\wedge n) (\wt{X_n^D}(1) + \wt{V_n}(1)) + \frac{1}{2} \int_0^1 \| \dot{\wt{V_n}} \|_{\dot H^1_{\rm{rad}}}^2 dt\bigg) \bigg] \\
\le &\ \liminf_{n\to \infty} \wt \E\bigg[ (F\wedge n) (\wt{X_n^D}(1) + \wt{V_n}(1)) + \frac{1}{2} \int_0^1 \| \dot{\wt{V_n}} \|_{\dot H^1_{\rm{rad}}}^2 dt \bigg] \\
=&\ \liminf_{n\to \infty} \E\bigg[ (F\wedge n) (X_n^D(1) + V_n(1)) + \frac{1}{2} \int_0^1 \| \dot V_n \|_{\dot H^1_{\rm{rad}}}^2 dt \bigg],
\end{align*}

\noi
where in the second step we used that, for any sequences $(y_n)$ in $L^2_{0, {\rm rad}} (D)$ and $(e_n)$ in $[0, +\infty]$,
$\liminf_{n \to \infty} \big( (F\wedge n)(y_n) + e_n \big) = \liminf_{n \to \infty} \big( F(y_n) + e_n \big)$.
Indeed, the inequality $\le$ follows from $F \wedge n \le F$. Conversely, if the left-hand side is equal to some $L < \infty$, then, along a subsequence realising it, we have $(F\wedge n)(y_n) \le L + 1 < n$ for $n$ large, since $e_n \ge 0$; hence $(F\wedge n)(y_n) = F(y_n)$ along this subsequence, and the right-hand side is at most $L$.
In the third step we used Fatou's lemma, which applies since $F \wedge n \ge \inf F$ uniformly in $n$.
Since the subsequence was chosen so that the liminf in \eqref{var_limit} is a limit, this is precisely \eqref{var_step2}.

It remains to check that $\wt V$ is a weak drift, i.e.\ that $\wt{X^D}$ is still a cylindrical Brownian motion for a filtration to which $\wt V$ is adapted.
Set
\[
\wt \G_t := \s \big( \wt{X^D}(u), \wt V(u) : 0 \le u \le t\big) .
\]

\noi
Since $\wt{X^D}$ has the same law as $X^D$, it is a cylindrical Brownian motion with respect to its own filtration. Hence, it suffices to show that, for every $0 \le s \le 1$, the $\s$-algebra $\wt \G_s$ is independent of $\s \big( \wt{X^D}(t) - \wt{X^D}(s) : s \le t \le 1\big)$.
To this end, let $j, k \in \N$, let $G: \big(L^2_{0, {\rm rad}} (D)\big)^{2j} \to \R$ and $H: \big(L^2_{0, {\rm rad}} (D)\big)^{k} \to \R$ be continuous and bounded, and let $0 \le u_1 \le \dots \le u_j \le s \le s_1 \le \dots \le s_k \le 1$.
We write
\begin{align*}
\wt{\mathbf p} & = \big(\wt V(u_1), \dotsc, \wt V(u_j), \wt{X^D}(u_1), \dotsc, \wt{X^D}(u_j)\big), \\
\wt{\mathbf q} & = \big(\wt{X^D}(s_1) - \wt{X^D}(s), \dotsc, \wt{X^D}(s_k) - \wt{X^D}(s)\big),
\end{align*}

\noi
and we define $\wt{\mathbf p}_n, \wt{\mathbf q}_n$ and $\mathbf p_n, \mathbf q_n$ in the same way, with $(\wt V, \wt{X^D})$ replaced by $(\wt{V_n}, \wt{X^D_n})$ and by $(V_n, X^D_n)$, respectively.
We have that
\begin{align*}
\wt\E\big[G(\wt{\mathbf p}) H(\wt{\mathbf q}) \big]
& = \lim_{n \to \infty} \wt\E\big[G(\wt{\mathbf p}_n) H(\wt{\mathbf q}_n) \big]\\
&= \lim_{n \to \infty} \E\big[G(\mathbf p_n) H(\mathbf q_n) \big] \\
& = \lim_{n \to \infty} \E\big[G(\mathbf p_n)\big] \, \E\big[H(\mathbf q_n) \big]\\
&= \wt\E\big[G(\wt{\mathbf p})\big] \, \wt\E\big[H(\wt{\mathbf q}) \big],
\end{align*}
where we used \eqref{Skorokhod} and the boundedness and continuity of $G$ and $H$ in the first step, the equality of laws in the second step,  the fact that $\mathbf p_n$ is $\F^{X^D_n}_s$-measurable, hence independent of the increments of the cylindrical Brownian motion $X^D_n$ after time $s$, in the third step, and equality of laws together with \eqref{Skorokhod} in the last step.

\noi
Since the random vectors of the form $\wt{\mathbf p}$ generate $\wt \G_s$, and those of the form $\wt{\mathbf q}$ generate $\s \big( \wt{X^D}(t) - \wt{X^D}(s) : s \le t \le 1\big)$, a monotone class argument then gives the claimed independence, so that $\wt{X^D}$ is a cylindrical Brownian motion with respect to $(\wt \G_t)$.
Since $\wt V$ is $(\wt \G_t)$-adapted and almost surely continuous in time, it is $(\wt \G_t)$-progressively measurable, and therefore $(\wt \O, \wt \F, (\wt \G_t), \wt \P, \wt{X^D}, \wt V) \in \mathbb H^{\rm w}$.
Together with \eqref{var_step2}, this proves \eqref{var_limit}.
\end{proof}

\begin{corollary}
\label{COR:BD}
Let $F : L^2_{0, {\rm rad}} (D) \to (-\infty, +\infty]$ be lower semicontinuous, and suppose that $\E [e^{-F( X^D(1))}] < \infty$.
Then we have
\begin{align}
- \log \E\Big[e^{-F( X^D(1))}\Big] = \inf_{V^{\rm w} \in {\mathbb H^{\rm w}}}
\wt \E\bigg[ F(\wt X^D(1) + \wt V(1)) + \frac{1}{2} \int_0^1 \| \dot{\wt V} \|_{\dot H^1}^2 dt \bigg],
\label{CP3}
\end{align}
with the convention $-\log 0 = +\infty$.
Here, the expectation on the right-hand side is set to be $+\infty$ for the weak drifts with $\wt \E \big[ \int_0^1 \| \dot{\wt V} \|_{\dot H^1}^2 dt \big] = \infty$.
\end{corollary}
Before proceeding to the proof, we note that for every weak drift with $\wt \E \big[ \int_0^1 \| \dot{\wt V} \|_{\dot H^1}^2 dt \big] < \infty$, the negative part of $F(\wt X^D(1) + \wt V(1))$ is integrable by \eqref{Fminus} below, so that the expectation is well defined in $(-\infty, +\infty]$.

\begin{proof}
We first consider the case where $F$ is bounded from below. We start by proving \eqref{CP3}. By \eqref{HHw}, we already have the $\le$; in particular, if $\E[e^{-F( X^D(1))}] = 0$, then both sides of \eqref{CP3} are $+\infty$. We may thus assume that $\E[e^{-F( X^D(1))}] > 0$, and we just need to prove that
$$ - \log \E\Big[e^{-F( X^D(1))}\Big] \ge \inf_{V^{\rm w} \in {\mathbb H^{\rm w}}}
\wt \E\bigg[ F(\wt X^D(1) + \wt V(1)) + \frac{1}{2} \int_0^1 \| \dot{\wt V} \|_{\dot H^1}^2 dt \bigg].$$
We consider the approximating sequence of functionals
$$
F_n := F \wedge n , \qquad \text{ that is, } \quad
F_n (u) =
\begin{cases}
F (u) & \textup{ if } F(u) \le n,\\
n & \textup{ otherwise}.
\end{cases}
$$

Let $n\gg 1$.
Since $F_n$ is bounded and measurable, Proposition \ref{LEM:BD} applies to it and provides a drift $V_n \in \mathbb H$ which is a $\frac1n$-almost optimiser, in the sense that
\begin{align}
\label{P33}
- \log \E\Big[e^{-F_n( X^D(1))}\Big] \ge \E\bigg[ F_n(X^D(1) + V_n(1)) + \frac{1}{2} \int_0^1 \| \dot V_n \|_{\dot H^1}^2 dt \bigg] - \frac1n.
\end{align}
Since $F_n \le F$, we conclude that
\[
- \log \E\Big[e^{-F_n( X^D(1))}\Big] \le - \log \E\Big[e^{-F( X^D(1))}\Big] < \infty,
\]

\noi
which, together with the lower bound $F_n \ge \inf F$, further implies that
\begin{align}
\label{P331}
 \frac{1}{2} \E\bigg[ \int_0^1 \| \dot V_n \|_{\dot H^1}^2 dt \bigg] \le \max(-\inf F, 0) +  \E\bigg[ F_n(X^D(1) + V_n(1)) + \frac{1}{2} \int_0^1 \| \dot V_n \|_{\dot H^1}^2 dt \bigg] \les 1
\end{align}

\noi
uniformly in $n \in \N$, where the implicit constant may depend on the particular choice of $F$.
In particular, the hypothesis \eqref{var_e} is satisfied.
Therefore, by Lemma \ref{lem:var_limit}, we have that there exists a weak drift $V_\ast = (\wt \O, \wt \F, (\wt \G_t), \wt \P, \wt X^D, \wt V_\ast) \in \mathbb H^{\rm w}$ such that
\begin{align}
\label{P332}
\begin{split}
\wt \E & \bigg[ F (\wt X^D (1) + \wt V_\ast(1)) + \frac{1}{2} \int_0^1 \| \dot{\wt V}_\ast \|_{\dot H^1}^2 dt \bigg]  \\
& \le  \liminf_{n \to \infty} \E\bigg[ F_n (X^D(1) +  V_n(1)) + \frac{1}{2} \int_0^1 \| \dot{ V_n} \|_{\dot H^1}^2 dt \bigg].
\end{split}
\end{align}

\noi
Moreover, $F_n \uparrow F$ pointwise and $e^{-F_n} \le e^{-\inf F}$, so that, by the dominated convergence theorem,
\begin{align}
\label{P311}
 \lim_{n \to \infty} \log \E\Big[e^{-F_n( X^D(1))}\Big] = \log \E\Big[e^{-F( X^D(1))}\Big].
\end{align}

\noi
From \eqref{P33}, \eqref{P311}, \eqref{P331}, and \eqref{P332}, we have
\begin{align}
\label{P34}
\begin{split}
\inf_{V^{\rm w} \in \mathbb H^{\rm w}} &
\wt \E\bigg[ F (\wt X^D(1) + \wt V(1)) + \frac{1}{2} \int_0^1 \| \dot{\wt V} \|_{\dot H^1}^2 dt \bigg] \\
& \le \wt \E\bigg[ F (\wt X^D(1) + \wt V_\ast(1)) + \frac{1}{2} \int_0^1 \| \dot{\wt V}_\ast \|_{\dot H^1}^2 dt \bigg]\\
& \le \liminf_{n \to \infty} \E\bigg[ F_n (X^D(1) +  V_n(1)) + \frac{1}{2} \int_0^1 \| \dot{ V_n} \|_{\dot H^1}^2 dt \bigg]\\
& \le \liminf_{n \to \infty} \Big( -\log \E \big[ \exp(-F_n (X^D (1))) \big] + \frac 1n \Big) \\
& = -\log \E \big[ \exp(-F (X^D (1))) \big],
\end{split}
\end{align}
which concludes the proof of \eqref{CP3} in the case where $F$ is bounded from below.

We now consider a general $F$. Let $V^{\rm w} \in \mathbb H^{\rm w}$ be a weak drift, with components denoted as in Lemma~\ref{lem:weak_drift}, such that $\wt\E \big[ \int_0^1 \| \dot{\wt V} \|_{\dot H^1}^2 dt \big] < \infty$, and write $Y = \wt X^D(1) + \wt V(1)$ and $F^- = \max (-F, 0)$.
We first show that
\begin{align}
\label{Fminus}
\wt \E \big[ F^-(Y) \big] \le \log \Big( 1 + \E \Big[ e^{-F(X^D(1))} \Big] \Big) + \frac12 \wt\E \bigg[ \int_0^1 \| \dot{\wt V} \|_{\dot H^1}^2 dt \bigg] < \infty.
\end{align}

\noi
Indeed, for $k \in \N$, the functional $-(F^- \wedge k)$ is bounded and Borel measurable, so that, by \eqref{HHw},
\[
- \log \E \Big[ e^{(F^- \wedge k)(X^D(1))} \Big] \le - \wt \E \big[ (F^- \wedge k) (Y) \big] + \frac12 \wt\E \bigg[ \int_0^1 \| \dot{\wt V} \|_{\dot H^1}^2 dt \bigg].
\]
Since $e^{F^- \wedge k} \le e^{F^-} \le 1 + e^{-F}$, \eqref{Fminus} follows by letting $k \to \infty$, by monotone convergence.
In particular, the right-hand side of \eqref{CP3} is well defined.

For $a < 0$, the functional $F_a = \max (F, a)$ is lower semicontinuous and bounded below, and $\E [ e^{-F_a (X^D(1))} ] \le \E [ e^{-F (X^D(1))} ] + e^{-a} < \infty$, so that \eqref{CP3} holds for $F_a$ by the first part of the proof.
As $a \to - \infty$, we have $e^{-F_a} \uparrow e^{-F}$, and thus $- \log \E [ e^{-F_a (X^D(1))} ] \to - \log \E [ e^{-F (X^D(1))} ]$ by monotone convergence.
Since $F_a \ge F$, \eqref{CP3} for $F_a$ gives
\[
- \log \E \Big[ e^{-F_a (X^D(1))} \Big] \ge \inf_{V^{\rm w} \in {\mathbb H^{\rm w}}} \wt \E\bigg[ F(\wt X^D(1) + \wt V(1)) + \frac{1}{2} \int_0^1 \| \dot{\wt V} \|_{\dot H^1}^2 dt \bigg],
\]
and letting $a \to -\infty$ yields the inequality $\ge$ in \eqref{CP3}.
Conversely, since weak drifts with infinite energy do not contribute to the infimum in \eqref{CP3}, it suffices to consider $V^{\rm w}$ and $Y$ as above.
If $\wt \E [ \max (F(Y), 0) ] = \infty$, then $\wt \E [F(Y)] = +\infty$ by \eqref{Fminus}, and there is nothing to prove.
Otherwise, since $- F^- \le F_a \le \max (F, 0)$ for $a < 0$ and $F_a \downarrow F$ as $a \to -\infty$, the dominated convergence theorem and \eqref{Fminus} give $\wt \E [F_a(Y)] \to \wt \E [F(Y)]$, so that, by \eqref{CP3} for $F_a$,
\[
- \log \E \Big[ e^{-F (X^D(1))} \Big] = \lim_{a \to -\infty} \bigg( - \log \E \Big[ e^{-F_a (X^D(1))} \Big] \bigg) \le \wt \E\big[ F(Y) \big] + \frac{1}{2} \wt\E \bigg[ \int_0^1 \| \dot{\wt V} \|_{\dot H^1}^2 dt \bigg].
\]
Taking the infimum over $V^{\rm w}$ gives the inequality $\le$ in \eqref{CP3}, which concludes the proof.
\end{proof}

\begin{remark}
\rm
\label{RMK:Hw}
In the rest of the paper we use Corollary \ref{COR:BD} in the form \eqref{CP3}, and, in order not to burden the notation, we keep writing $\inf_{V \in \mathbb H}$ for the infimum over weak drifts; that is, from now on a {\it drift} means a weak drift in the sense introduced in Lemma \ref{lem:weak_drift}.
\end{remark}

\begin{remark}
\rm
\label{RMK:domain}
In this subsection, we  did not specify the domain $D$.
In our later analysis, we will mainly consider two types of domains for $D$. One type consists of all balls $B_\dl$ with radius $\dl$ for $\dl \le 1$; we also consider shells $B_1 \backslash B_\dl$ for $\dl < 1$, which are the regions between two concentric spheres.
\end{remark}

\section{Markov property}
\label{SEC:Markov}
This section is devoted to the study of the Markov property of the radial Gaussian free field (GFF), which is one of the crucial ingredients in the proof of Theorem \ref{THM:main}.
Nelson \cite{Nel73}  first observed the \emph{spatial} Markov property for the GFF (without radial assumptions). This informally states that for a ``nice" open domain $U \subset \R^d$, and for every observable $F$ that depends only on  $u|_{U}$, one has that
$$ \E\big[F(u) \big| u|_{U^c}\big] = \E\big[F(u)\big|u|_{\partial U}\big]. $$
Namely, for observables of this form, knowing $u$ on the boundary $\partial U$ gives the same amount of information as knowing $u$  on the whole  complement $U^c$.
This enables one to localise random fields into sub-domains without having to retain the whole knowledge of the field on the complement of the domain.
For a detailed discussion on the Markov property of the GFF, see \cite[Subsection 2.6]{Sh07} and \cite[Chapter 4]{WP21}.

Given a radially symmetric bounded domain $D \subset \R^d$ and an open subset $U$ of $D$, we denote by $H_0^1(U)$ the subspace of $H_0^1(D)$ given by (radial) functions that vanish on $D\setminus U$.
When it comes to our work, we will
exploit the Markov property of the GFF to isolate the singularity of the quartic interaction near the origin, inside $B_\dl$ for $\dl \ll 1$. To this aim, we introduce some notations.

Similarly to the random field $X^+_\dl$ defined in \eqref{mu+}, we define $X^-_\dl$ as a Gaussian random field on $B_\dl$ vanishing at the boundary,
whose law is given by the Gaussian measure
\begin{align}\label{mu-}
d \mu^-_\dl (u) = \frac 1 {Z^-_\dl} \exp\Big( - \frac 12 \int_{ B_\dl} |\nabla u|^2 dx\Big) du,
\end{align}

\noi
where $\mu^-_\dl$ is
supported on distributions with $0$ boundary conditions on $\partial B_\dl = \{|x| = \dl\}$, and extended to $0$ on  $B_1\backslash B_\dl = \{\dl \le |x| \le 1\}$.

In order to fully decompose the radial GFF $X = X^{B_1}$ as in \eqref{gaussian},
we need an additional series of one-parameter fields.

\begin{definition}
\label{DEF:harmonic}
Define the random field $X^\partial_\dl$ over $B_1$ by
\begin{align} \label{Xp2}
X^\partial_\dl(r) = \Big(\frac{r^{-d+2} - 1}{\dl^{-d+2} -1}\ind_{B_1\backslash B_\dl} + \ind_{B_\dl}\Big)\cdot X (\dl),
\end{align}

\noi
where $X(\dl)$ is the random variable given by the random field $X$ in \eqref{gaussian}  evaluated at $r = \dl$.
\end{definition}

As $X(\dl)$ given by \eqref{gaussian} is a Gaussian random variable, it follows that $X^\partial_\dl$ is also Gaussian.
It is easy to check that
$X^\partial_\dl$ satisfies the equation
\begin{equation}
\label{Xp1}
\Delta X^\partial_\dl (x) = 0 \text{ for } |x|<1 \textup{ and } |x| \neq \dl.
\end{equation}

\noi
In what follows, we realise $X^+_\dl$, $X^\partial_\dl$ and $X^{-}_\dl$ on a common probability space, such that they are independent.
Moreover, since $X^\pm_\dl$ vanish on $\{|x| = \dl\}$, \eqref{Xp2} gives $(X^+_\dl + X^\partial_\dl + X^-_\dl)(\dl) = X^\partial_\dl(\dl) = X(\dl)$.

Given any radial test functions $\ph^+$ and $\ph^-$ supported on $\cj{B_1\backslash B_\dl}$ and $B_\dl$, respectively, a similar argument as in \eqref{test} yields
\begin{align}
    \label{test2}
    \E \big[ \jb{X^\pm_\dl,\ph^\pm}_{\dot H^1}^2 \big] = \|\ph^\pm\|_{\dot H^1}^2.
\end{align}

We now show that the law of the random variable $\wt X = X^+_\dl + X^\partial_\dl + X^-_\dl$ is the same as that of the Gaussian random variable $X$ given in \eqref{gaussian}.

\begin{lemma}[Markov property]
\label{LEM:dec}
Let $X^+_\dl, X^\partial_\dl$ and $X^-_\dl$ be defined as above. Then $\Law(X^+_\dl + X^\partial_\dl + X^-_\dl) = \mu$, where $\mu$ is given in \eqref{law}, the law of the Gaussian random field $X$ defined in \eqref{gaussian}.
\end{lemma}

\begin{remark}
\rm
Recalling that  $X^\partial_\dl$ depends only on $X(\dl)$, we have that Lemma \ref{LEM:dec} manifests the Markov property of the radial GFF that we explained in the first paragraph of this section, by taking  $D = \{|x| < 1\}$ and  $U = D \backslash \{|x| = \dl\}$.
Additionally, due to the radial assumption,
the space $(H_0^1 (U))^{\perp}$ is only one-dimensional and has an explicit formula \eqref{Xp2}.
\end{remark}

\begin{proof}
As $X^+_\dl + X^\partial_\dl + X^-_\dl$ is Gaussian, in order to prove the lemma, one only needs to check the covariance relationship. More precisely, for every radially symmetric test function $\ph$, it suffices to show
\begin{align}
\E[\jb{X^+_\dl + X^\partial_\dl + X^-_\dl, \ph}_{\dot H^1}^2] = \|\ph\|_{\dot H^1}^2,
\label{Mar}
\end{align}

\noi
in view of \eqref{test}.
To prove \eqref{Mar}, we decompose $\ph$ as
$$\ph(r) = \ph_1 (r) + \ph_2 (r),$$
where $\ph_1 (r)$ and $\ph_2 (r)$ are in  $H^1_0 (B_1)$, and
\begin{align}
\begin{split}
\ph_1 (r) & = \Big[\ph(r) - \Big(\frac{r^{-d+2} - 1}{\dl^{-d+2} -1}\ind_{B_1\backslash B_\dl} + \ind_{B_\dl}\Big)\cdot \ph(\dl)\Big], \\
\ph_2 (r) & = \Big(\frac{r^{-d+2} - 1}{\dl^{-d+2} -1}\ind_{B_1\backslash B_\dl} + \ind_{B_\dl}\Big)\cdot \ph(\dl).
\end{split}
\label{ph12}
\end{align}

\noi
As $\ph_1$ vanishes at $r = \dl$ and $\ph_2$ is harmonic away from the boundary $\{|x| = 1\} \cup \{|x| = \dl\}$, it follows that $\jb{\ph_1,\ph_2}_{\dot H^1} = 0$ and thus
\begin{align}
    \label{oth}
    \|\ph \|_{\dot H^1}^2 = \|\ph_1 \|_{\dot H^1}^2 + \|\ph_2 \|_{\dot H^1}^2.
\end{align}

\noi
Similarly, we note that
\begin{align}
\label{M1}
\E [\jb{X^+_\dl , \ph}_{\dot H^1}^2] =  \E [\jb{X^+_\dl , \ph_1 }_{\dot H^1}^2] = \|\ph_1 \|_{\dot H^1 (B_1 \backslash B_\dl)}^2,
\end{align}
and
\begin{align}\label{M2}
\E [\jb{X^-_\dl, \ph}_{\dot H^1}^2] =  \E [\jb{X^-_\dl , \ph_1 }_{\dot H^1}^2]  = \|\ph_1 \|_{\dot H^1 ( B_\dl)}^2.
\end{align}

\noi
Since the random field $X^\partial_\dl$ is harmonic  away from $\{|x| = \dl\}$ by \eqref{Xp1}, while $\ph_1$ vanishes on $\{|x| = \dl\} \cup \partial B_1$, by Stokes' theorem, we have that $\jb{X^\partial_\dl, \ph_1}_{\dot H^1} = 0$, which together with \eqref{Xp2} implies
\begin{align}
\label{M3}
\E [ \jb{X^\partial_\dl, \ph}_{\dot H^1}^2]  & = \E [\jb{X^\partial_\dl, \ph_2}_{\dot H^1}^2] \notag\\
& = \E \bigg[ \Big \langle  \Big(\frac{r^{-d+2} - 1}{\dl^{-d+2} -1}\ind_{B_1\backslash B_\dl} + \ind_{B_\dl}\Big)\cdot X (\dl), \ph_2 \Big \rangle_{\dot H^1}^2 \bigg] \notag\\
& =  |\ph (\dl)|^2 \Big \| \Big(\frac{r^{-d+2} - 1}{\dl^{-d+2} -1}\ind_{B_1\backslash B_\dl} + \ind_{B_\dl}\Big)\Big \|_{\dot H^1}^4 \E \big[ |X (\dl) |^2 \big], \notag\\
\intertext{which, in view of \eqref{ph12} and $\E [|X(\dl)|^2] = K(\dl,\dl)$ (see \eqref{V_dl} below), we may rewrite as}
& = \|\ph_2 \|_{\dot H^1}^2  K(\dl,\dl) \int_\dl^1 \frac{(d-2)^2 r^{- d+1}}{(\dl^{-d+2} - 1)^2} dr \notag\\
& = \|\ph_2 \|_{\dot H^1}^2,
\end{align}

\noi
where the radial Green function is given in \eqref{Green_r}. Here, we used
\[
\int_\dl^1 \frac{(d-2)^2 r^{- d+1}}{(\dl^{-d+2} - 1)^2}  dr
= \frac{(d-2) }{\dl^{-d+2} -1} = (K(\dl,\dl))^{-1}
\]
in the last step.
By collecting \eqref{oth}, \eqref{M1}, \eqref{M2}, and \eqref{M3}, together with the independence of $X^+_\dl, X^\partial_\dl, X^-_\dl$, we arrive at \eqref{Mar}.
\end{proof}

To have a further description of the measure $\mu$ in Lemma \ref{LEM:dec}, we define
\begin{align}
\begin{split}
\mu^+_\dl[\ld] &= \Law (X^+_\dl + \ld  h_\dl ),\\
\mu^-_\dl[\ld] &= \Law(X^-_\dl + \ld  \ind_{B_\dl}),
\end{split}
\label{mu_pm}
\end{align}

\noi
where $X^+_\dl$ and $X^-_\dl$ are given by \eqref{mu+} and \eqref{mu-}, and
\begin{align}
\label{hdelta}
h_\dl (x) = \frac{|x|^{-d+2} - 1}{\dl^{- d+ 2} -1}\ind_{B_1\backslash B_\dl}.
\end{align}

\noi
It is easy to see that when $\ld = 0$ in \eqref{mu_pm}, the measures $\mu_\dl^\pm [0]$ coincide with the measures $\mu_\dl^\pm$ defined in \eqref{mu+} and \eqref{mu-}.
We also define a measure for the boundary value at $r = |x| = \dl$ by
\begin{align}
\label{mu_b}
\mu^\partial_\dl & = \Law  (X  (\dl)),
\end{align}

\noi
where $X(\dl)$ is the random variable given by the random field $X$ in \eqref{gaussian}  evaluated at $r = \dl$.
From \eqref{gaussian}, we see that $X(\dl)$ is a  centred Gaussian random variable with mean $0$ and variance
\begin{align}
\label{V_dl}
{\rm Var} (X(\dl)) = \sum_{n \ge 0} \frac{e_n^2 (\dl)}{\ld_n^2} = K(\dl,\dl) = \frac1{d-2} \big( \dl^{-d+2} -1 \big),
\end{align}

\noi
where $K(\cdot, \cdot)$ is the Green function defined in \eqref{Green_k}.

\begin{lemma}
\label{LEM:Xal}
Let $d \ge 3$ and $\al \in (0,\frac12)$. Then the radial Gaussian free field $X$ in \eqref{gaussian} has a version with values in the space $\mathbb X^\al$ defined in \eqref{Xp}.
In particular, $\mu$ in \eqref{law} and each $\rho_\eps$ in \eqref{rho_eps}, $0 < \eps < 1$, can be regarded as Borel probability measures on $\mathbb X^\al$.
\end{lemma}

\begin{proof}
Fix $k \ge 2$.
By \eqref{Green_k} and \eqref{Green_r}, for $\frac1k \le s \le t \le 1$, we have
\begin{align*}
\E \big[ |X(t) - X(s)|^2 \big] & = K(t,t) - 2 K(t,s) + K(s,s) = K(s,s) - K(t,t) \\
& = \frac{s^{-d+2} - t^{-d+2}}{d-2} \le k^{d-1} (t-s),
\end{align*}

\noi
where we used the mean value theorem in the last step.
Since $X$ is Gaussian, Kolmogorov's continuity theorem, as in the proof of Lemma \ref{LEM:reg}, implies that the restriction of $X$ to $B_1 \backslash B_{1/k}$ has a version in $C^{\frac12 - \ta} (B_1 \backslash B_{1/k})$ for every $\ta > 0$. Here we used that, for radial functions, $|u(x) - u(y)| = |u(|x|) - u(|y|)|$ and $\big| |x| - |y| \big| \le |x - y|$.
Choosing $0 < \ta < \frac12 - \al$, and since $C^{\frac12 - \ta} (B_1 \backslash B_{1/k}) \subseteq c^\al (B_1 \backslash B_{1/k})$, this version takes values in $c^\al (B_1 \backslash B_{1/k})$.
Since two continuous versions of $X$ on $B_1 \backslash B_{1/k}$ are indistinguishable, almost surely these versions agree on their overlaps for all $k \ge 2$, and thus define a function in $\mathbb X^\al$, which is the desired version of $X$.
Finally, the density of $\rho_\eps$ with respect to $\mu$ in \eqref{rho_eps} is a continuous function on $\mathbb X^\al$, since $\|u\|_{L^\infty (B_1 \backslash B_\eps)}$ is controlled by one of the seminorms \eqref{Xpsemi}.
\end{proof}

Given $\dl \in (0,1)$, define the natural restriction map $R_\dl : \mathbb X^\al \to c^\al_{\rad} (B_1 \backslash B_\dl)$ by
\begin{align}
\label{R_res}
R_\dl u := u|_{B_1 \backslash B_\dl},
\end{align}

\noi
where $\mathbb X^\al$ is the space introduced in \eqref{Xp}.
By \eqref{Xpsemi}, the map $R_\dl$ is continuous, and so is $R_\dl : \mathbb X^{\al} \to L^p_{\rad}(B_1\backslash B_\dl)$ for every $1 \le p \le \infty$.

\noi
Then, we have the following representation results.

\begin{lemma}
\label{LEM:meas_rep}
Let $X^+_\dl$, $X^-_\dl$, and $X^\partial_\dl$ be as above.
Given $\ld \in \mathbb R$,
let $\mu^+_\dl[\ld]$, $\mu^-_\dl[\ld]$, and $\mu^\partial_\dl$ be as in \eqref{mu_pm} and \eqref{mu_b}.
Then, we have that
\begin{gather} \label{eq:law}
\Law (X^+_\dl + X^\partial_\dl + X^-_\dl) = \int \mu^+_\dl[\ld] \otimes \mu^-_\dl[\ld]  \mu^\partial_\dl (d \lambda).
\end{gather}
In particular, we have that
\begin{equation}\label{eq:law_R}
(R_\dl)_{\#} \Law(X) = \int \mu^+_\dl[\ld]  \mu^\partial_\dl (d \lambda),
\end{equation}

\noi
where $R_\dl$ is the natural restriction map defined in \eqref{R_res}.
\end{lemma}

\begin{proof}
To prove \eqref{eq:law}, we only need to check that both sides of \eqref{eq:law} share the same characteristic function.
To simplify the notation, we denote
\[
\wt \mu = \int \mu^+_\dl[\ld] \otimes \mu^-_\dl[\ld]  \mu^\partial_\dl (d \lambda).
\]
To prove \eqref{eq:law},
it suffices to show that for any radial test function $\phi$ on $B_1$, we have
\begin{align} \label{charac}
\int \exp  (i \jb{u, \phi})  \wt \mu (d u) = \E \big[ \exp (i \jb{X^+_\dl + X^-_\dl + X^\partial_\dl, \phi} )\big].
\end{align}
Recalling that $X^+_\dl$ and $X^-_\dl$ are supported respectively outside and inside the ball $B_\dl$, we identify $L^2(B_1) \cong L^2(B_1\backslash B_\dl) \oplus L^2(B_\dl)$.
We decompose the test function $\phi$ as $\phi = \phi^+ + \phi^-$, where $\phi^+ = \phi \ind_{B_1\backslash B_\dl}$ and $\phi^- = \phi \ind_{B_\dl}$.
Similarly, we can write $u = u^+ + u^-$, where $u^+ = u \ind_{B_1\backslash B_\dl}$ and $u^- = u \ind_{B_\dl}$.
It then follows that
\begin{align*}
\int \exp & (i \jb{u, \phi})  d \wt \mu ( u)  = \int \Big(\int \exp(i\jb{u,\phi}) d \mu^+_\dl[\ld](u^+) d \mu^-_\dl[\ld](u^-)\Big)  \mu^\partial_\dl (d \lambda) \\
& = \int \Big(\int \exp(i(\jb{u^+,\phi^+} + \jb{u^-,\phi^-})) d \mu^+_\dl[\ld](u^+) d \mu^-_\dl[\ld](u^-)\Big)  \mu^\partial_\dl (d \lambda) \\
&= \int \E \big[\exp ( i \langle {X^+_\dl + \ld   h_\dl,\phi^+} \rangle ) \big] \E \big[\exp (i\jb{X^-_\dl + \ld   \ind_{B_\dl},\phi^-} ) \big]  \mu^\partial_\dl (d \lambda), \\
\intertext{where $h_\dl$ is given in \eqref{hdelta}. Then, by the independence of $X^+_\dl$ and $X^-_\dl$, we can proceed with}
& = \int \E \big[\exp ( i \langle {X^+_\dl + \ld  h_\dl,\phi^+} \rangle  + i\jb{X^-_\dl + \ld  \ind_{B_\dl},\phi^-} ) \big]  \mu^\partial_\dl (d \lambda). \\
\intertext{By Fubini's theorem and the support property of $X^+_\dl$ and $X^-_\dl$, we have}
& = \E \big[ \exp (i \jb{X^+_\dl, \phi} + i \jb{X^-_\dl, \phi})\big] \cdot \int \exp ( i \langle {\ld  h_\dl, \phi^+} \rangle  + i\jb{ \ld  \ind_{B_\dl},\phi^-} )  \mu^\partial_\dl (d \lambda) \\
& = \E \big[ \exp (i \jb{X^+_\dl, \phi} + i \jb{X^-_\dl, \phi})\big] \cdot  \int  \exp ( i \langle {\ld  (h_\dl + \ind_{B_\dl}), \phi} \rangle ) \mu_\dl^\partial (d \ld), \\
\intertext{from \eqref{Xp2}, we observe that}
& = \E \big[ \exp (i \jb{X^+_\dl + X^-_\dl, \phi} )\big]  \cdot  \E \big[ \exp (i \jb{X^\partial_\dl, \phi})\big], \\
\intertext{finally, we use the independence again to get}
& = \E \big[ \exp (i \jb{X^+_\dl + X^-_\dl + X^\partial_\dl, \phi} )\big],
\end{align*}
which finishes the proof of \eqref{charac} and thus \eqref{eq:law}. The proof of \eqref{eq:law_R} is similar by noticing that $R_\dl (u) = u^+$.
We omit its proof.
\end{proof}

In the following, we use the Markov property for the Gaussian measure in Lemmas \ref{LEM:dec} and \ref{LEM:meas_rep} to study the  approximating Gibbs  measures $\rho_\eps$ in \eqref{rho_eps}.
Given  $0 < \eps \le \dl < 1$, let
\begin{align} \label{pot_cut}
\begin{split}
V^+_\dl (u) = \frac 14 \int_{B_1\backslash B_\dl} |u|^4 \, dx , \qquad
V^-_{\eps,\dl}(u) = \frac 14 \int_{B_\dl\backslash B_\eps} |u|^4 \, dx.
\end{split}
\end{align}
We define probability measures
\begin{align}
    \label{meas_cut}
    \begin{split}
    d \rho_\dl^+[\ld] (u) & = \frac1{Z^+_\dl[\ld]} \exp\big(-V^+_\dl (u) \big) d\mu^+_\dl[\ld](u), \\
    d \rho^-_{\eps, \dl} [\ld] (u) & = \frac1{Z^-_{\eps, \dl}[\ld]}\exp\big(-V^-_{\eps,\dl}(u) \big) d\mu^-_\dl[\ld](u),
    \end{split}
\end{align}

\noi
where $\mu_\dl^\pm [\ld]$ are given in \eqref{mu_pm}, and $Z^+_\dl[\ld]$, $Z^-_{\eps, \dl}[\ld]$ are normalising constants.
When $\eps = \dl$, we have $V^-_{\dl,\dl} = 0$, since $B_\dl \backslash B_\dl = \emptyset$, so that $\rho^-_{\dl,\dl}[\ld] = \mu^-_\dl[\ld]$ and $Z^-_{\dl,\dl}[\ld] = 1$.
Now, we are ready to state the following Markov property for the measures $\rho_\eps$.

\begin{proposition}[Markov property for $\rho_\eps$]
\label{prop:Markov}
With the above notations and  $0 < \eps \le \dl < 1$,
we have
\begin{equation} \label{eq:rhoMarkov}
(R_\dl)_{\#} \rho_\eps = \int \rho_\dl^+[\ld]  \nu_{\eps,\dl} (d \lambda),
\end{equation}
where $\rho_\eps$ is given in \eqref{rho_eps},  $\rho_\dl^+[\ld]$ is the measure defined in \eqref{meas_cut}, and $ \nu_{\eps,\dl}$ is the measure given in  \eqref{nued} below.
Moreover, $\nu_{\eps,\dl}$ is also the law of $u(\dl)$ for $u$ distributed according to $\rho_\eps$.
\end{proposition}

\begin{proof}
We first note from \eqref{rho_eps} and \eqref{pot_cut} that
$$
d\rho_\eps(u) = \frac 1 {Z_\eps} e^{-V^+_\dl (u)} e^{-V^-_{\eps,\dl}(u) } d\mu(u).
$$

\noi
Since $V^+_\dl$ is a function of $u$ restricted to $B_1\backslash B_\dl$, and $V^-_{\eps,\dl}(u)$ is a function of $u$ restricted to $B_\dl$,
it follows from \eqref{rho_eps}, Lemma \ref{LEM:dec}, and \eqref{eq:law} that
\begin{equation}
\label{rho_eps1}
\begin{split}
d \rho_\eps(u) = \frac 1 {Z_\eps} \int \Big(e^{-V^+_\dl(u^+)} d \mu^+_\dl[\ld](u^+)\Big)\otimes\Big(e^{-V^-_{\eps,\dl}(u^-)} d\mu^-_\dl[\ld](u^-)\Big) \mu^\partial_\dl (d \lambda),
\end{split}
\end{equation}
where $u^+$ and $u^-$ are restrictions of $u$ onto $B_1\backslash B_\dl$ and $B_\dl$ respectively.
In particular, for any bounded Borel function $F$, we note that
\begin{align*}
\int F(u^+) & d \big( (R_\dl)_{\#} \rho_\eps \big) (u^+)  = \int F(u^+) d \rho_\eps(u) \\
& = \frac 1{Z_\eps} \int \bigg(\int e^{-V^-_{\eps,\dl} (u^-)} d \mu^-_\dl[\ld](u^-)\bigg) \cdot \bigg( \int   F (u^+)  e^{-V^+_\dl(u^+)} d \mu^+_\dl[\ld](u^+)  \bigg)  \mu^\partial_\dl (d \lambda) \\
& = \int \bigg( \int F(u^+)   d \rho^+_{\dl} [\ld] (u^+) \bigg)  \frac{Z^+_\dl[\ld] Z^-_{\eps,\dl}[\ld]}{Z_\eps} \mu^\partial_\dl (d \lambda) ,
\end{align*}

\noi
where $Z_\dl^+ [\ld]$ and $Z^-_{\eps,\dl}[\ld]$ are given in \eqref{meas_cut}.
Thus, \eqref{eq:rhoMarkov} follows by setting
\begin{align}
\label{nued}
\nu_{\eps,\dl} (d \lambda) : = \frac{Z^+_\dl[\ld] Z^-_{\eps,\dl}[\ld]}{Z_\eps} \, \mu^\partial_\dl (d \lambda).
\end{align}
To see that $\nu_{\eps,\dl}$ is a probability measure, we note from \eqref{meas_cut} that
\[
\begin{split}
\int Z^+_\dl[\ld] Z^-_{\eps,\dl}[\ld] \mu^\partial_\dl (d\ld) & = \int \int \exp\big(-V^+_\dl (u^+) -V^-_{\eps,\dl}(u^-) \big)   d \big( \mu^+_\dl[\ld] \otimes \mu^-_\dl[\ld]\big)(u^+ , u^-) \, \mu^\partial_\dl (d\ld)\\
& = \int \exp\bigg(- \frac14 \int_{B_1 \backslash B_\eps} |u|^4 dx \bigg) d\mu = Z_\eps,
\end{split}
\]
where we used Lemma \ref{LEM:dec} and Lemma \ref{LEM:meas_rep} in the second step.

To complete the proof, we also need to show that with this definition, $\nu_{\eps,\dl}$ is the law of $u(\dl)$ when $u$ is sampled according to $\rho_\eps$. By Lemma \ref{LEM:dec} and \eqref{eq:law} once more, we obtain that for any bounded Borel function $F$,
\begin{align*}
\int F & (u(\dl))  d\rho_\eps(u) = \frac 1 {Z_\eps} \int F(u(\dl)) e^{-V^+_\dl(u)}e^{-V^-_{\eps,\dl}(u)} d\mu(u)\\
& = \frac 1 {Z_\eps}  \E \Big[ F(X (\dl)) e^{-V^+_\dl(X^+_\dl + X^\partial_\dl )} e^{-V^-_{\eps,\dl}(X^-_\dl + X^\partial_\dl )} \Big],  \\
\intertext{ where $X^\partial_\dl$ is defined in Definition \ref{DEF:harmonic}. Then we plug in \eqref{Xp2} and continue with}
& = \frac 1 {Z_\eps}  \E \bigg[ F (X (\dl)) \cdot \exp \Big(- V^+_\dl \big(X^+_\dl + (h_\dl + \ind_{B_\dl} ) \cdot X(\dl) \big)\Big) \\
& \hphantom{XXXXXX} \times \exp \Big(-V^-_{\eps,\dl} \big(X^-_\dl +  (h_\dl + \ind_{B_\dl} )\cdot X (\dl) \big) \Big) \bigg], \\
\intertext{ where $h_\dl$ is given in \eqref{hdelta}. By the independence of $X^+_\dl, X^-_\dl, X^\partial_\dl$, we have}
& =\frac 1 {Z_\eps}  \int F(\ld) \cdot \E \bigg[ \exp \Big(- V^+_\dl \big(X^+_\dl + \ld  (h_\dl + \ind_{B_\dl} )  \big)\Big) \bigg]\\
& \hphantom{XXXXXX} \times \E \bigg[\exp \Big(-V^-_{\eps,\dl} \big(X^-_\dl + \ld (h_\dl + \ind_{B_\dl} ) \big) \Big) \bigg]  \mu^\partial_\dl (d \lambda),\\
\intertext{from \eqref{mu_pm}, we may continue with}
& = \frac 1 {Z_\eps} \int \bigg(\int e^{-V^+_\dl(u^+)} d\mu^+_\dl[\ld](u^+) \bigg) \cdot \bigg(\int e^{-V^-_{\eps,\dl}(u^-)} d\mu^-_\dl[\ld](u^-) \bigg) \cdot F(\ld)  \mu^\partial_\dl (d \lambda), \\
\intertext{then we use \eqref{meas_cut} to get}
&= \int \frac{Z^+_\dl[\ld]Z^-_{\eps,\dl}[\ld]}{Z_\eps} F(\ld)  \mu^\partial_\dl (d \lambda)  = \int F(\ld)  \nu_{\eps,\dl} (d \lambda),
\end{align*}
which concludes the proof.
\end{proof}

\section{On the measure \texorpdfstring{$\rho_\dl^+[\ld]$}{rho+[lambda]}}
\label{SEC:rho+}

In this section, we  characterise the measure $\rho_\dl^+[\ld]$ and show several properties that will be useful in the proof of Theorem \ref{THM:main}.

\subsection{Variational formula for the Gibbs measure}
From now on, $X^+_\dl(t) := X^{B_1\backslash B_\dl}(t)$, $t \in [0,1]$, denotes the process \eqref{Yt} with $D = B_1\backslash B_\dl$, so that $X^+_\dl(1)$ has law $\mu^+_\dl$ in \eqref{mu+}, and we write $X^+_\dl = X^+_\dl(1)$.
Here, $t$ is the time variable; when the spatial variable is displayed, we write $X^+_\dl(t,r)$ with $r = |x|$, so that $X^+_\dl(1,r)$ is the field $X^+_\dl(r)$ of Lemma \ref{LEM:X+}.
With $h_\dl$ as in \eqref{hdelta},
we obtain the following characterisation of the measure $\rho_\dl^+[\ld]$.

\begin{proposition}
\label{PROP:rho}
Let $\ld \in \R$ and $\rho_\dl^+[\ld]$ be the measure defined in \eqref{meas_cut}, and let $F : L^2_{0,{\rm rad}} (B_1\backslash B_\dl) \to \R$ be such that
\[
u \longmapsto F(u) + \frac 14 \int_{B_1\backslash B_\dl} |u|^4 dx
\]
is lower semicontinuous and such that $\int e^{-F} d \rho^+_\dl[\ld] < \infty$.
Then
\begin{equation}\label{eq:rhoVar0}
\begin{aligned}
 \log & \int \exp(-F(u)) d \rho_\dl^+[\ld](u) \\
&= \inf_{\Theta \in  \mathbb H}
\E\bigg[ \frac14 \int_{B_1\backslash B_\dl} |X^+_\dl(1) +  \ld h_\dl + \Theta (1)|^4 dx + \frac{1}{2} \int_0^1 \| \dot \Theta(t) \|_{\dot H^1({B_1\backslash B_\dl})}^2 dt\bigg] \\
&\hphantom{XX}  - \inf_{\Theta \in  \mathbb H} \E\bigg[F\big(X^+_\dl(1) +  \ld h_\dl + \Theta(1)\big) + \frac 14\int_{B_1\backslash B_\dl} |X^+_\dl(1) +  \ld h_\dl + \Theta(1)|^4dx \\
& \hphantom{XXXXXX}  + \frac{1}{2} \int_0^1 \| \dot \Theta(t) \|_{\dot H^1({B_1\backslash B_\dl})}^2 dt	\bigg],
\end{aligned}
\end{equation}

\noi
where $X^+_\dl(t)$ is the process defined at the beginning of this subsection, $h_\dl$ is defined in \eqref{hdelta}, and $\mathbb H$ is the space of drifts as in Corollary \ref{COR:BD}, with the convention of Remark \ref{RMK:Hw}.
\end{proposition}

\begin{proof}
From the definition \eqref{meas_cut} of $\rho^+_\dl[\ld]$ and the definition \eqref{mu_pm} of $\mu_\dl^+ [\ld]$, we see that
\[
\log \int \exp(-F(u)) d \rho_\dl^+[\ld](u) = \log \bigg(  \frac1{Z^+_\dl[\ld]} \int \exp\big(- F(u) -V^+_\dl (u) \big) d\mu^+_\dl[\ld](u) \bigg),
\]

\noi
where $V_\dl^+$ is as in \eqref{pot_cut}, and $Z^+_\dl[\ld]$ is given in \eqref{meas_cut}.
Then it follows that
\begin{align}
\label{rhoVar1}
\begin{split}
\log & \int \exp(-F(u)) d \rho_\dl^+[\ld](u) \\
& = \log \int \exp \bigg (-F(u)  -\frac14 \int_{B_1 \backslash B_\dl} |u|^4 dx \bigg) d \mu_\dl^+[\ld](u) \\
& \hphantom{XX} - \log \int \exp \bigg( -\frac14 \int_{B_1 \backslash B_\dl} |u|^4 dx\bigg) d \mu_\dl^+ [\ld] (u)\\
& = \log \E \bigg[  \exp \Big (-F(X^+_\dl + \ld h_\dl)  -\frac14 \int_{B_1 \backslash B_\dl} |X^+_\dl + \ld h_\dl|^4 dx \Big) \bigg]  \\
& \hphantom{XX} - \log \E \bigg [ \exp \Big( -\frac14 \int_{B_1 \backslash B_\dl} |X^+_\dl + \ld h_\dl|^4 dx\Big) \bigg],
\end{split}
\end{align}

\noi
where we used \eqref{mu_pm} in the last step.
We first note that both functionals appearing in the right-hand side of \eqref{rhoVar1} are admissible for Corollary \ref{COR:BD}.
Indeed, $\frac14 \int_{B_1\backslash B_\dl} | \cdot |^4 dx$ is lower semicontinuous on $L^2_{0,{\rm rad}} (B_1\backslash B_\dl)$ and non-negative, while $F + \frac14 \int_{B_1\backslash B_\dl} | \cdot |^4 dx$ is lower semicontinuous by hypothesis. Moreover, the two exponential moments are finite, the first because the functional is non-negative, and the second by the assumption $\int e^{-F} d\rho^+_\dl[\ld] < \infty$, together with \eqref{rhoVar1}.
For the first term in the right-hand side of \eqref{rhoVar1}, we apply Corollary \ref{COR:BD} to get
\begin{equation}
\label{BDP1}
    \begin{aligned}
    -\log & \E\Big[\exp\Big(- F \big( X^+_\dl + \ld h_\dl \big) - \frac14 \int_{B_1 \backslash B_\dl}  |X^+_\dl +  \ld h_\dl |^4 dx\Big)\Big] \\
    &= \inf_{\Theta \in  \mathbb H}
	\E\bigg[ F \big(X^+_\dl(1) +  \ld h_\dl + \Theta (1)\big) + \frac 14\int_{B_1\backslash B_\dl} |X^+_\dl(1) +  \ld h_\dl + \Theta (1)|^4dx \\
&\phantom{=- \inf_{\Theta \in  \mathbb H}
	\E\bigg[]} + \frac{1}{2} \int_0^1 \| \dot \Theta (t) \|_{\dot H^1({B_1\backslash B_\dl})}^2 dt
	\bigg].
    \end{aligned}
\end{equation}

\noi
Similarly, we apply Corollary \ref{COR:BD} to the second term in the right-hand side of \eqref{rhoVar1} to get
\begin{equation}
\label{BDP2}
    \begin{aligned}
    -\log & \E\Big[\exp\Big( - \frac14 \int_{B_1 \backslash B_\dl}  |X^+_\dl +  \ld h_\dl |^4dx\Big)\Big] \\
    &= \inf_{\Theta \in  \mathbb H}
	\E\bigg[ \frac 14\int_{B_1\backslash B_\dl} |X^+_\dl(1) +  \ld h_\dl + \Theta (1)|^4dx + \frac{1}{2} \int_0^1 \| \dot \Theta (t) \|_{\dot H^1({B_1\backslash B_\dl})}^2 dt
	\bigg].
    \end{aligned}
\end{equation}

\noi
Then \eqref{eq:rhoVar0} follows from \eqref{rhoVar1}, \eqref{BDP1}, and \eqref{BDP2}.
\end{proof}

While the formula \eqref{eq:rhoVar0} is already extremely useful, we can characterise the measure $\rho_\dl^+[\ld]$ even more precisely.
In particular, we would like to expand the drift term $\Theta$ in such a way that there are crucial cancellations between the two terms in \eqref{eq:rhoVar0}.
To this purpose, let $z_{\dl,\ld}$ be the radial solution of the equation
\begin{equation}
\begin{cases}
\Delta z_{\dl,\ld} = z_{\dl,\ld}^3, & \text{ on } B_1\backslash B_\dl, \\
z_{\dl,\ld}(x) \big|_{|x| = \dl}  = \ld, \\
z_{\dl,\ld}(x) \big|_{|x| = 1}  = 0.
\end{cases}
\label{z}
\end{equation}

\noi
Such a solution exists and is unique: it is the unique minimiser over $z - \ld h_\dl \in H^1_{0,{\rm rad}}(B_1\backslash B_\dl)$ of the strictly convex functional $z \mapsto \frac12 \int |\nabla z|^2 dx + \frac14 \int z^4 dx$, and uniqueness follows from the strict monotonicity of $z \mapsto z^3$.
It turns out that $z_{\dl,\ld}$ captures the main defocusing structure hidden in the $\Phi^4$ measure.
This structure ensures that $z_{\dl,\ld}$ blows up in a controllable way;
see Lemma \ref{LEM:z} for further details.
The first attempt to expand $\Theta$ is
\begin{align}
\label{firstEX}
\Theta (t) = t (z_{\dl,\ld} - \ld h_\dl) + {\rm remainder}.
\end{align}

\noi
However, the expansion \eqref{firstEX} is not enough to observe the full  cancellation.
To explore the nonlinear  cancellations in \eqref{eq:rhoVar0}, further expansion of the remainder term in \eqref{firstEX} is needed.
We define a process $Z_{\dl, \ld} (t)$ as
\begin{equation}
\label{Z}
\begin{cases}
\Delta Z_{\dl,\ld} (t,x) = 3 X^+_\dl (t) z_{\dl,\ld}^2 (x), \\
Z_{\dl,\ld}(t, x) \big|_{|x| = \dl} = 0, \\
Z_{\dl,\ld}(t, x) \big|_{|x| = 1} =0.
\end{cases}
\end{equation}

\noi
Moreover, let $\Xi(t) = \Xi_{\dl,\ld}(t)$ be defined  as
\begin{equation}\label{Xidef}
\Xi(t) = \int_0^t Z_{\dl,\ld} (s) \, ds,
\end{equation}

\noi
so that $\dot \Xi = Z_{\dl,\ld}$ is a martingale.
We will also need an extra correction term. Define $\U = \U_{\dl,\ld}$
as the solution to the following differential equation
\begin{equation}
\label{Upsilon}
{\U} (t) = - \int_0^t \big( 3(1-s) z_{\dl,\ld}^2 - \Delta \big)^{-1} \Big(3\big((1-s)\dot \Xi (s) +  \Xi(s) +\U(s)\big) z_{\dl,\ld}^2\Big) ds,
\end{equation}

\noi
where $\big( 3(1-s) z_{\dl,\ld}^2 - \Delta \big)^{-1}$ denotes the inverse with Dirichlet boundary conditions on $\partial (B_1\backslash B_\dl)$; see Lemma \ref{LEM:IDe} for the well-posedness of \eqref{Upsilon}.
Equivalently, $\dot \U(t)$ vanishes on $\partial (B_1\backslash B_\dl)$ and satisfies
\begin{align}
    \label{Upsilon1}
    \Delta \dot \U = 3\Big( \big(\U(t) + \Xi(t) + (1-t)\big(\dot \U(t) + \dot \Xi(t)\big) \big) z_{\dl,\ld}^2 \Big) .
\end{align}

\noi
Note that $\dot\U$ is also a martingale.
Indeed, $\dot \U(0) = 0$, since $\dot \Xi(0) = Z_{\dl,\ld}(0) = 0$, and, taking the differential of \eqref{Upsilon1} in $t$, the terms involving $dt$ cancel and we obtain
\begin{align}
\begin{split}
\Delta \, d\dot \U(t) & = 3 (1-t) z_{\dl,\ld}^2 \big( d \dot \U(t) + d \dot \Xi(t) \big), \quad \text{that is,} \\
d \dot \U(t) & = - \big( 3(1-t) z_{\dl,\ld}^2 - \Delta \big)^{-1} \big( 3 (1-t) z_{\dl,\ld}^2 \, d \dot \Xi(t) \big).
\end{split}
\label{dUpsilon}
\end{align}

\noi
Hence, $\dot \U$ is a stochastic integral of a deterministic operator against the martingale $\dot \Xi$.

Now, we are ready to rewrite \eqref{eq:rhoVar0} in a concise manner.

\begin{proposition}\label{prop:rhoVar}
With the same  assumptions as in Proposition \ref{PROP:rho},
we have
\begin{equation}
\label{eq:rhoVar}
\begin{split}
\log & \int \exp(-F(u)) d \rho_\dl^+[\ld](u) \\
 & = \inf_{W \in  \mathbb H}
	\E\big[ P_{\dl,\ld}(W) \big]
	- \inf_{W \in  \mathbb H}
	\E\big[ F(X^+_\dl + z_{\dl,\ld} + W + \Xi + \U) + P_{\dl,\ld}(W) \big],
\end{split}
\end{equation}

\noi
where $P_{\dl,\ld}$ is given  by
\begin{align}
\label{P1}
\begin{split}
P_{\dl,\ld}  (W)
& = \frac 14 \int_{B_1\backslash B_\dl} \big(W +\Xi  + \U \big)^4 + \frac 12 \int_0^1 \| \dot W(s)\|_{\dot H^1}^2 ds\\
& \phantom{=} + \int_{B_1\backslash B_\dl} \big(W +\Xi  + \U \big)^3 (z_{\dl,\ld} + X^+_\dl )  \\
& \phantom{=} + \frac32 \int_{B_1\backslash B_\dl} W ^2(z_{\dl,\ld} + X^+_\dl )^2   \\
& \phantom{=} + 3\int_{B_1\backslash B_\dl} W (\Xi + \U ) \big((z_{\dl,\ld} + X^+_\dl)^2 - z_{\dl,\ld}^2\big)  \\
& \phantom{=} + \int_{B_1\backslash B_\dl} (W +\Xi  + \U ) \big [(z_{\dl,\ld} + X^+_\dl )^3 - z_{\dl,\ld}^3 - 3 X^+_\dl  z_{\dl,\ld}^2  \big].
\end{split}
\end{align}
Here $\Xi$ and $\U$ are defined by \eqref{Xidef} and \eqref{Upsilon}, respectively.
\end{proposition}

\begin{proof}
We expand $\Theta (t)$ as
\begin{align}\label{CoV}
\Theta (t) = (z_{\dl,\ld} - \ld h_\dl) t +  W (t) + \Xi (t) + \U (t),
\end{align}

\noi
where $W \in \mathbb H$ is a drift term, $z_{\dl, \ld}$ and $h_\dl$ are given by  \eqref{z} and \eqref{hdelta} respectively,
and $\Xi$ and $\U$ are defined in \eqref{Xidef} and \eqref{Upsilon}.
Following the convention of Remark \ref{RMK:Hw}, $W$ is a weak drift, and $\Xi$ and $\U$ are built through \eqref{Z}, \eqref{Xidef}, and \eqref{Upsilon} from the cylindrical Brownian motion of that weak drift, which we still denote by $X^+_\dl(t)$.
As $X^+_\dl(t)$ is a cylindrical Brownian motion with respect to the filtration $(\wt \G_t)$ of the weak drift, $\Xi$ and $\U$ are $(\wt \G_t)$-adapted, and $X^+_\dl(t)$, $\dot \Xi$, and $\dot \U$ are $(\wt \G_t)$-martingales, by \eqref{Z} and \eqref{dUpsilon}.
Moreover, $|z_{\dl,\ld}| \le |\ld|$ by the maximum principle.
Hence, testing \eqref{Z} against $\dot \Xi(t) = Z_{\dl,\ld}(t)$ and \eqref{Upsilon1} against $\dot \U(t)$, and using the Poincar\'e inequality $\|f\|_{L^2} \le C_\dl \|f\|_{\dot H^1}$ on $H^1_{0,\rad}(B_1 \backslash B_\dl)$, we obtain
\begin{align*}
\|\dot \Xi(t)\|_{\dot H^1} & \le 3 C_\dl \ld^2 \|X^+_\dl(t)\|_{L^2}, \\
\|\dot \U(t)\|_{\dot H^1} & \le 3 C_\dl^2 \ld^2 \big( \|\U(t)\|_{\dot H^1} + \|\Xi(t)\|_{\dot H^1} + \|\dot \Xi(t)\|_{\dot H^1} \big).
\end{align*}

\noi
Since $\|\U(t)\|_{\dot H^1} \le \int_0^t \|\dot \U(s)\|_{\dot H^1} ds$ and $\|\Xi(t)\|_{\dot H^1} \le \int_0^t \|\dot \Xi(s)\|_{\dot H^1} ds$, and since $t \mapsto \|\dot \U(t)\|_{\dot H^1}$ is bounded on $[0,1]$ by Lemma \ref{LEM:IDe}, Gronwall's inequality yields
\begin{align}
\label{XiU}
\sup_{0 \le t \le 1} \big( \|\dot \Xi(t)\|_{\dot H^1} + \|\dot \U(t)\|_{\dot H^1} \big) \le C_{\dl,\ld} \sup_{0 \le t \le 1} \|X^+_\dl(t)\|_{L^2},
\end{align}

\noi
whose right-hand side has finite moments of all orders by Doob's maximal inequality, $X^+_\dl(1)$ being Gaussian.
Since also $z_{\dl,\ld} - \ld h_\dl \in H^1_{0,\rad}(B_1 \backslash B_\dl)$, the relation \eqref{CoV} defines a bijection $\Theta \leftrightarrow W$ of $\mathbb H$ onto itself, which maps drifts of finite energy, namely with $\E \int_0^1 \|\dot \Theta(t)\|_{\dot H^1}^2 dt < \infty$, onto drifts of finite energy.
Drifts of infinite energy do not contribute to the infima in \eqref{eq:rhoVar0} (see Corollary \ref{COR:BD}), and we use the same convention in \eqref{eq:rhoVar}; hence, in what follows, we may assume that $W$ has finite energy.
Then, we plug \eqref{CoV} into the right-hand side of \eqref{eq:rhoVar0} term by term.
We first consider the first term in the right-hand side of \eqref{eq:rhoVar0} to get
\begin{align}
\label{CoV1}
\begin{split}
 \E & \bigg[ \frac14 \int_{B_1\backslash B_\dl} |X^+_\dl +  \ld h_\dl + \Theta |^4 dx + \frac{1}{2} \int_0^1 \| \dot \Theta (t) \|_{\dot H^1({B_1\backslash B_\dl})}^2 dt \bigg] \\
 & = \E \bigg[  \frac14 \int_{B_1\backslash B_\dl} |X^+_\dl  +  z_{\dl, \ld} + W  + \Xi  + \Upsilon  |^4 dx \\
 & \hphantom{XXX} + \frac{1}{2} \int_0^1 \| z_{\dl,\ld} - \ld h_\dl +  \dot W (t) + \dot \Xi  (t) + \dot \Upsilon (t) \|_{\dot H^1({B_1\backslash B_\dl})}^2 dt \bigg].
 \end{split}
 \end{align}

\noi
Here, as before, we  write $\Theta = \Theta (1)$, and similarly for $W$, $\Xi$, and $\U$.
We shall show that \eqref{CoV1} agrees with $\E\big[ P_{\dl,\ld}(W) \big]$ up to an additive constant which does not depend on $W$.
For this purpose, we expand the integrand in \eqref{CoV1} to get
 \begin{align*}
 \begin{split}
 & \frac 14 \int \big(W +\Xi  + \Upsilon  \big)^4 + \frac14 \int \big(z_{\dl,\ld} + X^+_\dl  \big)^4  + \int \big(W +\Xi  + \Upsilon  \big)^3 (z_{\dl,\ld} + X^+_\dl )  \\
& \phantom{=} + \frac32 \int \big(W +\Xi  + \Upsilon  \big)^2 (z_{\dl,\ld} + X^+_\dl )^2 + \int \big(W +\Xi  + \Upsilon  \big) (z_{\dl,\ld} + X^+_\dl )^3  \\
& \phantom{=} + \frac 12 \int_0^1 \Big( \| \dot W(t)\|_{\dot H^1_{\rad}}^2 +  \| \dot \Xi(t)\|_{\dot H^1_{\rad}}^2  +  \| \dot \Upsilon(t)\|_{\dot H^1_{\rad}}^2 + \| z_{\dl,\ld} - \ld h_\dl\|_{\dot H^1_{\rad}}^2\Big) dt \\
& \phantom{=} + \int_0^1 \jb{ z_{\dl,\ld} - \ld h_\dl, \dot W(t) + \dot \Xi (t) + \dot \Upsilon (t)}_{\dot H^1_{\rad}}  dt\\
& \hphantom{X} + \int_0^1 \big( \jb{\dot W(t), \dot \Xi (t)}_{\dot H^1_{\rad}} + \jb{\dot W(t), \dot \Upsilon (t)}_{\dot H^1_{\rad}} +  \jb{\dot \Upsilon(t), \dot \Xi (t)}_{\dot H^1_{\rad}} \big) dt .
\end{split}
\end{align*}

\noi
By recalling the definition of $P_{\dl,\ld} (W)$ in \eqref{P1} and by further expanding the above, we arrive at
\begin{align}
\begin{split}
\eqref{CoV1}  &  = \E \bigg[ \frac32 \int_{B_1 \backslash B_\dl} \big(\Xi  + \Upsilon  \big)^2 (z_{\dl,\ld} + X^+_\dl )^2 \\
& \hphantom{X} + \int_0^1 \jb{ z_{\dl,\ld} - \ld h_\dl,  \dot \Xi (t) + \dot \Upsilon (t)}_{\dot H^1_{\rad}}  dt \\
&  \phantom{=}   + \int_{B_1\backslash B_\dl} \big( \Xi  + \U\big)  \big(  z_{\dl,\ld}^3 + 3 X^+_\dl  z_{\dl,\ld}^2  \big)  +  \frac14 \int \big(z_{\dl,\ld} + X^+_\dl  \big)^4 \\
&  \phantom{=}  + \frac 12 \int_0^1 \Big(\| \dot \Xi(t)\|_{\dot H^1_{\rad}}^2  +  \| \dot \Upsilon(t)\|_{\dot H^1_{\rad}}^2 + \| z_{\dl,\ld} - \ld h_\dl\|_{\dot H^1_{\rad}}^2 \Big) dt + \int_0^1 \jb{\dot \Upsilon(t), \dot \Xi (t)}_{\dot H^1_{\rad}} dt \\
& \phantom{=} + P_{\dl,\ld} (W)  \\
& \phantom{=}  + \int_{B_1\backslash B_\dl} W    z_{\dl,\ld}^3 + \int_0^1 \jb{ z_{\dl,\ld} - \ld h_\dl, \dot W(t) }_{\dot H^1_{\rad}}  dt  \\
& \hphantom{X} + 3 \int_{B_1\backslash B_\dl} W    X^+_\dl  z_{\dl,\ld}^2 + \int_0^1 \jb{\dot W(t), \dot \Xi (t)}_{\dot H^1_{\rad}} dt \\
& \hphantom{X}  +  3 \int_{B_1 \backslash B_\dl} W  \big(\Xi  + \Upsilon  \big) z_{\dl,\ld}^2 + \int_0^1  \jb{\dot W(t), \dot \Upsilon (t)}_{\dot H^1_{\rad}} dt \bigg] .
\end{split}
\label{CoV10}
\end{align}

\noi
We note that the first four lines on the right-hand side of \eqref{CoV10} do not depend on $W$,  and, by the previous estimates, that their expectation $C_{\dl,\ld}$ is finite and depends only on the law of $X^+_\dl$.
Thus, the constant $C_{\dl,\ld}$ cancels out in \eqref{eq:rhoVar0}, and, to show \eqref{eq:rhoVar}, it suffices to show that the last three lines on the right-hand side of \eqref{CoV10} have zero expectation.
Namely, we have the following cancellations
\begin{align}
\label{CoV11}
 \E\bigg[ \int_{B_1\backslash B_\dl}  W  z_{\dl,\ld}^3 \, dx + \int_0^1 \jb{ z_{\dl,\ld} - \ld h_\dl, \dot W(t)}_{\dot H^1_{\rad}}  dt \bigg] = 0,
\end{align}

\noi
and
\begin{align}
\label{CoV12}
\E\bigg[ 3 \int_{B_1\backslash B_\dl}  W   z_{\dl,\ld}^2  X^+_\dl \, dx  +  \int_0^1  \jb{\dot W(t), \dot \Xi (t)}_{\dot H^1_{\rad}}   dt \bigg] = 0,
\end{align}

\noi
and
\begin{align}
\label{CoV13}
\E\bigg[ 3 \int_{B_1\backslash B_\dl}  W   \big(\Xi  + \Upsilon  \big) z_{\dl,\ld}^2 \, dx  +  \int_0^1  \jb{\dot W(t), \dot \Upsilon (t)}_{\dot H^1_{\rad}}  dt \bigg] = 0.
\end{align}

In what follows, we shall prove these three identities one by one.
We first consider \eqref{CoV11}.
By Stokes' theorem, and recalling that $\Delta h_\dl =0$ and \eqref{z}, we have that
\begin{align*}
& \E\bigg[ \int z_{\dl,\ld}^3 W  + \int_0^1 \jb{\dot W(s), z_{\dl,\ld} -\ld h_{\dl}}_{\dot H^1} ds\bigg] \\
& = \E\bigg[ \int z_{\dl,\ld}^3 W  - \int_0^1 \int_{B_1\backslash B_\dl } \dot W(s)  \Delta( z_{\dl,\ld} -\ld h_{\dl})ds\bigg] \\
& = \E\bigg[ \int z_{\dl,\ld}^3  W  - \int_{B_1\backslash B_\dl } W   z_{\dl,\ld}^3 \bigg]  = 0,
\end{align*}

\noi
which concludes \eqref{CoV11}.
Next, we turn to \eqref{CoV12}.
By the martingale property of $X^+_\dl(t)$, integration by parts, \eqref{Xidef}, and \eqref{Z}, we have
\begin{align*}
\E\bigg[ 3 \int &  W   z_{\dl,\ld}^2  X^+_\dl  +  \int_0^1  \jb{\dot W(t), \dot \Xi (t)}_{\dot H^1}   dt \bigg] \\
& = \E\bigg[ 3 \int_0^1 \int \dot W(s) X^+_\dl(s) z_{\dl,\ld}^2 \, dx \, ds  -  \int_0^1 \int \dot W(s) \Delta \dot \Xi(s) \, dx \, ds\bigg] \\
& = \E\bigg[ 3 \int_0^1 \int \dot W(s) X^+_\dl(s) z_{\dl,\ld}^2 \, dx \, ds  -  \int_0^1 \int \dot W(s) \, 3 X^+_\dl(s) z_{\dl,\ld}^2 \, dx \, ds\bigg] ,
\end{align*}

\noi
which vanishes, and we get \eqref{CoV12}. Finally, we consider \eqref{CoV13}.
By using the fundamental theorem of calculus in time variables, we have that
\begin{align*}
& 3\E\Big[ \int W \big(\U + \Xi \big)z_{\dl,\ld}^2 \Big] = 3\E\Big[ \int \int_0^1 \frac{d}{dt} \big( W (t) \big(\U(t) + \Xi (t) \big)z_{\dl,\ld}^2 \big) dt \Big] \\
&= 3\E\Big[ \int \int_0^1 \dot W(t)\big(\U(t) + \Xi(t)\big)z_{\dl,\ld}^2 dt + \int \int_0^1 W(t)\big(\dot \U(t) + \dot \Xi(t)\big)z_{\dl,\ld}^2 dt\Big] \\
&= 3\E\Big[ \int \int_0^1 \dot W(t)\big(\U(t) + \Xi(t)\big)z_{\dl,\ld}^2 dt + \int \int_0^1 \int_0^t \dot W(s)\big(\dot \U(t) + \dot \Xi(t)\big)z_{\dl,\ld}^2 ds dt\Big]. \\
\intertext{Here we recall that $\dot \Xi$ and $\dot \U$ are martingales. We may continue with}
&= 3\E\Big[ \int \int_0^1 \dot W(t)\big(\U(t) + \Xi(t)\big)z_{\dl,\ld}^2 dt + \int \int_0^1 \int_0^t \dot W(s)\big(\dot \U(s) + \dot \Xi(s)\big)z_{\dl,\ld}^2 ds dt\Big] \\
&= 3\E\Big[ \int \int_0^1 \dot W(t)\big(\U(t) + \Xi(t)\big)z_{\dl,\ld}^2 dt + \int \int_0^1 (1-t) \dot W(t)\big(\dot \U(t) + \dot \Xi(t)\big)z_{\dl,\ld}^2 dt \Big] \\
&= 3\E\Big[ \int \int_0^1 \dot W(t) \Big( \big(\U(t) + \Xi(t) + (1-t)\big(\dot \U(t) + \dot \Xi(t)\big) \big) z_{\dl,\ld}^2\Big) dt\Big]\\
& = - \E \int_0^1 \jb{\dot \Upsilon (t), \dot W(t)}_{\dot H^1} dt,
\end{align*}

\noi
where in the last step we used \eqref{Upsilon1}.

We get \eqref{eq:rhoVar} by collecting \eqref{CoV1}, \eqref{CoV10}, \eqref{CoV11}, \eqref{CoV12}, and \eqref{CoV13}. Thus, we complete the proof.
\end{proof}

\begin{remark}
\rm
\label{RMK:alm}
Such an expression \eqref{eq:rhoVar} allows us to estimate expectations of exponentials relatively easily. Indeed, if $W_1$ is an almost optimiser for the first infimum of \eqref{eq:rhoVar}, in the sense that
\begin{align}
\label{W1}
	\E\big[ P_{\dl,\ld}(W_1)\big]
	\le \ 1 + \inf_{W \in  \mathbb H}
	\E\big[ P_{\dl,\ld}(W)\big],
\end{align}

\noi
then we have
 $$
 \log \int \exp(-F(u)) d \rho_\dl^+[\ld](u) \ge - 1 -\E[F\big(X^+_\dl  +  z_{\dl,\ld} + W_1 + \Xi + \Upsilon  \big)].
 $$
Similarly, if $W_2$  is an almost optimiser for the second infimum of \eqref{eq:rhoVar}, in the sense that
\begin{align}
\label{W2}
\begin{split}
\E  \big[ & F(X^+_\dl + z_{\dl,\ld} + W_2 + \Xi + \U) + P_{\dl,\ld}(W_2) \big] \\
& \le 1 + \inf_{W \in  \mathbb H} \E\big[ F(X^+_\dl + z_{\dl,\ld} + W + \Xi + \U) + P_{\dl,\ld}(W) \big],
\end{split}
\end{align}

\noi
then it follows from \eqref{eq:rhoVar} that
\begin{equation}
 \log \int \exp(-F(u)) d \rho_\dl^+[\ld](u) \le 1 -\E[F (X^+_\dl +  z_{\dl,\ld} + W_2 + \Xi + \Upsilon )], \label{varub}
\end{equation}
where in the above $F$ is given as in Proposition \ref{PROP:rho}.
\end{remark}

\subsection{A priori estimates} \label{sub:APriori}

To make use of the formula \eqref{eq:rhoVar},
we need to control  $z_{\dl,\ld}$, $\Xi$, and $\U$ involved in the change of variables \eqref{CoV}.
The next two lemmas will play a central role in both the construction and the non-triviality statements for the limiting measure.

\begin{lemma}[A priori estimates for $z_{\dl,\ld}$]
\label{LEM:z}
Let $z$ be a radial solution of the equation
\begin{equation}\label{zeq}
\begin{cases}
\Delta z (x) = z^3 (x), & x \in B_1\backslash B_\dl, \\
z(x) = 0, & x \in \partial B_1,
\end{cases}
\end{equation}
which is $C^1$ up to the boundary of $B_1\backslash B_\dl$. Then, for  $\dl < r \le 1$,
\begin{equation}\label{zestimate}
|z(r)| \le \frac{\sqrt2}{r-\dl}.
\end{equation}
\end{lemma}

\begin{remark}\rm
Note that Lemma \ref{LEM:z}  manifests the defocusing structure of the equation \eqref{zeq},
which cannot hold in general for other semi-linear elliptic equations.
\end{remark}

\begin{proof}
If $z\equiv 0$ in $B_1\backslash B_\dl$, \eqref{zestimate} obviously follows. Therefore, without loss of generality, we can assume $z \not\equiv 0$. Moreover, since the equation \eqref{zeq} is invariant under the transformation $z \mapsto -z$, we can further assume that  $z(r_0) > 0$ for some  $\dl \le r_0 < 1$.

Since $z$ is radial, by expressing the Laplacian in spherical coordinates, we see that $z$ satisfies the equation
\begin{align}
    \label{r_z}
    \partial_r^2 z + \frac {d-1} r \partial_r z = z^3
\end{align}
for $\dl \le r \le 1$. Since $z \not\equiv 0$, we deduce that $\partial_r z(1) \neq 0$, as otherwise $z(1) = 0$ and $\partial_r z(1) = 0$ would force $z \equiv 0$ by uniqueness for the ODE \eqref{r_z}. By multiplying both sides of this equation by $\partial_r z$, we obtain that
$$ \partial_r \Big(\frac 12 (\partial_r z)^2 - \frac 14 z^4\Big) = -\frac {d-1}r (\partial_r z)^2 \le 0.  $$
Integrating this between $r$ and $1$, and recalling that $z(1) =0$, we deduce that
\begin{equation*}
\frac 12 (\partial_r z)^2(1) \le \frac 12 (\partial_r z)^2(r)   - \frac 14 z^4(r),
\end{equation*}
or equivalently,
\begin{equation}\label{elem}
\big|\partial_r z (r)\big|  \ge \sqrt{(\partial_r z)^2(1) + \frac 12 z^4(r)} \ge \big|\partial_r z (1)\big| > 0.
\end{equation}

\noi
Since $z(r_0) > 0$, by the mean value theorem, one must have that $\partial_r z (r_1) < 0$ for some $\dl < r_1 < 1$. Together with \eqref{elem}, this implies that $\partial_r z(r) < 0$ and $z(r) >0$ for every $\dl \le r <1$. Therefore, from the first inequality in \eqref{elem}, we obtain that
\begin{equation}\label{elem2}
-\frac{\partial_r z (r)}{\sqrt{(\partial_r z)^2(1) + \frac 12 z^4(r)}} \ge 1.
\end{equation}
Integrating both sides of \eqref{elem2} between $\dl$ and $r$, we obtain
\begin{align*}
r - \dl & \le \int_{\dl}^{r} -\frac{z' (s)}{\sqrt{(z' (1))^2 + \frac 12 z^4(s)}} ds \\
& \le \int_{\dl}^{r} -\frac{ \sqrt 2  z' (s)}{ z^2(s)} ds \\
& = \sqrt 2 \bigg(\frac1{z(r)} - \frac1{z(\dl)} \bigg)   \le \frac{\sqrt2}{z(r)},
\end{align*}
from which we obtain
$$ z(r) \le \frac{\sqrt2}{r-\dl}.$$
Since the case $z \le 0$ follows by applying the above to $-z$, we obtain \eqref{zestimate}.
Thus, we finish the proof.
\end{proof}

\begin{remark}
\rm
Later, it will be clear that the singularity of \eqref{zestimate} at $|x| = \dl$ is the main source of difficulty. We will manage to balance this singularity by the decay at $\dl$ of $\Xi$ and $\U$, which will be discussed in the following lemmas.
\end{remark}

\begin{lemma}[A priori estimates for $\Xi$]
\label{LEM:Xi}
Let $\Xi$ be as in \eqref{Xidef}.
Then
\begin{align}
\label{Xiestimate}
\sup_{0 \le s \le 1} \E\big[|\dot \Xi(s,r) |^2\big] \les \frac{r-\dl}{r^{d-1}},
\end{align}
uniformly in $\dl,\ld$.
As a consequence, we have
\begin{align}
\label{Xiestimate1}
\sup_{0 \le t \le 1} \E\big[|\Xi(t,r) |^2\big] \les \frac{r-\dl}{r^{d-1}}.
\end{align}
Here, the implicit constant is independent of $\dl \in (0,\frac12]$.
\end{lemma}
\begin{proof}
It suffices to show the estimate for $\dot \Xi(s)$, uniformly in $0\le s\le 1$ and $0 < \dl \le \frac12$;
the bound \eqref{Xiestimate1} then follows from \eqref{Xidef} and Minkowski's integral inequality.
From \eqref{Z} and \eqref{Xidef}, and since $K_\dl$ is the kernel of $(-\Delta)^{-1}$ by \eqref{r_green3}, we have that
$$\dot \Xi(s,r) = - 3 \int_\dl^1 K_\dl(r,y) X^+_\dl(s,y) z_{\dl,\ld}^2(y) \, y^{d-1} dy,
$$

\noi
where $K_\dl$ is the radial Green function defined in \eqref{kernel0}.
Recall that
$$ \E[X^+_\dl(s,y) X^+_\dl(s,y')] = s K_\dl(y,y'). $$

\noi
Then, we obtain that
\begin{align}
\label{Xi0}
\begin{split}
\E [  &  |\dot {\Xi} (s,r)|^2] = 9\, \E \bigg[ \Big| \int_\dl^1 K_\dl(r,y) X^+_\dl(s,y) z_{\dl,\ld}^2(y) \, y^{d-1} dy \Big|^2 \bigg] \\
& = 9 \int_\dl^1 \int_\dl^1 K_\dl(r,y)   K_\dl(r,y') \E [X^+_\dl(s,y) X^+_\dl(s,y')] z_{\dl,\ld}^2(y) z_{\dl,\ld}^2(y') (y y')^{d-1} dy dy' \\
& = 9 s \int_\dl^1 \int_\dl^1 K_\dl(r,y)   K_\dl(r,y') K_\dl (y,y') z_{\dl,\ld}^2(y) z_{\dl,\ld}^2(y') (y y')^{d-1} dy dy' \\
& = 9 s \int_\dl^1 \int_\dl^1  z_{\dl,\ld}^2 (u) z_{\dl,\ld}^2 (u') K_\dl(r,u)   K_\dl(r,u') K_\dl (u,u')   u^{d-1} (u')^{d-1} du du'\\
& = 18 s \intt_{\dl \le u' \le u \le 1} z_{\dl,\ld}^2 (u) z_{\dl,\ld}^2 (u') K_\dl(r,u)   K_\dl(r,u') K_\dl (u,u')  u^{d-1} (u')^{d-1} du du'.
\end{split}
\end{align}

\noi
Then by \eqref{zestimate} and \eqref{KE}, the above is bounded, up to a multiplicative constant, by
\begin{align}
\label{Xi1}
\begin{split}
\intt_{\dl \le u' \le u \le 1} \frac{((r \wedge u)^{d-2} - \dl^{d-2}) ((u \wedge u')^{d-2} - \dl^{d-2}) ((r \wedge u')^{d-2} - \dl^{d-2}) (uu')^{d-1}}{(ruu')^{2(d-2)} (u-\dl)^2(u'-\dl)^2} dudu' .
\end{split}
\end{align}

\noi
In what follows, we estimate \eqref{Xi1} by splitting the integral region into several cases.

{\it Case 1}: the integral of \eqref{Xi1} over $\dl \le u' \le u \le r \le 1$.
Then, we have
\begin{align}
\label{Xi11}
\begin{split}
\eqref{Xi1} & \les \intt_{\dl \le u' \le u \le r} \frac{( u^{d-2} - \dl^{d-2}) (( u')^{d-2} - \dl^{d-2}) (( u')^{d-2} - \dl^{d-2}) (uu')^{d-1}}{(ruu')^{2(d-2)} (u-\dl)^2(u'-\dl)^2} dudu' \\
& \les \intt_{\dl \le u' \le u \le r} \frac{( u  - \dl) u^{d-3} ( u' - \dl )^2 ( u')^{2d-6}  (uu')^{d-1}}{(ruu')^{2(d-2)} (u-\dl)^2(u'-\dl)^2} dudu' \\
& \les \frac1{r^{2d-4}} \intt_{\dl \le u' \le u \le r} \frac{(u')^{d-3}}{u-\dl}  dudu'\les  \frac1{r^{d-1}} \intt_{\dl \le u' \le u \le r} \frac{ 1}{u-\dl}  dudu'\\
& = \frac{r-\dl}{r^{d-1}},
\end{split}
\end{align}
which implies \eqref{Xiestimate} for this case.

{\it Case 2}: the integral of \eqref{Xi1} over $\dl \le u' \le r \le u \le 1$.
Then, we have
\begin{align}
\label{Xi12}
\begin{split}
\eqref{Xi1} & \les \intt_{\dl \le u' \le r \le u \le 1} \frac{( r^{d-2} - \dl^{d-2}) (( u')^{d-2} - \dl^{d-2}) (( u')^{d-2} - \dl^{d-2}) (uu')^{d-1}}{(ruu')^{2(d-2)} (u-\dl)^2(u'-\dl)^2} dudu' \\
& \les \intt_{\dl \le u' \le r \le u \le 1} \frac{(r  - \dl) r^{d-3} ( u' - \dl )^2 ( u')^{2d-6}  (uu')^{d-1}}{(ruu')^{2(d-2)} (u-\dl)^2(u'-\dl)^2} dudu' \\
& \les \frac{r-\dl}{r^{d-1}} \intt_{\dl \le u' \le r \le u \le 1} \frac{1}{(u-\dl)^2}  dudu' =  \frac{(r-\dl)^2}{r^{d-1}} \int_{r}^1 \frac{ 1}{(u-\dl)^2}  du\\
& = \frac{(r-\dl)^2}{r^{d-1}} \bigg( \frac1{r-\dl} - \frac1{1-\dl}\bigg),
\end{split}
\end{align}
which is sufficient for \eqref{Xiestimate}.

{\it Case 3}: the integral of \eqref{Xi1} over $\dl \le r \le u' \le u \le 1$.
Then, we have
\begin{align}
\label{Xi13}
\begin{split}
\eqref{Xi1} & \les \intt_{\dl \le  r \le u' \le u \le 1} \frac{( r^{d-2} - \dl^{d-2}) (r^{d-2} - \dl^{d-2}) (( u')^{d-2} - \dl^{d-2}) (uu')^{d-1}}{(ruu')^{2(d-2)} (u-\dl)^2(u'-\dl)^2} dudu' \\
& \les \intt_{\dl \le  r \le u' \le u \le 1} \frac{(r  - \dl)^2 r^{2d-6} ( u' - \dl ) ( u')^{d-3}  (uu')^{d-1}}{(ruu')^{2(d-2)} (u-\dl)^2(u'-\dl)^2} dudu' \\
& = \frac{(r  - \dl)^2}{r^{2}} \intt_{\dl \le  r \le u' \le u \le 1} \frac{1}{u^{d-3} (u-\dl)^2 (u'-\dl)}  dudu' \\
& \les  \frac{(r-\dl)^2}{r^{d-1}} \intt_{\dl \le  r \le u' \le u \le 1} \frac{1}{(u-\dl)^2 (u'-\dl)}  du du'\\
& = \frac{(r-\dl)^2}{r^{d-1}} \int_{r}^1 \frac{1}{(u'-\dl)}  \bigg( \frac1{u' - \dl} - \frac1{1-\dl}\bigg) du'\\
& = \frac{(r-\dl)^2}{r^{d-1}} \bigg( \frac1{r-\dl} - \frac1{1-\dl} - \frac{\ln (1-\dl) - \ln (r - \dl)}{1-\dl} \bigg),
\end{split}
\end{align}
which is again sufficient for our purpose since $\dl \in (0, \frac12]$.
\end{proof}

Finally, we also need an a priori estimate for $\U$ given in \eqref{Upsilon}.

\begin{lemma}[A priori estimates for $\Upsilon_{\dl,\ld}$]
\label{LEM:Upsilon}
Let $\Upsilon$ be as in \eqref{Upsilon}.
Then, for $\dl < r \le 1$, we have
\begin{align}
\sup_{0 \le t \le 1} \E\big[|\U(t,r) |^2\big] & \les \frac{r-\dl}{r^{d-1}}, \label{Uestimate} \\
\sup_{0 \le t \le 1} \E\big[|\dot \U(t,r) |^2\big] & \les \frac{r-\dl}{r^{d-1}}, \label{dUestimate}
\end{align}
uniformly in $\dl \in (0,\frac12]$ and $\ld$.
\end{lemma}

\begin{proof}
We first note that for $t\in [0,1]$ we have $3 (1-t) z_{\dl,\ld}^2 \ge 0$, which implies the operator inequality $- \Delta \le - \Delta + 3 (1-t) z_{\dl,\ld}^2$.
Therefore, by \eqref{Upsilon} and the comparison principle of Lemma \ref{LEM:CP} with $V_1 = 3(1-t)z_{\dl,\ld}^2 \ge 0$ and $V_2 = 0$, we have the following point-wise estimate
\begin{align}
\label{UE1}
\begin{split}
|\dot \Upsilon (t,x)| & = \Big| \big( 3(1-t)z_{\dl,\ld}^2 - \Delta\big)^{-1} \Big(3\big((1-t)\dot \Xi (t) +  \Xi(t) +\U(t)\big) z_{\dl,\ld}^2\Big) (x) \Big| \\
& \le \big( 3(1-t)z_{\dl,\ld}^2 - \Delta\big)^{-1} 3 \big|\big((1-t)\dot \Xi (t) +  \Xi(t) +\U(t)\big) z_{\dl,\ld}^2\big| (x) \\
& \le ( - \Delta )^{-1} 3 \big|\big((1-t)\dot \Xi (t) +  \Xi(t) +\U(t)\big) z_{\dl,\ld}^2\big| (x)  \\
& \le 3 \int_{B_1 \backslash B_\dl} G_\dl (x,y) \big| \big((1-t)\dot \Xi (t) +  \Xi(t) +\U(t)\big) z_{\dl,\ld}^2 (y) \big| dy,
\end{split}
\end{align}
where the kernel $G_\dl$ is the Green function of $-\Delta$ defined in \eqref{Green_dl}, whose radial version is $K_\dl$.
Then from \eqref{KE}, \eqref{zestimate}, \eqref{Xiestimate}, \eqref{Xiestimate1}, and \eqref{UE1}, together with Minkowski's integral inequality, we have
\begin{align}
\label{UE2}
\begin{split}
\Big( \E \big[ |\dot \Upsilon (t, r)|^2 \big] \Big)^{\frac12}
& \les \int_{\dl}^1 K_\dl (r,\rho) \big( \|\dot \Xi (t,\rho) \|_{L^2 (\O)} + \|\Xi(t,\rho)\|_{L^2 (\O)} \big)z_{\dl,\ld}^2 (\rho) \rho^{d-1} d\rho  \\
& \hphantom{XXXX} + \int_{\dl}^1 K_\dl (r, \rho)  \|\U(t, \rho)\|_{L^2 (\O)}  z_{\dl,\ld}^2 (\rho) \rho^{d-1} d\rho \\
& \les \int_{\dl}^1 \frac{ (r \wedge \rho)^{d-2} - \dl^{d-2}}{r^{d-2} \rho^{d-2}}  \sqrt{\frac{\rho- \dl}{\rho^{d-1}}} \bigg( \frac1{\rho - \dl}\bigg)^2 \rho^{d-1} d\rho \\
& \hphantom{XXXX} + \int_{\dl}^1 \frac{ (r \wedge \rho)^{d-2} - \dl^{d-2}}{r^{d-2} \rho^{d-2}} \frac{\rho^{d-1}}{(\rho - \dl)^2} \|\U(t,\rho)\|_{L^2 (\O)}   d \rho,
\end{split}
\end{align}

\noi
where $r = |x|$.
To proceed, we first estimate the first term on the  right-hand side in \eqref{UE2}.
 Using $(r \wedge \rho)^{d-2} - \dl^{d-2} \le (d-2) (r \wedge \rho)^{d-3} (r \wedge \rho - \dl)$, we have that
\begin{align}
\label{UGE}
\begin{split}
  \int_{\dl}^1 & \frac{ (r \wedge \rho)^{d-2} - \dl^{d-2}}{r^{d-2} \rho^{d-2}} \sqrt{\frac{\rho- \dl}{\rho^{d-1}}} \bigg( \frac1{\rho - \dl}\bigg)^2 \rho^{d-1} d\rho \\
  & \les \int_\dl^1 \frac{(r \wedge \rho)^{d-3} (r\wedge \rho - \dl) }{(\rho-\dl)^\frac32 r^{d-2} \rho^{\frac{d}2 - \frac32}} d\rho \\
  &\les \frac1 {r^{\frac{d}2 - \frac12}} \int_\dl^r \frac{d\rho}{(\rho-\dl)^\frac12} + \frac{r-\dl}{r^{\frac{d}2 - \frac12}} \int_r^1 \frac{d\rho}{ (\rho-\dl)^\frac32} \\
  & \les r^{- \frac{d}2 + \frac12} (r-\dl)^\frac12 + r^{- \frac{d}2 + \frac12} (r-\dl) ((r-\dl)^{-\frac12} - (1- \dl)^{-\frac12}) \\
  & \les r^{- \frac{d}2 + \frac12} (r-\dl)^\frac12 .
\end{split}
\end{align}

\noi
In the above computation,
we note that the singularity from $z_{\dl,\ld}$ is balanced by the decay (near $\dl$) of  $\Xi$ and the Green function $K_\dl$.
Then we obtain from \eqref{UE2} and \eqref{UGE} that
\begin{align}
\label{UE3}
\begin{split}
\| & \Upsilon (t, r)\|_{L^2(\O)} \\
& \les \int_0^t \bigg( r^{- \frac{d}2 + \frac12} (r-\dl)^\frac12 + \int_{\dl}^1 \frac{ (r \wedge \rho)^{d-2} - \dl^{d-2}}{r^{d-2} \rho^{d-2}} \frac{\rho^{d-1}}{(\rho - \dl)^2} \|\U(s,\rho)\|_{L^2 (\O)}   d \rho  \bigg) ds \\
& \les \int_0^t \bigg( r^{- \frac{d}2 + \frac12} (r-\dl)^\frac12 +  \int_\dl^1 \frac{ (r \wedge \rho)^{d-3} (r \wedge \rho - \dl)}{r^{d-2} \rho^{\frac{d}2 - \frac32} (\rho - \dl)^{\frac32}} \frac{\rho^{\frac{d}2-\frac12}}{\sqrt{\rho - \dl}} \| \U(s,\rho)\|_{L^2 (\O)}   d \rho \bigg) ds.
\end{split}
\end{align}

\noi
For $0 \le t \le 1$ and $\dl < r \le 1$, we set
\begin{align*}
 H(t) & := \sup_{\dl < r \le 1}  \frac{r^{\frac{d}2 - \frac12}}{\sqrt{r-\dl}} \| \Upsilon (t, r)\|_{L^2(\O)}, \\
 M_\dl (r) & := \frac{r^{\frac{d}2 - \frac12}}{\sqrt{r-\dl}} \int_\dl^1 \frac{ (r \wedge \rho)^{d-3} (r \wedge \rho - \dl)}{r^{d-2} \rho^{\frac{d}2 - \frac32} (\rho - \dl)^{\frac32}} d\rho.
\end{align*}

\noi
Since $\| \U(s,\rho)\|_{L^2 (\O)} \le \rho^{-\frac{d}2 + \frac12} \sqrt{\rho - \dl} \, H(s)$, the bound \eqref{UE3} yields, for every $\dl < r \le 1$,
\begin{align}
\label{UE4}
\frac{r^{\frac{d}2 - \frac12}}{\sqrt{r-\dl}} \| \Upsilon (t, r)\|_{L^2(\O)} \le C_0 \bigg( t + M_\dl (r) \int_0^t  H (s) ds \bigg),
\end{align}

\noi
where the constant $C_0 > 0$ is independent of $\dl \in (0,\frac12]$, $\ld$, $t$, and $r$.
Moreover, for every $\dl < r \le 1$, we have
\[
\begin{split}
 M_\dl (r) & \le \frac{1}{\sqrt{r-\dl}} \int_\dl^r \frac{1}{ \sqrt{\rho - \dl}} d\rho + r^{\frac{d}2 - \frac32} \sqrt{r-\dl} \int_r^1 \frac{1}{ \rho^{\frac{d}2 - \frac32} (\rho - \dl)^{\frac32}} d\rho\\
 & \le 2 +  \sqrt{r-\dl} \int_r^1 \frac{1}{ (\rho - \dl)^{\frac32}} d\rho \le 4.
\end{split}
\]

\noi
Hence, taking the supremum over $\dl < r \le 1$ in \eqref{UE4}, we obtain
\begin{align}
\label{UE40}
H(t) \le C_1 \bigg( t + \int_0^t H(s) ds \bigg), \qquad 0 \le t \le 1,
\end{align}

\noi
with $C_1 = 4 C_0$.

If $H$ were known to be bounded on $[0,1]$, Gronwall's inequality would now give $H \les 1$. To see that this is the case, we note that the map $\G$ defined in \eqref{Gamma_fixed_point}, with $V(t) = 3(1-t) z_{\dl,\ld}^2$, $a = -3 z_{\dl,\ld}^2$, and $b(t) = -3 \big((1-t) \dot \Xi (t) + \Xi(t) \big) z_{\dl,\ld}^2$, preserves the bound $H(t) \le C_1 t e^{C_1 t}$. This follows from the argument leading to \eqref{UE40}, together with the inequality $C_1 t + C_1^2 \int_0^t s e^{C_1 s} ds \le C_1 t e^{C_1 t}$. In particular, since the set of functions satisfying this bound is closed, also the fixed point $\Upsilon$ satisfies this bound, which shows that $H$ is bounded. Therefore, we obtain \eqref{Uestimate}.
Finally, inserting \eqref{Uestimate} into the second term on the right-hand side of \eqref{UE2}, and using \eqref{UGE} and $M_\dl (r) \le 4$, we obtain \eqref{dUestimate}.
\end{proof}

Since $z_{\dl,\ld}$ is deterministic, $\dot \Xi(t) = Z_{\dl,\ld}(t)$ and $\Xi(t)$ depend linearly on the Gaussian process $X^+_\dl$, by \eqref{Z} and \eqref{Xidef}, and so do $\U (t)$ and $\dot \U (t)$, since this property is preserved by the map $\G$ in \eqref{Gamma_fixed_point} and is stable under almost sure limits.
Since almost sure limits of centred Gaussian random variables are centred Gaussian, it follows that, for every $t \in [0,1]$ and $\dl < r \le 1$, the random variables $\Xi(t,r)$, $\dot \Xi(t,r)$, $\U(t,r)$, and $\dot \U(t,r)$, as well as all their finite linear combinations, are centred Gaussian.
In particular, each such random variable $Y$ satisfies $\E|Y|^4 = 3 (\E|Y|^2)^2$ and, by hypercontractivity, $(\E|Y|^p)^{\frac1p} \le C_p (\E|Y|^2)^{\frac12}$ for every $p \ge 2$; these facts are used in Lemmas \ref{LEM:L4-2} and \ref{lem:xizC12} below and in Subsection \ref{sub:optimisers}.

By a similar argument as in the proof of Corollary \ref{COR:L4est}, the bounds \eqref{Xiestimate1} and \eqref{Uestimate} imply the following.

\begin{lemma}
\label{LEM:L4-2}
Let $d \ge 4$ and $\dl \in (0,\frac12]$. Then we have
\begin{align}
\label{eq:UX}
\E\big[ \| \U (1)\|_{L^4 (B_1 \backslash B_\dl)}^4 + \| \Xi (1)\|_{L^4 (B_1 \backslash B_\dl)}^4  \big] \les \Ld (\dl),
\end{align}

\noi
where $\Ld (\cdot)$ is defined in \eqref{Ld}, and the implicit constant is independent of $\dl \in (0,\frac12]$ and $\ld \in \R$.
\end{lemma}

Let $E \subset \mathbb R^d$.
The H\"older space $C^\al (E)$ for $\al \in (0,1)$ is given by the norm
 \begin{align} \label{holder}
 \|u\|_{C^{\al}(E)} = \|u\|_{L^\infty (E)} + \sup_{\substack{x,y \in E \\ x \neq y}} \frac{|u(x) - u(y)|}{|x-y|^{\al}}.
 \end{align}

\noi
 We recall the Sobolev inequality from \cite[Theorem 4.12, Part I\!I]{Adams03}:
Given a strong local Lipschitz domain $E$ and $\al \in (0,1)$, there exists $C_E > 0$ such that
\begin{align}\label{SI}
\|f\|_{C^{\al}(E)} \le C_E \|f \|_{W^{2,d} (E)} ,
\end{align}

\noi
for any $f \in W^{2,d} (E)$.

\medskip

In constructing the $\Phi^4$ measure away from the origin, we also need the following estimates away from the small ball $B_\dl$.

\begin{lemma}\label{lem:xizC12}
Let $\dl \in (0,\frac12]$, and let $z_{\dl,\ld}$, $\Xi$, and $\U$ be defined in \eqref{z}, \eqref{Xidef}, and \eqref{Upsilon}, respectively.
Then, for every $\al \in (0,1)$ and every $p \ge 1$, there exists a constant $C_\dl = C_\dl(\al,p)$ such that
    \begin{align}
    \label{u_est}
    \begin{split}
        \|z_{\dl,\ld}\|_{C^{\al}(B_1\backslash B_{2\dl})} +
        \Big( \E\Big[ \|\Xi(1)\|^p_{C^{\al}(B_1\backslash B_{2\dl})} \Big] \Big)^{\frac1p}  +
        \Big( \E \Big[ \|\U(1)\|^p_{C^{\al}(B_1\backslash B_{2\dl})} \Big] \Big)^{\frac1p} &\le C_\dl,
    \end{split}
    \end{align}
 uniformly in $\ld \in \R$.
\end{lemma}
\begin{proof}
It is easy to see that the domain $B_1\backslash B_{2\dl}$ satisfies the strong local Lipschitz condition with the Lipschitz constant depending on $\dl$.
Therefore, we may apply \eqref{SI} to get
\[
\|z_{\dl,\ld}\|_{C^{\al}(B_1\backslash B_{2\dl})} \le C_\dl \|z_{\dl,\ld} \|_{W^{2,d} (B_1\backslash B_{2 \dl})},
\]

\noi
where $C_\dl$ is a constant that only depends on $d$, $\dl$, $\al$, and $p$ (but not on $\ld$), which may differ from line to line.
Moreover, for every radial $u \in C^2 (\cj{B_1\backslash B_{2\dl}})$, we have
\begin{align}
\label{W2d}
\|u\|_{W^{2,d} (B_1\backslash B_{2 \dl})} \le C_\dl \big( \|\Delta u\|_{L^d (B_1\backslash B_{2 \dl})} + \|u\|_{L^d (B_1\backslash B_{2 \dl})} \big).
\end{align}
Note that we do not have any boundary condition at $r = 2\dl$, so we cannot drop the term $\|u\|_{L^d (B_1\backslash B_{2 \dl})}$.
Using \eqref{W2d}, \eqref{z}, and Lemma \ref{LEM:z}, which gives $|z_{\dl,\ld}| \le \sqrt 2 \dl^{-1}$ on $B_1\backslash B_{2\dl}$, we obtain
$$
\|z_{\dl,\ld} \|_{W^{2,d} (B_1\backslash B_{2 \dl})} \le C_\dl \big( \| z_{\dl,\ld}^3 \|_{L^\infty(B_1\backslash B_{2 \dl})} + \|z_{\dl,\ld} \|_{L^\infty(B_1\backslash B_{2 \dl})} \big) \le C_\dl,
$$

\noi
which gives the first inequality in \eqref{u_est}.
Similarly, by the same Sobolev inequality \eqref{SI}, \eqref{W2d}, \eqref{Xidef}, and \eqref{Z}, we have, for $0 \le s \le 1$,
    \[
    \begin{split}
    \|\dot \Xi (s)\| _{C^{\al}(B_1\backslash B_{2\dl})} & \le C_\dl \big( \|\Delta \dot \Xi (s)\|_{L^{d}(B_1\backslash B_{2 \dl})} + \|\dot \Xi (s)\|_{L^{d}(B_1\backslash B_{2 \dl})} \big) \\
    & \le C_\dl \big( \| X^+_\dl(s) \|_{L^{d}(B_1\backslash B_{ 2 \dl})} \|z_{\dl,\ld}\|_{L^\infty(B_1\backslash B_{ 2 \dl})}^2 + \|\dot \Xi (s)\|_{L^{d}(B_1\backslash B_{2 \dl})} \big) \\
    & \le C_\dl \big( \| X^+_\dl(s) \|_{L^{d}(B_1\backslash B_{ 2 \dl})} + \|\dot \Xi (s)\|_{L^{d}(B_1\backslash B_{2 \dl})} \big).
    \end{split}
    \]

    \noi
    Recall from the discussion before Lemma \ref{LEM:L4-2} that $X^+_\dl (s,r)$ and $\dot \Xi (s,r)$ are centred Gaussian.
    Then, since $\Xi(1) = \int_0^1 \dot \Xi (s) ds$, Minkowski's integral inequality and hypercontractivity, together with \eqref{X+estimate} and \eqref{Xiestimate}, give, for $p \ge d$,
    \begin{align*}
    \begin{split}
    \Big( \E \Big[ \|\Xi(1)\|^p_{C^{\al}(B_1\backslash B_{2\dl})} \Big] \Big)^{\frac1p} & \le C_\dl \sup_{0 \le s \le 1} \Big\| \| X^+_\dl(s,r) \|_{L^p (\O)} + \| \dot \Xi (s,r) \|_{L^p (\O)} \Big\|_{L^d_r (B_1\backslash B_{2\dl})} \\
    & \le C_\dl \sup_{0 \le s \le 1} \sup_{2 \dl \le r \le 1} \Big( \big( \E [ |X^+_\dl(s,r)|^2 ]\big)^{\frac12} + \big( \E [ |\dot \Xi(s,r)|^2 ]\big)^{\frac12} \Big) \le C_\dl,
    \end{split}
    \end{align*}

    \noi
    which gives the second inequality in \eqref{u_est} for $p \ge d$, and hence for every $p \ge 1$ by H\"older's inequality.
    The proof for $\U$ in \eqref{u_est} is analogous.
\end{proof}

\subsection{Estimates for the drifts} \label{sub:optimisers}
Now, we are ready to estimate the functional $P_{\dl,\ld}$ defined in \eqref{P1},
which plays a crucial role in the proof of the main theorem in the following section.
We start with the lower bound.

\begin{proposition}
\label{PROP:opt}
Let $d \ge 4$, $\dl \in (0,\frac12]$, and $\ld \in \R$, and let $P_{\dl,\ld}$ be as in \eqref{P1}.
Then, for every $W \in \mathbb H$, we have
\begin{equation}\label{optestimate}
\E[P_{\dl,\ld}(W)] \ge c_1 \E\bigg[ \int_{B_1 \backslash B_\dl} W(1)^4 \, dx + \| W(1) \|_{\dot H^1_{\rad}}^2\bigg] - C \Ld (\dl),
\end{equation}
where  $c_1, C>0$ are universal constants (in particular, independent of $\dl$ and $\ld$), and $\Ld (\cdot)$ is defined in \eqref{Ld}.
\end{proposition}

\begin{proof}
We consider $P_{\dl, \ld} (W)$ in \eqref{P1} term by term; as in \eqref{P1}, $W$, $\Xi$, and $\U$ stand for $W(1)$, $\Xi(1)$, and $\U(1)$ when the time variable is not displayed.
For the first term in \eqref{P1}, we note that, since $W(0) = 0$, the Cauchy--Schwarz inequality in time gives
\[
\| W(1) \|_{\dot H^1_{\rad}}^2 = \bigg\| \int_0^1 \dot W(s) ds \bigg\|_{\dot H^1_{\rad}}^2 \le \int_0^1 \| \dot W(s)\|_{\dot H^1_{\rad}}^2 ds.
\]

\noi
Together with Young's inequality and \eqref{eq:UX}, this yields
\begin{align}
\begin{split}
\E \bigg[ \frac 14 \int_{B_1 \backslash B_\dl} & \big (W + \U + \Xi  \big)^4 dx + \frac 12 \int_0^1 \| \dot W(s)\|_{\dot H^1_{\rad}}^2 ds \bigg] \\
& \ge \E \bigg[ \frac 9{40} \int_{B_1 \backslash B_\dl} W^4 dx + \frac 12\| W \|_{\dot H^1_{\rad}}^2 \bigg] - C \, \E \int_{B_1 \backslash B_\dl} \big(\U +\Xi \big)^4 dx \\
& \ge \E \bigg[ \frac 9{40} \int_{B_1 \backslash B_\dl} W^4 dx + \frac 12\| W \|_{\dot H^1_{\rad}}^2 \bigg] - C \Ld(\dl),
\end{split}
\label{opt}
\end{align}

\noi
where $\Ld(\cdot)$ is given by \eqref{Ld},  and the constant $C$ may differ from line to line.

We turn to the second term in \eqref{P1}.
By Young's inequality and $2ab \le a^2 + b^2$, we have
\begin{align}
\begin{split}
 & \bigg| \int_{B_1\backslash B_\dl} \big(W +\Xi  + \U \big)^3 (z_{\dl,\ld} + X^+_\dl ) \, dx \bigg| \\
& \hphantom{X} \le \sqrt{\frac 65} \int_{B_1 \backslash B_\dl} |W |^3 |z_{\dl,\ld}+X^+_\dl |dx + C \int_{B_1 \backslash B_\dl} |\Xi + \U|^3 |z_{\dl,\ld}+X^+_\dl | dx \\
& \hphantom{X} \le \frac32 \int_{B_1 \backslash B_\dl} |W |^2 |z_{\dl,\ld}+X^+_\dl |^2 dx + \frac{1}{5} \int_{B_1 \backslash B_\dl} W^4 dx+ C \int_{B_1 \backslash B_\dl} |\Xi + \U |^3 |z_{\dl,\ld}+X^+_\dl | dx .
\end{split}
\label{opt0}
\end{align}

\noi
We need to estimate the third term on the right-hand side of \eqref{opt0}.
By H\"older inequality and hypercontractivity, we have
\begin{align}
\begin{split}
\E \bigg[ & \int_{B_1 \backslash B_\dl} |\Xi + \U|^3 |z_{\dl,\ld}+X^+_\dl | \bigg] \\
& \les \E \bigg[\int_{B_1 \backslash B_\dl} |\Xi + \U|^3 |z_{\dl,\ld}| \bigg] + \E \bigg[\int_{B_1 \backslash B_\dl} |\Xi + \U|^3 |X^+_\dl | \bigg] \\
& \les \int_{B_1 \backslash B_\dl} \big( \E [|\Xi + \U|^2] \big)^{\frac32} |z_{\dl,\ld}| + \E \int_{B_1 \backslash B_\dl} |\Xi |^4 + |\U|^4 + |X^+_\dl |^4 \\
& \les \int_{\dl}^1 \bigg(\frac{r-\dl}{r^{d-1}} \bigg)^{\frac32} \frac1{r - \dl} r^{d-1} dr + \E \int_{B_1 \backslash B_\dl} |\Xi |^4 + |\U|^4 + |X^+_\dl |^4 \\
& \les \Ld (\dl),
\end{split}
\label{opt1}
\end{align}

\noi
where we used \eqref{zestimate},  \eqref{Xiestimate1}, \eqref{Uestimate}, \eqref{eq:X+L4}, and \eqref{eq:UX} in the last  two steps.
By collecting \eqref{opt0} and \eqref{opt1}, we arrive at
\begin{align}
\begin{split}
\E \int_{B_1\backslash B_\dl} & \big(W+\Xi  + \U\big)^3 (z_{\dl,\ld} + X^+_\dl ) \\
& \ge - \frac32 \E \int_{B_1 \backslash B_\dl} |W|^2 |z_{\dl,\ld}+X^+_\dl |^2 - \frac{1}{5} \E \int_{B_1 \backslash B_\dl} W^4 - C \Ld (\dl).
\end{split}
\label{opt2}
\end{align}

\noi
We note that the first term on the right-hand side of \eqref{opt2} cancels with the expectation of the third term in \eqref{P1};
and the second term of \eqref{opt2} can be absorbed by the first term in \eqref{opt}.

Now we consider the fourth term in \eqref{P1}.
Expanding the square, we have
\begin{align}
    \label{opt3}
    \begin{split}
    3 & \int_{B_1\backslash B_\dl} W(\Xi  + \U) \big((z_{\dl,\ld} + X^+_\dl )^2 - z_{\dl,\ld}^2\big) \\
& = 3\int_{B_1\backslash B_\dl}  W(\Xi  + \U) \big(2 z_{\dl,\ld}  X^+_\dl  + (X^+_\dl)^2 \big) \\
& = \int_{B_1\backslash B_\dl} 6 W(\Xi  + \U) X^+_\dl  z_{\dl,\ld}  + 3 \int_{B_1\backslash B_\dl} W(\Xi  + \U) (X^+_\dl)^2.
\end{split}
\end{align}

\noi
For the first term in \eqref{opt3},
by using H\"older inequality and hypercontractivity,
we have
\begin{align}
\label{opt31}
\begin{split}
 \E & \int_{B_1\backslash B_\dl} W(\Xi  + \U) X^+_\dl  z_{\dl,\ld} \\
 & \ge - \frac1{960} \E \int_{B_1\backslash B_\dl} W^4  - C \int_{B_1\backslash B_\dl} \E \big| (\Xi  + \U) X^+_\dl \big|^{\frac43} |z_{\dl,\ld} |^{\frac43}  \\
 & \ge - \frac1{960} \E \int_{B_1\backslash B_\dl} W^4  - C \int_{B_1\backslash B_\dl} \Big( \E \big[ |\Xi|^{2} + |\U|^{2} + |X^+_\dl|^{2} \big] \Big)^{\frac43} |z_{\dl,\ld} |^{\frac43}  \\
 & \ge  - \frac1{960} \E \int_{B_1\backslash B_\dl} W^4  - C \int_{\dl}^1 \bigg( \frac{r-\dl}{r^{d-1}} \bigg)^{\frac43} \frac1{(r-\dl)^{\frac43}} r^{d-1} dr  \\
 & \ge - \frac1{960} \E \int_{B_1\backslash B_\dl} W^4  - C \Ld (\dl).
\end{split}
\end{align}

\noi
Similarly, for the second term in \eqref{opt3}, by H\"older and Young inequalities, we have
\begin{align}
\label{opt32}
\begin{split}
 \E & \int_{B_1\backslash B_\dl} W(\Xi  + \U) (X^+_\dl)^2 \\
 & \ge - \frac1{960} \E \int_{B_1\backslash B_\dl} W^4  - C \int_{B_1\backslash B_\dl} \E \big[ | (\Xi  + \U) (X^+_\dl)^2 |^{\frac43} \big] \\
 & \ge - \frac1{960} \E \int_{B_1\backslash B_\dl} W^4  - C \int_{B_1\backslash B_\dl} \E \big[ |\Xi|^{4} + |\U|^{4} + |X^+_\dl|^{4} \big]  \\
 & \ge  - \frac1{960} \E \int_{B_1\backslash B_\dl} W^4  - C \Ld (\dl).
\end{split}
\end{align}

\noi
By collecting \eqref{opt3}, \eqref{opt31}, and \eqref{opt32}, we arrive at
\begin{align}
\label{opt33}
3 \E \int_{B_1\backslash B_\dl} W(\Xi  + \U) \big((z_{\dl,\ld} + X^+_\dl )^2 - z_{\dl,\ld}^2\big) \ge  - \frac3{320} \E \int_{B_1\backslash B_\dl} W^4  - C \Ld (\dl),
\end{align}

\noi
which is again sufficient for our purpose.

Finally, we turn to the last term in \eqref{P1}.
We first write
\begin{align}
\label{opt4}
\begin{split}
\int_{B_1\backslash B_\dl} & \big(W+\Xi  + \U\big)  [(z_{\dl,\ld} + X^+_\dl )^3 - z_{\dl,\ld}^3 - 3 X^+_\dl  z_{\dl,\ld}^2  ] \\
& =  \int_{B_1\backslash B_\dl} \big(W+\Xi  + \U\big)  \big( 3 (X^+_\dl)^2 z_{\dl,\ld} + (X^+_\dl)^3 \big)\\
& = 3 \int_{B_1\backslash B_\dl} W (X^+_\dl)^2 z_{\dl,\ld} dx + \int_{B_1\backslash B_\dl} W (X^+_\dl)^3 dx \\
& \hphantom{XX} + 3 \int_{B_1\backslash B_\dl} \big(\Xi  + \U\big) (X^+_\dl)^2 z_{\dl,\ld} dx +  \int_{B_1\backslash B_\dl} \big(\Xi  + \U\big) (X^+_\dl)^3 dx \\
& = I_1 + I_2 + I_3 + I_4.
\end{split}
\end{align}

\noi
In what follows, we shall estimate the terms $I_1$, $I_2$, $I_3$, and $I_4$ one by one.
For the term $I_1$ in \eqref{opt4},
we have
\begin{align}
\label{opt4A}
\begin{split}
\E [I_1] & = 3 \E \int_{B_1\backslash B_\dl} W (X^+_\dl)^2 z_{\dl,\ld} \\
& \ge -\frac1{640} \E \int_{B_1\backslash B_\dl} W^4 - C \int_{B_1\backslash B_\dl} \E |X^+_\dl |^{\frac83} |z_{\dl,\ld}|^{\frac43} \\
& \ge -\frac1{640} \E \int_{B_1\backslash B_\dl} W^4 -  C \Ld (\dl),
\end{split}
\end{align}

\noi
where, in the last step, we used a similar argument as in \eqref{opt31}.
For the term $I_2$ in \eqref{opt4}, we have
\begin{align}
\label{opt4B}
\begin{split}
\E [I_2] & = \E \int_{B_1\backslash B_\dl} W (X^+_\dl) ^3 \\
& \ge -\frac1{640} \E \int_{B_1\backslash B_\dl} W^4 - C \int_{B_1\backslash B_\dl} \E \big[ |X^+_\dl |^{4} \big]\\
& \ge -\frac1{640} \E \int_{B_1\backslash B_\dl} W^4 -  C \Ld (\dl).
\end{split}
\end{align}

\noi
For the term $I_3$ in \eqref{opt4}, by a similar argument as in \eqref{opt31},
we have
\begin{align}
\label{opt4C}
\begin{split}
\E [I_3] & = 3 \E \int_{B_1\backslash B_\dl} \big(\Xi  + \U\big) (X^+_\dl) ^2 z_{\dl,\ld} \\
& \ge - C \int_{B_1\backslash B_\dl} \E \big[ |X^+_\dl|^3 + |\Xi|^3 + |\U|^3 \big] |z_{\dl,\ld}|\\
& \ge - C \int_{B_1\backslash B_\dl}\Big( \E \big[ |X^+_\dl|^2 + |\Xi|^2 + |\U|^2 \big] \Big)^{\frac32} |z_{\dl,\ld}|\\
& \ge - C \int_{\dl}^1 \bigg( \frac{r-\dl}{r^{d-1}} \bigg)^{\frac32} \frac{r^{d-1}}{r-\dl } dr \ge -  C \Ld (\dl),
\end{split}
\end{align}

\noi
which is sufficient for our purpose.
For the last term $I_4$ in \eqref{opt4}, by H\"older inequality,
we have
\begin{align}
\label{opt4D}
\begin{split}
\E [I_4] & = \E \int_{B_1\backslash B_\dl} \big(\Xi  + \U\big) (X^+_\dl)^3 \ge - C \int_{B_1\backslash B_\dl} \E \big[ |X^+_\dl|^4 + |\Xi|^4 + |\U|^4 \big] \ge -  C \Ld (\dl).
\end{split}
\end{align}

\noi
By collecting \eqref{opt4}, \eqref{opt4A}, \eqref{opt4B}, \eqref{opt4C}, and \eqref{opt4D}, we arrive at
\begin{align}
\label{opt41}
\begin{split}
\E \int_{B_1\backslash B_\dl} & \big(W+\Xi  + \U\big)  [(z_{\dl,\ld} + X^+_\dl)^3 - z_{\dl,\ld}^3 - 3 X^+_\dl z_{\dl,\ld}^2  ] \\
& \ge -\frac1{320} \E \int_{B_1\backslash B_\dl} W^4 -  C \Ld (\dl).
\end{split}
\end{align}

\noi
Therefore, collecting \eqref{P1}, \eqref{opt}, \eqref{eq:UX}, \eqref{opt2}, \eqref{opt33}, and \eqref{opt41}, the resulting coefficient of $\E \int_{B_1\backslash B_\dl} W^4 dx$ is
\[
\frac 9{40} - \frac 15 - \frac 3{320} - \frac 1{320} = \frac 1{80} > 0 ,
\]

\noi
while that of $\E \| W \|_{\dot H^1_{\rad}}^2$ is $\frac12$. This gives \eqref{optestimate} with $c_1 = \frac1{80}$, and thus finishes the proof.
\end{proof}

We also need the following upper bound.

\begin{lemma}
\label{LEM:P0}
Let $d \ge 4$, $\dl \in (0,\frac12]$, and $\ld \in \R$, and let $P_{\dl,\ld}$ be as in Proposition \ref{prop:rhoVar}. Then, we have
\begin{equation}
\label{P0}
\E[P_{\dl,\ld}(0)] \les \Ld (\dl),
\end{equation}

\noi
where $\Ld (\cdot)$ is given in \eqref{Ld}, and the implicit constant is independent of $\dl \in (0,\frac12]$ and $\ld \in \R$.
\end{lemma}

\begin{proof}
The proof is similar to that of Proposition \ref{PROP:opt}.
We give the proof for completeness.
From \eqref{P1}, we see that
\begin{align}
\label{P10}
\begin{split}
P_{\dl,\ld} (0)
& = \frac 14 \int_{B_1\backslash B_\dl} \big(\Xi  + \U\big)^4  + \int_{B_1\backslash B_\dl} \big(\Xi  + \U\big)^3 (z_{\dl,\ld} + X^+_\dl )  \\
& \phantom{=} + \int_{B_1\backslash B_\dl} \big(\Xi  + \U\big)  [(z_{\dl,\ld} + X^+_\dl )^3 - z_{\dl,\ld}^3 - 3 X^+_\dl  z_{\dl,\ld}^2  ] .
\end{split}
\end{align}

\noi
Using Young's inequality, the first two terms of the right-hand side of \eqref{P10} can be bounded by
\begin{align}\label{P11}
\begin{split}
\frac 14 & \int_{B_1\backslash B_\dl} \E  \big(\Xi  + \U\big)^4  + \int_{B_1\backslash B_\dl} \E \big(\Xi  + \U\big)^3 (z_{\dl,\ld} + X^+_\dl)\\
& \les \int_{B_1\backslash B_\dl} \E \big[\Xi ^4 + \U^4 + (X^+_\dl)^4 \big] dx + \int_{B_1\backslash B_\dl} \big( \E \big[|\Xi |^2 + |\U|^2\big]\big)^{\frac32} |z_{\dl,\ld}| dx \\
& \les \int_{B_1\backslash B_\dl} \E \big[\Xi ^4 + \U^4 + (X^+_\dl)^4 \big] dx +  \int_\dl^1 \bigg( \frac{r-\dl}{r^{d-1}}\bigg)^{\frac32} \frac{r^{d-1}}{r-\dl} dr \\
& \les \Ld(\dl),
\end{split}
\end{align}

\noi
where we used hypercontractivity together with \eqref{zestimate}, \eqref{Xiestimate1}, \eqref{Uestimate}, \eqref{eq:X+L4} and \eqref{eq:UX} in the last two steps, exactly as in \eqref{opt1};
this is sufficient for our purpose.

Now we turn to the last term of the right-hand side of \eqref{P10}. Using a similar argument as above yields
\begin{align}
\begin{split}
\E & \int_{B_1\backslash B_\dl} \big(\Xi  + \U\big)  [(z_{\dl,\ld} + X^+_\dl(1))^3 - z_{\dl,\ld}^3 - 3 X^+_\dl(1) z_{\dl,\ld}^2  ] \\
& = \E \int_{B_1\backslash B_\dl} \big(\Xi  + \U\big)  [ 3 X^+_\dl(1)^2 z_{\dl,\ld} + X^+_\dl(1)^3 ] \\
& = \E \int_{B_1\backslash B_\dl} \big(\Xi  + \U\big) X^+_\dl(1)^3 +   3 \E \int_{B_1\backslash B_\dl}  \big(\Xi  + \U\big) X^+_\dl(1)^2 z_{\dl,\ld} \\
& \les \int_{B_1\backslash B_\dl} \E \big[\Xi ^4 + \U^4 + X^+_\dl(1)^4 \big] dx + \int_{B_1\backslash B_\dl} \big( \E \big[|\Xi |^2 + |\U|^2 + |X^+_\dl(1)|^2 \big]\big)^{\frac32} |z_{\dl,\ld}| dx \\
& \les \int_{B_1\backslash B_\dl} \E \big[\Xi ^4 + \U^4 + (X^+_\dl)^4 \big] dx +  \int_\dl^1 \bigg( \frac{r-\dl}{r^{d-1}}\bigg)^{\frac32} \frac{r^{d-1}}{r-\dl} dr \\
& \les \Ld(\dl),
\end{split}
\label{P12}
\end{align}

\noi
where, again, we used hypercontractivity together with \eqref{zestimate}, \eqref{Xiestimate1}, \eqref{Uestimate}, \eqref{X+estimate}, \eqref{eq:X+L4} and \eqref{eq:UX} in the last two steps.
Thus, we finish the proof of \eqref{P0} from \eqref{P10}, \eqref{P11}, and \eqref{P12}.
\end{proof}

\begin{remark}
\rm
In the construction of the $\Phi^4$ measure on $B_1 \backslash \{0\}$ in the next section, we will only use that $\Ld (\dl)$ is a constant depending on $\dl$. However, to upgrade this measure to one supported on spaces of functions over $B_1$, the more precise dependence on $\dl$ recorded in \eqref{Ld} is needed.
\end{remark}

\section{Construction of the measure}
\label{SEC:proof}
Before we proceed to the proof of Theorem \ref{THM:main}, we need the following Sobolev embedding.

\begin{lemma}
\label{LEM:Sob}
Let $w \in \dot H^1_{\rad}(B_1 \backslash B_\dl)$ with $0$ Dirichlet boundary conditions. Then for every $\dl \le r \le 1$,
\begin{equation}\label{poincare}
 |w(r)|
 \les \dl^{-\frac{d-2}2} \|w\|_{\dot H^1_{\rad}(B_1 \backslash B_\dl)} .
\end{equation}

\noi
Furthermore, given $\al \in (0, \frac12]$, we have
\begin{equation}
\label{poin2}
\| w\|_{C^{\al } (B_1 \backslash B_{\dl})} \les \dl^{-\frac{d-2}2 - \al} \|w\|_{\dot H^1_{\rad}(B_1 \backslash B_\dl)},
\end{equation}

\noi
where $C^\al$ is the H\"older space defined in \eqref{holder}.
\end{lemma}
\begin{proof}
Since $w(\dl) = 0$, we use the fundamental theorem of calculus and  H\"older's inequality to get
\begin{align}
|w(r)| = \bigg| \int_\dl^r \partial_r w (s) ds \bigg| \le \bigg(\int_\dl^r |\partial_r w (s)|^2 s^{d-1} ds \bigg)^\frac12 \Big(\int_\dl^r s^{-d+1} ds\Big)^\frac12,
\label{poin10}
\end{align}

\noi
which is sufficient for \eqref{poincare}.

Then we turn to \eqref{poin2}.
A similar argument as in \eqref{poin10} yields, for the $C^{\frac12}$ seminorm,
\[
\begin{split}
\sup_{\substack{r_1,r_2 \in (\dl, 1) \\ r_1 \neq r_2}} \frac{|w(r_1) - w(r_2)|}{|r_1 - r_2|^{\frac12}}
& \le \sup_{\substack{r_1,r_2 \in (\dl, 1) \\ r_1 \neq r_2}} \frac{\big( \int_{r_1}^{r_2} | \partial_r w(r) |^2 r^{d-1} d r \big)^{\frac12} \big(\int_{r_1}^{r_2} r^{-d+1} dr\big)^{\frac12}}{|r_1 - r_2|^{\frac12}} \\
& \les \dl^{- \frac{d-1}2} \|w\|_{\dot H^1 (B_1 \backslash B_\dl)}.
\end{split}
\]

\noi
Combining this with \eqref{poincare} and the interpolation inequality
$\| w\|_{C^{\al}} \les \|w\|_{L^\infty}^{1-2\al} \|w\|_{C^{\frac12}}^{2\al}$, valid for $\al \in (0,\frac12]$, we obtain
\[
\| w\|_{C^{\al} (B_1 \backslash B_{\dl})} \les \dl^{-\frac{d-2}2 (1-2\al)} \dl^{- \frac{d-1}2 \cdot 2\al}\|w\|_{\dot H^1 (B_1 \backslash B_\dl)} = \dl^{- \frac{d-2}2 - \al} \|w\|_{\dot H^1 (B_1 \backslash B_\dl)},
\]

\noi
which finishes the proof of \eqref{poin2}.
\end{proof}

\subsection{Tightness}
In this subsection, we show that the family $\{\rho_\eps\}_{\eps \in (0,1)}$ is tight on the space $\mathbb X^\al$ defined in \eqref{Xp}, and we extract a weak limit point $\rho_0$ along a sequence $\eps_n \to 0$.
The main ingredient is the following exponential moment bound for H\"older norms away from the origin.

\begin{lemma} \label{LEM:EB}
Recall the truncated Gibbs measure $\rho_\eps$ defined in \eqref{rho_eps}.
For every $\dl \in (0,\frac12]$ and $\al \in (0, \frac12)$, there exist constants $\eta = \eta(\dl,\al) \ll 1$ and $C=C(\dl,\al) >0$ such that
    \begin{equation}
    \label{exp_bound}
    \int \exp \Big (\eta \|u \|^2_{C^{\al}(B_1\backslash B_{2\dl})} \Big) d \rho_{\eps}(u) \le C,
    \end{equation}
uniformly in $\eps \in (0, \dl]$.
\end{lemma}
\begin{proof}
The functional $F$ appearing below is upper semicontinuous on $L^2$, so that it does not satisfy the hypotheses of Proposition \ref{PROP:rho}, and we nonetheless apply the variational formula \eqref{eq:rhoVar} to it.
The complete argument follows by a standard approximation procedure, in which $u$ is replaced by a regularisation of $u$; we omit the details.

From \eqref{rho_eps} and \eqref{rho_eps1}, we can rewrite the  left-hand side of \eqref{exp_bound} as
\begin{align*}
\begin{split}
\int \exp & \Big (\eta \|u\|^2_{C^{\al}(B_1\backslash B_{2\dl})} \Big) d \rho_{\eps}(u)
 = \int \bigg( \int \exp \Big (\eta \|u\|^2_{C^{\al}(B_1\backslash B_{2\dl})} \Big)   d \rho^+_{\dl} [\ld] (u) \bigg) \nu_{\eps,\dl} (d \lambda),
\end{split}
\end{align*}

\noi
where the probability measure  $\nu_{\eps,\dl}$ is given in  \eqref{nued}.
Therefore, to prove \eqref{exp_bound}, it is enough to show that
\begin{align} \label{uni_b}
\int \exp \Big( \eta \|u\|^2_{C^{\al}(B_1\backslash B_{2\dl})} \Big) d \rho_{\dl}^+[\ld](u) \les 1,
\end{align}

\noi
uniformly in $\ld$.
Denote by
\[
F (u) = - \eta \|u\|^2_{C^{\al}(B_1\backslash B_{2\dl})}.
\]

\noi
Then, by the variational formula \eqref{eq:rhoVar} we get,
\begin{align}
\label{uni_b1}
\begin{split}
     \log \int &\exp \Big( -F(u) \Big) d \rho_{\dl}^+[\ld](u) \\
 & = \inf_{W \in  \mathbb H}
	\E\big[ P_{\dl,\ld}(W) \big]
	- \inf_{W \in  \mathbb H}
	\E\big[ F(X^+_\dl(1) + z_{\dl,\ld} + W + \Xi + \U) + P_{\dl,\ld}(W) \big].
\end{split}
\end{align}

\noi
We choose $W_2 \in \mathbb H$  as in \eqref{W2}.
It then follows from \eqref{uni_b1} that
\begin{align}
\begin{split}
\log \int \exp \big( -F(u) \big) d \rho_{\dl}^+[\ld](u)   & \le 1 - \E\big[ F(X^+_\dl(1) + z_{\dl,\ld} + W_2 + \Xi + \U) \big] \\
& = 1+ \eta  \E \Big[  \|X^+_\dl(1) + z_{\dl,\ld} + W_2 + \Xi + \U  \|^2_{C^{\al}(B_1\backslash B_{2\dl})} \Big] \\
& \le 1+  C_{\dl} \eta \Big(1 + \E \|W_2(1)\|_{C^{\al}(B_1\backslash B_{2\dl})}^2 \Big) \\
& \le 1 +  C_{\dl} \eta (1 + \E \|W_2(1)\|_{\dot H^1(B_1\backslash B_{\dl})}^2 ),
\end{split}
\label{poinC121}
\end{align}

\noi
where we used Lemma \ref{LEM:reg}, Lemma \ref{lem:xizC12} with $p = 2$, and $W_2 \in \mathbb H$ together with \eqref{poin2}.
Here $C_\dl$ is a constant depending on $\dl$ (and $\al$) that may vary from line to line.

It suffices to show that $\E \|W_2(1)\|_{\dot H^1}^2$ is uniformly bounded in $\ld$.
With the same choice of $W_2$ from \eqref{W2}, we have
\begin{align}
\begin{split}
& \E  \big[ F(X^+_\dl(1) + z_{\dl,\ld} + W_2 + \Xi + \U) + P_{\dl,\ld}(W_2) \big] \\
& \le 1+  \inf_{W \in  \mathbb H} \E\big[ F(X^+_\dl(1) + z_{\dl,\ld} + W + \Xi + \U) + P_{\dl,\ld}(W) \big] \\
& \le 1 + \E\Big[ F(X^+_\dl(1) + z_{\dl,\ld} + \Xi + \U)  + P_{\dl, \ld}(0) \Big] \\
& \le 1 + \E\Big[ P_{\dl, \ld}(0) \Big] \les  \Ld(\dl),
\end{split}
\label{C12lower}
\end{align}
where the last inequality follows from
Lemma \ref{LEM:P0}. On the other hand, by Lemma \ref{lem:xizC12} and \eqref{optestimate}, we obtain
\begin{align*}
\E & \big[ F(X^+_\dl(1) + z_{\dl,\ld} + W_2 + \Xi + \U) + P_{\dl,\ld}(W_2) \big] \\
& =  \E\big[ - \eta \|X^+_\dl(1) + z_{\dl,\ld} + W_2 + \Xi + \U\|_{C^{\al}(B_1\backslash B_{2\dl})}^2 + P_{\dl,\ld}(W_2) \big] \\
&\ge - C \E\Big[ \eta \big(  \big\| \Xi \big\|^2_{C^{\al}(B_1\backslash B_{2\dl})} +  \big\| X^+_\dl(1)\big\|^2_{C^{\al}(B_1\backslash B_{2\dl})} +  \big\|W_2 \big\|^2_{C^{\al}(B_1\backslash B_{2\dl})}  +  \big\|z_{\dl,\ld} \big\|^2_{C^{\al}(B_1\backslash B_{2\dl})} \\
& \hphantom{XX}
+  \big\|\U \big\|^2_{C^{\al}(B_1\backslash B_{2\dl})} \big) \Big] + c_1 \E\Big[ \int W_2 (1)^4 dx + \| W_2 (1) \|_{\dot H^1_{\rad}}^2\Big] -C_\dl,\\
\intertext{where $c_1$ is defined in Proposition \ref{PROP:opt}. We now insert \eqref{poin2} to continue with}
& \ge \E\Big[ c_1 \int W_2 (1)^4 dx + \big(c_1 - \eta C \dl^{-(d-2)-2\al}\big)\| W_2 (1) \|_{\dot H^1_{\rad}}^2\Big] - C C_\dl.\\
\intertext{We then use Lemma \ref{lem:xizC12} and choose $\eta < \frac{c_1\dl^{d-2+2\al}}{2 C}$ to get}
& \ge \frac{c_1}2 \E\Big[ \int W_2 (1)^4 dx + \| W_2 (1) \|_{\dot H^1_{\rad}}^2\Big] - C^2 C_\dl.
\end{align*}

\noi
In particular, together with \eqref{C12lower}, the above implies
\begin{align}\label{uni_b2}
\E \|W_2(1)\|_{\dot H^1  (B_1\backslash B_{\dl})}^2 \les C_\dl
\end{align}
where the constant is independent of $\ld$ and $\eps$.
Then, the uniform estimate \eqref{uni_b} follows from \eqref{poinC121} and \eqref{uni_b2}.
Thus, we finish the proof of the lemma.
\end{proof}

Now, we are ready to prove the tightness of the set of measures  $\{\rho_\eps\}_{\eps \in (0,1)}$.

\begin{lemma}
\label{LEM:tight}
Let $0 < \al <\frac12$, and fix $\g > 0$ with $\al + \g < \frac12$.
Given $\ta > 0$, there exists a sequence $\{N_k\}_{k \ge 2}$ of positive numbers such that the set
\begin{align}
\label{Kset}
K := \bigcap_{k \ge 2} \Big\{ u \in \mathbb X^\al : \ \|u\|_{C^{\al+\g} (B_1 \backslash B_{1/k})} \le N_k \Big\}
\end{align}

\noi
is a compact subset of $\mathbb X^\al$ and satisfies
\begin{align}
\label{Ktight}
\sup_{\eps \in (0,1)} \rho_\eps (K^c) \le \ta .
\end{align}

\noi
In particular, the family $\{\rho_\eps\}_{\eps \in (0,1)}$ is tight on $\mathbb X^\al$.
\end{lemma}

\begin{proof}
We first check that $K$ is compact.
Since $\mathbb X^\al$ is metrizable, it suffices to show that $K$ is sequentially compact.
Let $\{u_j\}_{j \ge 1} \subseteq K$.
For each fixed $k \ge 2$, the sequence $\{u_j\}_{j\ge1}$ is bounded in $C^{\al+\g} (B_1 \backslash B_{1/k})$, hence relatively compact in $c^{\al} (B_1 \backslash B_{1/k})$ by the Arzel\`a--Ascoli theorem and interpolation between $L^\infty$ and $C^{\al+\g}$, where we note that $C^{\al+\g} (B_1 \backslash B_{1/k}) \subseteq c^{\al} (B_1 \backslash B_{1/k})$. Therefore, there exists a subsequence $u_{l_{j,k}}$, and a function $u_k \in c^{\al} (B_1 \backslash B_{1/k})$ such that
$$ \lim_{j \to \infty} \| u_{l_{j,k}} - u_k\|_{c^{\al} (B_1 \backslash B_{1/k})} =0. $$
Moreover, by extracting the subsequence from $ u_{l_{j,k}}$ rather than $u_j$, we can pick the indices $l_{j,k+1}$ to be a subsequence of $\{l_{j,k}\}_j$. Therefore, we also have that
$$  \lim_{j \to \infty}  \| u_{l_{j,j}} - u_k\|_{c^{\al} (B_1 \backslash B_{1/k})} =0$$
as well for every $k \ge 2$. In particular, by compatibility of the limits, on the intersection of the domains we must have that $u_k = u_h$ for every $k,h\ge 2$. Therefore, defining $u(x) := u_k(x)$ for any $k > \frac{1}{|x|}$, we have that $u_{l_{j,j}} \to u$ in the space $\mathbb X^{\al}$.
Moreover, $u_{l_{j,j}} \to u$ pointwise on each $B_1 \backslash B_{1/k}$, and the $C^{\al+\g}$ norm is lower semicontinuous under pointwise convergence, so that
$$ \|u\|_{C^{\al+\g} (B_1 \backslash B_{1/k})} \le \liminf_{j \to \infty} \|u_{l_{j,j}}\|_{C^{\al+\g} (B_1 \backslash B_{1/k})} \le N_k $$
for every $k \ge 2$, i.e.\ $u \in K$. This shows that $K$ is compact.

We now turn to \eqref{Ktight}. Fix $k \ge 2$. By Lemma \ref{LEM:EB}, applied with $\frac1{2k}$ in place of $\dl$ and $\al + \g$ in place of $\al$, and Chebyshev's inequality, we have that for $N_k$ big enough,
\begin{align*}
\rho_\eps \Big( \big\{ \|u\|_{C^{\al+\g} (B_1 \backslash B_{1/k})} > N_k \big\} \Big) \le 2^{-k}\ta,
\end{align*}
as long as $\eps \le \frac 1 {2k}$. Moreover, the partition function $Z_\eps$ in \eqref{partition_eps} is non-decreasing in $\eps$, and $Z_{1/(2k)} > 0$ by Proposition \ref{LEM:Ga}, so that $d \rho_\eps / d \mu \le Z_\eps^{-1} \le Z_{1/(2k)}^{-1}$ for $\eps > \frac1{2k}$. Hence, for every $N > 0$,
$$ \sup_{\eps > \frac 1{2k}} \rho_\eps \Big( \big\{ \|u\|_{C^{\al+\g} (B_1 \backslash B_{1/k})} > N \big\} \Big) \le Z_{1/(2k)}^{-1} \, \mu\Big( \big\{ \|u\|_{C^{\al+\g} (B_1 \backslash B_{1/k})} > N \big\} \Big) \to 0   $$
as $N\to \infty$, since $\al + \g < \frac12$ and, as shown  in the proof of Lemma \ref{LEM:Xal}, the restriction of $\mu$-almost every $u$ to $B_1 \backslash B_{1/k}$ belongs to $C^{\al+\g} (B_1 \backslash B_{1/k})$. Therefore, up to making $N_k$ bigger, we have that
\begin{align}
\label{tight1}
\sup_{ \eps} \rho_\eps \Big( \big\{ \|u\|_{C^{\al+\g} (B_1 \backslash B_{1/k})} > N_k \big\} \Big) \le 2^{-k}\ta.
\end{align}
Therefore,
\begin{align*}
\sup_{\eps} \rho_\eps (K^c) \le \sum_{k \ge 2} \sup_{\eps} \rho_\eps \Big( \big\{ \|u\|_{C^{\al+\g} (B_1 \backslash B_{1/k})} > N_k \big\} \Big) \le \ta.
\end{align*}
\end{proof}

\begin{lemma}
\label{LEM:limit}
Let $\al$ be as in Lemma \ref{LEM:tight}.
Then there exist a sequence $\{\eps_n\}_{n \ge 1}$ with $\eps_n \to 0$ and a Borel probability measure $\rho_0$ on $\mathbb X^\al$ such that
\begin{align}
\label{limit_rho}
\rho_{\eps_n} \rightharpoonup \rho_0 \qquad \text{in } \mathbb X^\al
\end{align}

\noi
as $n \to \infty$.
In particular, for every $\dl \in (0,1)$ we have
\begin{align}
\label{limit_rho2}
\lim_{n \to \infty} (R_\dl)_{\#} \rho_{\eps_n} = (R_\dl)_{\#} \rho_0 ,
\end{align}

\noi
weakly on $c^\al_{\rad} (B_1 \backslash B_\dl)$, where $R_\dl$ is the restriction map defined in \eqref{R_res}.
\end{lemma}

\begin{proof}
Since the space $\mathbb X^\al$ is Polish (see the discussion after \eqref{Xpsemi}), and by Lemma \ref{LEM:tight} the family $\{\rho_\eps\}_{\eps \in (0,1)}$ is tight on it, Prokhorov's theorem \cite[Theorem 5.1]{Bil99} provides a sequence $\eps_n \to 0$ and a Borel probability measure $\rho_0$ on $\mathbb X^\al$ for which \eqref{limit_rho} holds.
Moreover, since $R_\dl : \mathbb X^\al \to c^\al_{\rad}(B_1\backslash B_\dl)$ is continuous, \eqref{limit_rho2} follows from the weak convergence of $\rho_{\eps_n} $ to $\rho_0$.
\end{proof}

The Markov property \eqref{eq:rhoMarkov} passes to the limit along the subsequence of Lemma \ref{LEM:limit}.

\begin{proposition}[Markov property for $\rho_0$]
\label{PROP:Markov0}
Let $\{\eps_n\}_{n \ge 1}$ and $\rho_0$ be as in Lemma \ref{LEM:limit} and, given $\dl \in (0,1)$, let $\nu_{0,\dl} := \Law (u(\dl))$ for $u$ distributed according to $\rho_0$.
Then
\begin{align}
\label{eq:rhoMarkov0}
(R_\dl)_{\#} \rho_0 = \int \rho_\dl^+[\ld] \, \nu_{0,\dl} (d \lambda) .
\end{align}
\end{proposition}

\begin{proof}

Fix $\dl \in (0,1)$.
Since $|u(\dl)| \le \|u\|_{C^\al (B_1\backslash B_\dl)}$, Lemma \ref{LEM:EB} applied with $\frac \dl 2$ in place of $\dl$ gives
\[
\int \exp \big( \eta \ld^2 \big) \, \nu_{\eps, \dl} (d \ld) = \int \exp \big( \eta\, u(\dl)^2 \big) d \rho_\eps (u) \le C_\dl ,
\]

\noi
uniformly in $\eps \in (0, \frac \dl 2]$, where we used that $\nu_{\eps,\dl}$ is the law of $u(\dl)$ under $\rho_\eps$, as stated in Proposition \ref{prop:Markov}.
In particular, the family $\{ \nu_{\eps,\dl}\}_{\eps}$ is tight.

Moreover, since the evaluation $u \mapsto u(\dl)$ is continuous on $\mathbb X^\al$, by \eqref{limit_rho} we get
\begin{align}
\label{nulimit}
\nu_{\eps_n, \dl} \rightharpoonup \nu_{0, \dl} \qquad \text{as } n \to \infty .
\end{align}

\noi
It remains to pass to the limit in \eqref{eq:rhoMarkov}.
Let $F$ be bounded and continuous on $c^\al (B_1\backslash B_\dl)$.
From \eqref{meas_cut} and \eqref{mu_pm} we have
\[
\int F \, d \rho_\dl^+[\ld] = \frac{\E \big[ F (X^+_\dl + \ld h_\dl) \, e^{-V^+_\dl (X^+_\dl + \ld h_\dl)}\big]}{\E \big[ e^{-V^+_\dl (X^+_\dl + \ld h_\dl)}\big]} ,
\]

\noi
where the denominator is positive.
By Lemma \ref{LEM:reg} with $0 < \ta < \frac12 - \al$, we have $X^+_\dl \in C^{\frac12 - \ta} (B_1 \backslash B_\dl) \subseteq c^\al (B_1 \backslash B_\dl)$ almost surely, and $h_\dl$ in \eqref{hdelta} is smooth on the closed shell, so that $\ld \mapsto X^+_\dl + \ld h_\dl$ is almost surely continuous from $\R$ into $c^\al (B_1 \backslash B_\dl) \cap L^4 (B_1 \backslash B_\dl)$.
Since $F$ is bounded and continuous on $c^\al (B_1\backslash B_\dl)$, and $V^+_\dl$ is continuous and non-negative on $L^4 (B_1\backslash B_\dl)$, dominated convergence, with the integrands bounded by $\|F\|_{L^\infty}$ and $1$, shows that the numerator and the denominator above are continuous in $\ld$.
Hence, $\ld \mapsto \int F \, d \rho_\dl^+[\ld]$ is continuous, and it is bounded by $\|F\|_{L^\infty}$.
Testing \eqref{eq:rhoMarkov} against $F$ along $\{\eps_n\}$ therefore gives
\[
\int F \, d (R_\dl)_{\#} \rho_{\eps_n} = \int \bigg( \int F \, d \rho_\dl^+[\ld] \bigg) \nu_{\eps_n, \dl} (d \ld)
\longrightarrow \int \bigg( \int F \, d \rho_\dl^+[\ld] \bigg) \nu_{0, \dl} (d \ld)
\]

\noi
by \eqref{nulimit}, while the left-hand side converges to $\int F \, d (R_\dl)_{\#} \rho_0$ by \eqref{limit_rho2}.
This proves \eqref{eq:rhoMarkov0}.
\end{proof}

\subsection{Limiting measures}
In this subsection, we shall show that the limiting measures constructed in Lemma \ref{LEM:limit} are actually concentrated in $L^{p}(B_1)$ for every $p < \frac{2d}{d-2}$.
In other words, a function sampled from $\rho_0$ is integrable up to the origin, at the same critical exponent as a sample from the Gaussian measure $\mu$, see \eqref{GMap4_B} and Lemma \ref{LEM:GMap_B}.

\begin{proposition}\label{PROP:PW}
There exist universal constants $\eta \ll 1$, $C >0$, such that for every $r \in (0,1)$ and every $0 < \eps \le \frac r2$,
\begin{equation}
\int \exp \bigg ( \frac{ - F(u,r)}{{1 + | \Ld(r)|}} \bigg) d \rho_\eps(u) \le C ,
\end{equation}

\noi
where $\Ld(\cdot)$ is defined in \eqref{Ld}, and
\begin{align}\label{Fur}
F(u,r) := - \eta r^{d-2} u^2 (r) \qquad \text{or} \qquad F(u,r) := - \eta \| u\|_{L^4(B_1 \backslash B_r)}^4 ;
\end{align}

\noi
the statement is asserted for each of these two choices of $F$.
\end{proposition}

Before the proof of Proposition \ref{PROP:PW}, we need to recall the following weighted radial Sobolev embedding.

\begin{lemma}[Radial lemma, \cite{Ni}]
\label{LEM:Ni}
Let $d \ge 3$, and let $u$ be a radially symmetric function on the unit ball $B_1$ of $\R^d$ with $u(1) = 0$. Then
\[
\sup_{x \in B_1} |x|^{\frac{d-2}2} |u(x)| \le \frac1{\sqrt{\o_d (d-2)}} \| \nabla u\|_{L^2 (B_1)},
\]
where $\o_d$ is the surface area of the unit sphere in $\R^d$.
\end{lemma}

We are now ready to  prove Proposition \ref{PROP:PW}.

\begin{proof}[Proof of Proposition \ref{PROP:PW}]
For the first choice of $F$ in \eqref{Fur}, the functional does not satisfy the hypotheses of Proposition \ref{PROP:rho}, since the point evaluation $u \mapsto u(r)$ is not defined on all of $L^2$, and we nonetheless apply the variational formula \eqref{eq:rhoVar} to it.
The complete argument follows by a standard approximation procedure, in which $u$ is replaced by a regularisation of $u$, and we omit the details.

By Proposition \ref{prop:Markov}, it is enough to show that there exist universal constants $\eta,C$ such that
\begin{align}
\label{PW0}
\int \exp \bigg ( \frac{- F(u,r) }{{1 + | \Ld(r)|}} \bigg) d \rho_{\frac r2}^+[\ld](u) \le C ,
\end{align}

\noi
uniformly in $\ld \in \R$,
where $F(u,r)$ is given in \eqref{Fur}.
To prove \eqref{PW0}, it suffices to show
\begin{align}
\label{PW01}
\log \int \exp\big(- F(u,r) \big) d \rho_{\frac r2}^+[\ld](u) \les  1 + \Ld(r)
\end{align}

\noi
 independently of $\ld \in \R$, where $\Ld(\cdot)$ is defined in \eqref{Ld}.
To see why \eqref{PW01} implies \eqref{PW0}, let us first assume that \eqref{PW01} is true, then
\begin{align}
\begin{split}
\int \exp \bigg ( \frac{-F(u,r)}{ {1 + \Lambda (r)}} \bigg) d \rho_{\frac r2}^+[\ld](u)
& = \int \bigg( \exp \big ( -F(u,r) \big) \bigg)^{\frac1{ {1 +  \Ld(r)}}} d \rho_{\frac r2}^+[\ld](u) \\
& \le \bigg( \int \exp \big ( -F(u,r) \big)  d \rho_{\frac r2}^+[\ld](u) \bigg)^{\frac1{ {1 +  \Ld(r)}}} \\
& \le \Big( e^{C (1 + \Ld(r))}\Big)^{\frac1{ {1 +  \Ld(r)}}} = e^C,
\end{split}
\end{align}
\noi
where we used H\"older's inequality in the second step and \eqref{PW01} in the third step.

In what follows, we shall prove \eqref{PW01}.
For this purpose,
we apply Proposition \ref{prop:rhoVar} to the left-hand side of \eqref{PW01}.
To be more precise,
from \eqref{varub} with \eqref{Fur}, we can find $W_2 \in \mathbb H$ such that
\begin{align}
\label{PW_1}
\log \int \exp ( -F(u)) d \rho_{\frac r2}^+[\ld](u) \le 1 - \E [F(X^+_{r/2}  + z_{r/2, \ld} + W_2  + \Xi + \U)].
\end{align}

\noi
In particular, this $W_2$ is the almost optimiser for
\[
\inf_{W \in \mathbb H} \E [ F(X^+_{r/2}  + z_{r/2, \ld} +W + \Xi + \U) + P_{r/2,\ld} (W)],
\]

\noi
in the sense that
\begin{align}
\label{PW_opti}
\begin{split}
\E [ & F(X^+_{r/2}  + z_{r/2, \ld} +W_2 + \Xi + \U) + P_{r/2, \ld} (W_2)] \\
& \le 1 + \inf_{W \in \mathbb H} \E [ F(X^+_{r/2}  + z_{r/2, \ld} +W  + \Xi + \U) + P_{r/2, \ld} (W)].
\end{split}
\end{align}

\noi
Therefore, from \eqref{PW_1}, it suffices to show that
\begin{align}
    \label{PW_PW2}
    - \E [F(X^+_{r/2}  + z_{r/2, \ld} + W_2   + \Xi + \U)] \les 1 + \Ld(r) .
\end{align}

\noi
We note from \eqref{PW_opti} that
\begin{align}
\label{PW_opti2}
\begin{split}
\E & [F(X^+_{r/2}  + z_{r/2, \ld} +W_2  + \Xi + \U) + P_{r/2, \ld} (W_2)]  \\
& \le 1 + \inf_{W \in \mathbb H} \E [F(X^+_{r/2}  + z_{r/2, \ld} +W  + \Xi + \U) + P_{r/2, \ld} (W)] \\
& \le 1 + \E [F(X^+_{r/2}  + z_{r/2, \ld}  + \Xi + \U) + P_{r/2, \ld} (0)]\\
& \le 1 + \E [P_{r/2, \ld} (0)] \les 1 + \Ld(r),
\end{split}
\end{align}

\noi
where we used the fact that $F \le 0$, Lemma \ref{LEM:P0}, and $\Ld(\frac r2) \les 1 + \Ld(r)$.

To proceed with the proof of \eqref{PW01}, we distinguish two cases.

First, let us consider when $F$ is given by $-\eta r^{d-2} u^2 (r)$ in \eqref{Fur}.
By \eqref{X+estimate}, \eqref{Xiestimate}, \eqref{optestimate} and \eqref{poincare}, and Proposition \ref{PROP:opt}, we obtain that for $\eta \ll 1$,
\begin{align}
\begin{split}
\E & [F(X^+_{r/2}  + z_{r/2, \ld} +W_2  + \Xi + \U) + P_{r/2, \ld} (W_2)]\\
&= \E \bigg[ -\eta r^{d-2} (W_2 + z_{r/2, \ld}  + \Xi + X^+_{r/2} + \U)^2 (r) + P_{\frac r2, \ld}(W_2) \bigg]\\
& \ge \  \E\big[ - C \eta r^{d-2} W_2^2 (r) + P_{\frac r2, \ld}( W_2) \big] - C \eta r^{d-2}  \E \big[ ( \Xi + z_{r/2, \ld}   + X^+_{r/2}  + \U  \big)^2 (r) \big],
\end{split}
\label{PW_PW3}
\end{align}

\noi
where $F(u)$ is given in \eqref{Fur}.
Recall from the radial  lemma, Lemma \ref{LEM:Ni}, that
\begin{align}
\label{RL}
 r^{d-2} W_2^2 (r)  \le \frac1{\o_d (d-2)} \| \nabla W_2\|^2_{L^2 (B_1)}.
\end{align}

\noi
By using Lemma \ref{LEM:X+}, Lemma \ref{LEM:Xi}, Lemma \ref{LEM:Upsilon}, and Lemma \ref{LEM:z},
we obtain
\begin{align}
\label{RE2}
\E \big[ r^{d-2} ( \Xi + z_{r/2, \ld}  + X^+_{r/2}  + \U)^2 (r) \big] \les 1.
\end{align}

\noi
Therefore, from \eqref{PW_PW3}, \eqref{RL}, \eqref{RE2}, and Proposition \ref{PROP:opt}, we obtain
\begin{align}
\label{PWPW}
\begin{split}
\E \big [ & F(X^+_{r/2}  + z_{r/2, \ld} +W_2  + \Xi + \U) + P_{r/2, \ld} (W_2) \big] \\
& \ge \frac{c_1}2\E\Big[  \int W_2^4 dx + \| W_2  \|_{\dot H^1_{\rad}}^2\Big] - C  (1 + \Ld(r)),
\end{split}
\end{align}

\noi
provided $C \eta \le \frac12 c_1 \o_d (d-2)$.
Therefore, from \eqref{PWPW} and \eqref{PW_opti2}, we obtain that $$
\E\Big[  \int W_2^4 dx + \| W_2 \|_{\dot H^1_{\rad}}^2\Big]  \les 1+ \Ld(r) ,
$$

\noi
which together with  \eqref{RL} and \eqref{RE2} implies
\[
\begin{split}
- & \E \big [ F(X^+_{r/2}  + z_{r/2, \ld} +W_2  + \Xi + \U) \big] \\
& = \eta \, \E \big[ r^{d-2} (X^+_{r/2}  + z_{r/2, \ld} +W_2  + \Xi + \U)^2 (r) \big] \\
& \les \eta r^{d-2} \big( \E [|X^+_{r/2} (r)|^2 ] + |z_{r/2, \ld} (r)|^2 +  \E [|\Xi (r)|^2] +  \E [|\U (r)|^2] \big) + \eta \, \E \|W_2 \|_{\dot H^1_{\rm rad}}^2    \\
& \les  r^{d-2} \big( r^{-d+2} + (r - r/2)^{-2} \big) + 1 + \Ld(r) \les 1 + \Ld(r) ,
\end{split}
\]

\noi
which finishes the proof of \eqref{PW_PW2}, and thus the  proposition for the case $ F = -\eta r^{d-2} u^2 (r)$.

It remains to consider the case when $F = - \eta \| u \|_{L^4 (B_1 \backslash B_r)}^4$ in \eqref{Fur}.
The argument is similar to the previous case.
By \eqref{X+estimate}, \eqref{Xiestimate}, \eqref{optestimate}, and Proposition \ref{PROP:opt}, we obtain that for $\eta \ll 1$,
\begin{align}
\begin{split}
\E & [F(X^+_{r/2}  + z_{r/2, \ld} +W_2  + \Xi + \U) + P_{r/2, \ld} (W_2)]\\
&= \E \bigg[ -\eta \| W_2 + \Xi + z_{r/2, \ld}  + X^+_{r/2} + \U \|_{L^4 (B_1 \backslash B_r)}^4 + P_{\frac r2, \ld}(W_2) \bigg]\\
& \ge \  \E\big[  P_{\frac r2, \ld}( W_2) \big] - C \eta \E \big[  \| W_2\|_{L^4 }^4 + \|  \Xi \|_{L^4 }^4 + \| z_{r/2, \ld} \|_{L^4 }^4 + \| X^+_{r/2} \|_{L^4 }^4 + \| \U \|_{L^4 }^4 \big]\\
& \ge \  \E\big[ ( c_1 - C \eta) \| W_2\|_{L^4 }^4 +  c_1 \| W_2 \|_{\dot H^1_{\rad}}^2\big] \\
& \hphantom{XX}- C \eta \E \big[ \|  \Xi \|_{L^4 }^4 + \| X^+_{r/2} \|_{L^4 }^4 + \| \U \|_{L^4 }^4  + \| z_{r/2, \ld} \|_{L^4 }^4\big]  - C (1 + \Ld (r)),
\end{split}
\label{PW_PW3_2}
\end{align}

\noi
where and in what follows, we slightly abuse notation by writing $L^4$ to denote the space $L^4(B_1 \backslash B_r)$.
By using Corollary \ref{COR:L4est} and Lemma \ref{LEM:L4-2} with $\dl = \frac r2$, and \eqref{zestimate},
we obtain
\begin{align}
\label{RE2_2}
\E \big[  \|  \Xi \|_{L^4 }^4 + \| X^+_{r/2} \|_{L^4 }^4 + \| \U \|_{L^4 }^4  + \| z_{r/2, \ld} \|_{L^4 }^4\big] \les 1 + \Ld (r).
\end{align}

\noi
Therefore, from \eqref{PW_PW3_2} and \eqref{RE2_2}, we obtain
\begin{align}
\label{PWPW_2}
\begin{split}
\E \big [ & F(X^+_{r/2}  + z_{r/2, \ld} +W_2  + \Xi + \U) + P_{r/2, \ld} (W_2) \big]  \ge \frac{c_1}2 \E \big[ \| W_2\|_{L^4 }^4  \big] - C  (1 + \Ld(r)),
\end{split}
\end{align}

\noi
provided $C \eta < \frac{c_1}2$.
By combining \eqref{PWPW_2} and \eqref{PW_opti2},
we obtain that
$$
\E \| W_2\|_{L^4 }^4  \les  1 + \Ld(r) ,
$$

\noi
which, together with \eqref{PWPW_2} and \eqref{RE2_2}, implies
\[
\begin{split}
- & \E \big [ F(X^+_{r/2}  + z_{r/2, \ld} +W_2  + \Xi + \U) \big] \\
& = \eta \E \big[ \| X^+_{r/2}  + z_{r/2, \ld} +W_2  + \Xi + \U\|_{L^4}^4 \big] \\
& \les \eta \E \Big(  \| X^+_{r/2} \|_{L^4}^4 +  \| z_{r/2, \ld} \|_{L^4}^4 +  \| W_2 \|_{L^4}^4   + \| \Xi \|_{L^4}^4  + \| \U \|_{L^4}^4  \Big) \\
& \les 1+ \Ld(r),
\end{split}
\]

\noi
which finishes the proof of \eqref{PW_PW2} with $F = - \eta \| u\|_{L^4(B_1 \backslash B_r)}^4 $, and thus the proof of the  proposition.
\end{proof}

{In the following, we show that $u\in L^{\frac{2d}{d-2}-} (B_1)$ almost surely with respect to the limiting measure $\rho_0$.

\begin{lemma}
\label{LEM:L4-}
For $d \ge 4$, let  $1 \le p < \frac{2d}{d-2}$. We have
\[
\E_{\rho_0} \bigg[\int_{B_1} |u|^p\bigg] < +\infty,
\]

\noi
where $\rho_0$ is a limiting measure constructed in Lemma \ref{LEM:limit}.
\end{lemma}

\begin{proof}
Fix $r \in (0,1)$.
The map $u \mapsto \exp \big( \eta (1+\Ld(r))^{-1} \|u\|^4_{L^4 (B_1 \backslash B_r)}\big)$ is non-negative and continuous on $\mathbb X^\al$.
Hence, by the portmanteau theorem \cite[Theorem 2.1]{Bil99}, applied to the bounded continuous functions obtained by truncating this map at level $M > 0$, together with \eqref{limit_rho}, Proposition \ref{PROP:PW} with $F = - \eta \| u\|_{L^4(B_1 \backslash B_r)}^4$, and monotone convergence as $M \to \infty$, we have, for any limit point $\rho_0$ as in Lemma \ref{LEM:limit},
\begin{align}\label{P5.5}
\int \exp\bigg( \eta  \frac{\| u\|_{L^4(B_1 \backslash B_r)}^4}{1+ \Ld(r)} \bigg) d \rho_0 (u) \le \liminf_{n \to \infty} \int \exp \bigg( \eta  \frac{\| u\|_{L^4(B_1 \backslash B_r)}^4}{1+ \Ld(r)} \bigg) d \rho_{\eps_n}(u) \les 1,
\end{align}

\noi
for any $r \in (0,1)$ and $\eta \ll 1$,
where $\Ld (\cdot)$ is given in \eqref{Ld}.
In particular, from the elementary inequality $x^q \les_q e^x$, we deduce that for every $q > 0$,
\begin{equation}
\label{Lp-1}
\begin{split}
\E_{\rho_0} \bigg|\int_{B_{2r}\backslash B_r} u^4\bigg|^q
& = \eta^{-q} (1+ \Ld(r))^q \E_{\rho_0} \bigg|\eta  (1+ \Ld(r))^{-1} \int_{B_{2r}\backslash B_r} u^4 \bigg|^q \\
& \les \eta^{-q}  (1+ \Ld(r))^q \int \exp\bigg( \eta  \frac{\| u\|_{L^4(B_1 \backslash B_r)}^4}{1+ \Ld(r)} \bigg) d \rho_0 (u) \\
& \les  (1+ \Ld(r))^q,
\end{split}
\end{equation}
for $r \le \frac12$.
By H\"older's inequality on the bounded domain $B_1$, it suffices to treat $2 \le p < \frac{2d}{d-2}$.
Given such a $p$,
let $\sigma \in (1,\infty)$ be such that $\frac{1}{p} = \frac{1}{4} + \frac{1}{\sigma}$.
Note that such $\sigma$ exists since $p < \frac{2d}{d-2} \le 4$ from the assumption $d\ge 4$.
By H\"older's inequality, \eqref{Lp-1}, and \eqref{Ld},
we conclude that
\begin{align}
\begin{split}
\label{Lp-2}
\E_{\rho_0} \big[  \|u\|_{L^p({B_{2r}\backslash B_r})}^{4q}\big]  & \les \E_{\rho_0} \big[ \|u\|_{L^4({B_{2r}\backslash B_r})}^{4q} \big]  \| \ind_{B_{2r}\backslash B_r} \|_{L^\sigma}^{4q} \\
& \les (1+ \Ld(r))^q r^{\frac{4qd}{\sigma}}  \\
& \les
\begin{cases}
( (1+|\ln r|) r^{\frac{4d}\sigma} )^q, & d = 4; \\
r^{(-d + 4 + \frac{4d}\sigma)q}, & d \ge 5.
\end{cases}
\end{split}
\end{align}
Since $\frac{4d}\sigma > 0$ and $4-d + \frac{4d}{\sigma} = 4 - 2d + \frac{4d}{p} > 0$,
we deduce from \eqref{Lp-2} with $q = \frac p4$ that there exists $\ta > 0$ such that
\begin{equation}
\label{Lp-3}
\E_{\rho_0} \big[  \|u\|_{L^p({B_{2r}\backslash B_r})}^{p} \big]  \les_p r^{p\ta}.
\end{equation}
By summing \eqref{Lp-3} over $r$ dyadically, we obtain that
\begin{align*}
    \E_{\rho_0} \big[ \|u\|_{L^p({B_{1}})}^{p} \big] & =  \sum_{N\ge 1 \text{ dyadic}}  \E_{\rho_0}  \big[ \|u\|_{L^p({B_{N^{-1}}\backslash B_{(2N)^{-1}}})}^p \big] \les \sum_{N \ge 1 \text{ dyadic}} N^{- p\ta} < \infty,
\end{align*}
which finishes the proof.
\end{proof}
}

Now we are ready to prove Theorem \ref{THM:main}\,(ii).

\begin{proof}[Proof of Theorem \ref{THM:main}\,(ii)]
The existence of a limit point $\rho_0$, as a Borel probability measure on the space $\mathbb X^\al$ of \eqref{Xp}, is established by Lemma \ref{LEM:limit}.
Applying Lemma \ref{LEM:L4-}, we further conclude that $\rho_0$ is concentrated in the space $L^{\frac{2d}{d-2}-}(B_1)$.
\end{proof}

\section{Non-triviality}
\label{SEC:sing}

In this section, we prove Theorem~\ref{THM:main_sub}, which establishes that the limiting measure $\rho_0$, a candidate for the radial $\Phi^4_d$ measure for $d \ge 4$,
is non-Gaussian and non-trivial.
This stands in sharp contrast to the results in \cite{Aiz,ADC21},
where the continuum limits of $\Phi^4_d$ lattice fields are shown to be Gaussian and free for $d \geq 4$.
Furthermore, we prove Theorem~\ref{THM:main_singular},
which shows that $\rho_0$ is singular with respect to the Gaussian measure $\mu$. We start by showing the non-triviality of the limiting measure.

\begin{proposition}[Non-triviality]
\label{PROP:nontriv}
Let $\dl \in (0,1)$, let $1 \le p < \infty$, and let $\rho_0$ be any limit point of the sequence $\rho_\eps$ in the sense of Lemma \ref{LEM:limit}.
Then, for every non-empty open set $U \subseteq L^p_{\rm rad} (B_1\backslash B_\dl)$, we have
\begin{align}
\label{nontriv}
\big( (R_\dl)_{\#} \rho_0 \big) (U) = \rho_0(R_\dl^{-1}(U))> 0 ,
\end{align}

\noi
where $L^p_{\rm rad}$ denotes the space of radial functions in $L^p$.
In particular, for every radial $\varphi \in C^\infty_c (B_1 \backslash B_\dl)$ with $\varphi \not\equiv 0$, the random variable $\jb{u, \varphi}$ is not $\rho_0$-almost surely constant.
\end{proposition}

\begin{proof}
Fix $\ld \in \R$.
By \eqref{mu_pm}, the measure $\mu^+_\dl[\ld]$ is the law of the Gaussian field $X^+_\dl$ shifted by $\ld h_\dl$, and, by \eqref{mu+}, its Cameron--Martin space is $H^1_{0,{\rm rad}} (B_1\backslash B_\dl)$.
By \eqref{X+estimate}, the variance $\E [ |X^+_\dl(r)|^2]$ is bounded on $[\dl,1]$, so that
\[
\E \Big[ \| X^+_\dl \|_{L^p (B_1\backslash B_\dl)}^p \Big] \les_p \int_{B_1\backslash B_\dl} \big( \E [ |X^+_\dl (x)|^2] \big)^{\frac p2} dx < \infty ,
\]

\noi
by Gaussian hypercontractivity.
Hence $\mu^+_\dl[\ld]$ is a Gaussian measure on the separable Banach space $L^p_{\rm rad} (B_1\backslash B_\dl)$, and, by the support theorem for Gaussian measures, see \cite[Theorem 3.6.1]{Bog98}, its topological support there is the closure of $\ld h_\dl + H^1_{0,{\rm rad}} (B_1\backslash B_\dl)$.
Since $p < \infty$, the radial functions in $C^\infty_c (B_1\backslash B_\dl)$ are dense in $L^p_{\rm rad} (B_1\backslash B_\dl)$ and are contained in $H^1_{0,{\rm rad}} (B_1\backslash B_\dl)$, so that this closure is the whole of $L^p_{\rm rad} (B_1\backslash B_\dl)$.
Consequently $\mu^+_\dl[\ld] (U) > 0$ for every non-empty open set $U \subseteq L^p_{\rm rad} (B_1\backslash B_\dl)$.
Moreover, $\rho_\dl^+[\ld]$ has a positive density with respect to $\mu_\dl^+[\ld]$ by \eqref{meas_cut}, so that the two measures are equivalent and $\rho_\dl^+[\ld] (U) > 0$ as well.
As this holds for every $\ld \in \R$, \eqref{eq:rhoMarkov0} yields
\[
\big( (R_\dl)_{\#} \rho_0 \big) (U) = \int \rho_\dl^+[\ld](U) \, \nu_{0,\dl} (d \lambda) > 0 ,
\]

\noi
which is \eqref{nontriv}.

As for the last assertion, suppose that $\jb{u,\varphi} = c$ holds $\rho_0$-almost surely, for some $c \in \R$.
Since $\varphi$ is supported in $B_1 \backslash B_\dl$, it follows from \eqref{eq:rhoMarkov0} that $\jb{u,\varphi} = c$ holds $\rho_\dl^+[\ld]$-almost surely, and hence $\mu_\dl^+[\ld]$-almost surely, for $\nu_{0,\dl}$-almost every $\ld$.
This is impossible, since under $\mu_\dl^+[\ld]$ the random variable $\jb{u,\varphi}$ is Gaussian with variance
$\jb{\varphi, (-\Delta)^{-1} \varphi}_{L^2 (B_1\backslash B_\dl)} > 0$,
where $(-\Delta)^{-1}$ is taken with Dirichlet boundary conditions on $\partial (B_1\backslash B_\dl)$.
\end{proof}

\subsection{Sub-Gaussian tails}
In this subsection,
we show that the measure $\rho_0$ constructed in the previous section has the following property: for every radial $\varphi \in C^\infty_c(B_1\backslash B_{\frac 34})$, the random variable $\jb{u,\varphi}$ has sub-Gaussian tails under $\rho_0$. Together with Proposition \ref{PROP:nontriv}, this shows that $\rho_0$ is not Gaussian.
We first prove the following uniform bound.

\begin{proposition}
\label{PROP:SG}
Let $\varphi \in C^\infty_c(B_1\backslash B_{\frac 34})$. Then there exists a positive $\eta = \eta(\varphi) \ll1$ such that, for every $0 <\eps \le \frac 12$,
\begin{equation}
\int \exp\big( \eta\cdot \jb{u,\varphi}^4 \big) d \rho_\eps(u) \le C,
\end{equation}

\noi
where the constant  $C = C(\varphi)$ is independent of $\eps$.
\end{proposition}
\begin{proof}

By Proposition \ref{prop:Markov} with $\dl = \frac12$, it is enough to show that there exist constants $\eta, C > 0$, depending only on $\varphi$, such that
$$
\int \exp\big( \eta \jb{u, \varphi}^4 \big) d \rho_{\frac 12}^+[\ld](u) \le C
$$
for every  $\ld \in \R$,
where $\rho_{\frac 12}^+[\ld]$ is given by \eqref{meas_cut}.
We first note that the left-hand side is finite for every $\ld \in \R$, provided $\eta \le (4 \|\varphi\|_{L^{4/3}}^4)^{-1}$, which we assume from now on.
Indeed, since $\supp \varphi \subset B_1 \backslash B_{\frac34}$, H\"older's inequality gives
\[
\eta \jb{u,\varphi}^4 \le \eta \|\varphi\|_{L^{4/3}}^4 \int_{B_1 \backslash B_{1/2}} |u|^4 dx \le \frac14 \int_{B_1 \backslash B_{1/2}} |u|^4 dx,
\]

\noi
so that, by \eqref{meas_cut}, $\int \exp ( \eta \jb{u, \varphi}^4 ) d \rho_{\frac 12}^+[\ld](u) \le 1 / Z^+_{1/2}[\ld] < \infty$.
Moreover, $u \mapsto - \eta \jb{u,\varphi}^4 + \frac14 \int_{B_1 \backslash B_{1/2}} |u|^4 dx$ is lower semicontinuous on $L^2 (B_1 \backslash B_{1/2})$.
Hence, by Proposition \ref{prop:rhoVar} with $F(u) = - \eta \jb{u,\varphi}^4$, we have
\begin{equation}
\label{eq:rhoVar1}
\begin{split}
\log & \int \exp(\eta \jb{u, \varphi}^4) d \rho_\frac12^+[\ld](u) \\
 & = \inf_{W \in  \mathbb H}
	\E\big[ P_{\frac12,\ld}(W) \big]
	- \inf_{W \in  \mathbb H}
	\E\big[ - \eta \jb{X^+_{1/2} + z_{\frac12,\ld} + W + \Xi + \U, \varphi}^4  + P_{\frac12,\ld}(W) \big]\\
 & \le 1 + \eta \E \big[ \jb{ X^+_{1/2}  + z_{\frac12,\ld} + W_2 + \Xi + \U, \varphi}^4 \big]\\
 & \le 1 + C \eta \E \big[ \jb{W_2, \varphi}^4 \big] + C \eta \E \big[ \jb{X^+_{1/2}(1), \varphi}^4 + \jb{z_{\frac12,\ld}, \varphi}^4 + \jb{\Xi, \varphi}^4 + \jb{\U, \varphi}^4  \big] ,
\end{split}
\end{equation}

\noi
where $W_2$ is the almost optimiser in the sense of \eqref{W2}, and $z_{\frac12,\ld}$ is given by \eqref{z}.
In what follows, we control the right-hand side of \eqref{eq:rhoVar1} term by term.
We first note that
\[
\begin{split}
\jb{z_{\frac12,\ld}, \varphi}^4 & = \bigg( \int z_{\frac12,\ld} \varphi dx \bigg)^4 \le C \bigg( \sup_{x \in B_1 \backslash B_{3/4}} |z_{\frac12,\ld} (x)| \bigg)^4 \les 1,
\end{split}
\]
where we used the fact that $\supp \varphi \subset B_1 \backslash B_{3/4}$ and Lemma \ref{LEM:z} with $\dl = \frac12$.
Similarly, we have
\[
\begin{split}
\E  \big[ \jb{X^+_{1/2}(1), \varphi}^4 \big]
& \les \Big( \E \big[ \jb{X^+_{1/2}(1), \varphi}^2 \big] \Big)^2  \\
& = \bigg( \iint \E \big( X^+_{1/2} (1) (x) X^+_{1/2} (1) (y)\big) \ph(x) \ph(y) dxdy \bigg)^2 \\
& \les \bigg( \iint \Big( \E \big[ (X^+_{1/2} (1) (x))^2 \big] + \E \big[ (X^+_{1/2} (1) (y))^2 \big] \Big) |\ph(x)| |\ph(y)| dxdy \bigg)^2 \\
& \les \bigg( \int \frac1{|x|^{d-2}} |\ph(x)| dx \bigg)^2 \les 1,
\end{split}
\]

\noi
where in the second to last step, we used Lemma \ref{LEM:X+}. A similar argument also yields
\[
\begin{split}
\E  \big[ \jb{\Xi (1), \varphi}^4 \big] \les \bigg( \iint \Big( \E \big[ (\Xi (1) (x))^2 \big] + \E \big[ (\Xi (1) (y))^2 \big] \Big) |\ph(x)| |\ph(y)| dxdy \bigg)^2 \les 1,
\end{split}
\]

\noi
where we used Lemma \ref{LEM:Xi} in the last step.
The term $\E \big[ \jb{\U, \varphi}^4 \big]$ is bounded in the same way, using Lemma \ref{LEM:Upsilon} instead of Lemma \ref{LEM:Xi}.

Therefore, it remains to control $\E \big[ \jb{W_2(1), \varphi}^4 \big] $ in a uniform way. Since $\supp \varphi \subset B_1 \backslash B_{\frac34}$, H\"older's inequality gives
\[
\E \big[ \jb{W_2(1), \varphi}^4 \big] \les \|\varphi\|_{L^{4/3}}^4 \, \E \bigg[ \int_{B_1\backslash B_{1/2}} W_2(1)^4 dx \bigg] .
\]

\noi
It thus suffices to show that
\[
\E \bigg[ \int_{B_1\backslash B_{1/2}} W_2(1)^4 dx \bigg] \les 1
\]

\noi
uniformly in $\ld$.
We note from \eqref{W2} that
\begin{align}
\label{G2}
\begin{split}
\E [ & - \eta \jb{X^+_{1/2} + z_{\frac12,\ld} + W_2 + \Xi + \U, \varphi}^4 + P_{1/2, \ld} (W_2)] \\
& \le 1 + \E [- \eta \jb{X^+_{1/2} + z_{\frac12,\ld} + \Xi + \U, \varphi}^4  + P_{1/2, \ld} (0)] \\
& \le 1 + \E [P_{1/2, \ld} (0)] \les 1,
\end{split}
\end{align}
where we used Lemma \ref{LEM:P0} with $\dl = 1/2$. On the other hand, we see that
\begin{align}
\label{G3}
\begin{split}
\E [ & - \eta \jb{X^+_{1/2} + z_{\frac12,\ld} + W_2 + \Xi + \U, \varphi}^4 + P_{1/2, \ld} (W_2)] \\
& \ge - C \eta \E [ \jb{W_2,\varphi}^4] +  \E [ P_{1/2, \ld} (W_2)] - C \eta \E [ \jb{X^+_{1/2} + z_{\frac12,\ld} + \Xi + \U, \varphi}^4] \\
& \ge - C \eta \E [ \jb{W_2, \varphi}^4]  + c_1 \E \bigg[ \int W_2^4 dx + \|W_2\|_{\dot H^1}^2 \bigg] - C \\
& \ge - C \eta \E \bigg[ \int W_2^4 dx  \bigg]+ c_1 \E \bigg[ \int W_2^4 dx + \|W_2\|_{\dot H^1}^2 \bigg] - C \\
& \ge \frac{c_1}2 \E \bigg[ \int W_2^4 dx + \|W_2\|_{\dot H^1}^2 \bigg] - C,
\end{split}
\end{align}
where we used Proposition \ref{PROP:opt} and chose $\eta < \frac{c_1}{2C}$, with $C$ depending only on $\varphi$ (in addition to $\eta \le (4 \|\varphi\|_{L^{4/3}}^4)^{-1}$).
By collecting \eqref{G2} and \eqref{G3}, we conclude that
\[
\E \bigg[ \int_{B_1 \backslash B_{1/2}} W_2^4 dx + \|W_2\|_{\dot H^1 (B_1 \backslash B_{1/2})}^2 \bigg] \les 1
\]
uniformly in $\ld$. Thus, we finish the proof.
\end{proof}

\subsubsection{Proof of Theorem \ref{THM:main_sub}}
Now, we are ready to prove Theorem \ref{THM:main_sub}.

\begin{proof}[Proof of Theorem \ref{THM:main_sub}]
Let $\{\eps_n\}_{n \ge 1}$ be a sequence with $\eps_n \to 0$ along which $\rho_{\eps_n} \rightharpoonup \rho_0$ in $\mathbb X^\al$, as in Lemma \ref{LEM:limit}.
Fix a radial $\varphi \in C^\infty_c(B_1\backslash B_{\frac 34})$ with $\varphi \not\equiv 0$, and let $\eta = \eta(\varphi)$ and $C = C(\varphi)$ be as in Proposition \ref{PROP:SG}.
Since $\varphi$ is supported in $B_1 \backslash B_{\frac34}$ and the restriction map $R_{3/4}$ in \eqref{R_res} is continuous, the map $u \mapsto \jb{u,\varphi}$ is continuous on $\mathbb X^\al$.
Hence, for every $M > 0$, the function $u \mapsto \min \{ \exp(\eta \jb{u,\varphi}^4), M \}$ is bounded and continuous on $\mathbb X^\al$, and the portmanteau theorem \cite[Theorem 2.1]{Bil99}, together with Proposition \ref{PROP:SG}, gives
\[
\int \min \big\{ \exp\big( \eta \jb{u,\varphi}^4 \big), M \big\} d \rho_0 (u) = \lim_{n \to \infty} \int \min \big\{ \exp\big( \eta \jb{u,\varphi}^4 \big), M \big\} d \rho_{\eps_n} (u) \le C .
\]

\noi
Letting $M \to \infty$, by monotone convergence we obtain
\[
\int \exp\big( \eta \jb{u,\varphi}^4 \big) d \rho_0 (u) \le C < \infty .
\]

\noi
In particular, by Markov's inequality, $\rho_0 ( \{ |\jb{u,\varphi}| > t \}) \le C e^{- \eta t^4}$ for every $t > 0$, that is, $\jb{u,\varphi}$ has sub-Gaussian tails under $\rho_0$.
Moreover, by Proposition \ref{PROP:nontriv} applied with $\dl = \frac 34$, the random variable $\jb{u,\varphi}$ is not $\rho_0$-almost surely constant.

It remains to show that $\rho_0$ is not Gaussian.
If it were, the continuous linear functional $u \mapsto \jb{u,\varphi}$ would be, under $\rho_0$, a Gaussian random variable $g$ with some mean $m \in \R$ and some variance $\s^2 \ge 0$, and $\s^2 > 0$ since $\jb{u,\varphi}$ is not $\rho_0$-almost surely constant.
But then
\[
\E \big[ \exp( \eta g^4 ) \big] = \frac{1}{\sqrt{2 \pi \s^2}} \int_\R \exp \Big( \eta x^4 - \frac{(x-m)^2}{2\s^2} \Big) dx = \infty,
\]

\noi
for every choice of  $\eta > 0$, independently of $m$, which contradicts the bound above.
Hence $\rho_0$ is not a Gaussian measure.
\end{proof}

\subsection{Singularity with respect to the Gaussian measure}
\label{SUB:sing}

As established in Theorem~\ref{THM:main_sub},
the Gibbs measure $\rho_0$ is not Gaussian: for every radial $\varphi \in C^\infty_c(B_1\backslash B_{\frac 34})$ with $\varphi \not\equiv 0$, the tails of $\jb{u,\varphi}$ under $\rho_0$ are  sub-Gaussian by Proposition~\ref{PROP:SG}, while $\jb{u,\varphi}$ is not $\rho_0$-almost surely constant by Proposition~\ref{PROP:nontriv}.
In this subsection, we further prove Theorem \ref{THM:main_singular}, which claims that $\rho_0$ is singular with respect to the Gaussian measure $\mu$.

The main reason we expect Theorem \ref{THM:main_singular} to hold is as follows.
If the measure $\rho_0$ were to be absolutely continuous with respect to $\mu$, then, defining $Z_\dl$ by
\begin{align*}
    \int \exp\Big(-\frac14 \int_{B_1\setminus B_\dl} |u|^4\Big) d\mu(u) = \exp(-Z_\dl),
\end{align*}
we would expect
$$ \frac{d\rho_\dl}{d \mu}=\exp\Big(-\frac14 \int_{B_1\setminus B_\dl} |u|^4 + Z_\dl \Big) \rightharpoonup \frac{d\rho_0}{d\mu}. $$
In particular, if this were the case, we would expect that, on a Borel set $E$ with $\mu(E) > 0$ on which $\frac{d\rho_0}{d\mu}$ is bounded away from $0$ and $\infty$, the random variables
\begin{align*}
    \frac14 \int_{B_1\setminus B_\dl} |u|^4 - Z_\dl, \qquad \dl \in (0,1),
\end{align*}
remain of order one as $\dl \to 0$, that is, that they form a tight family under $\mu( \, \cdot \, | E)$.
Therefore, in order to show that this is not the case, we aim to prove the non-concentration inequality
\begin{align}\label{e:nonConcentrationVague}
   \sup_{Z \in \R} \mu^+_\dl[\ld]\bigg( \bigg\{ \Big|\int_{B_1 \backslash B_\dl} |u|^4 dx - 4 Z\Big| = O(1)\bigg\}\bigg) = o(1)
\end{align}
as $\dl \to 0$, at least for typical values of $\ld$, i.e.\ $|\ld| \les \dl^{1 -\frac d2}$. As it turns out, the precise formulation of this non-concentration inequality strongly depends on the dimension $d$, with the critical case $d=4$ being significantly more involved.
We then split the proof of the inequality into the two cases $d>4$ and $d=4$, and proceed to show that this implies Theorem  \ref{THM:main_singular} via a soft argument.

\subsubsection{Non-concentration in high dimension}
We first consider the higher dimensional cases $d >4$.
Before presenting the main results, we recall the following harmonic analysis lemmas.

\begin{lemma}\label{l:VdC}
    Let $m \ge 2$, and let $H:\R \to \R$ be a smooth function such that $ H^{(m)} \ge 1. $
    Then there exists a constant $C = C(m)$, independent of $H$, such that, for every $\xi \in \R$,
    \begin{equation}
        \Big| \int_{\R} e^{iH(x)\xi} e^{-\frac{x^2}2} dx \Big| \le \frac{C}{\jb{\xi}^\frac1m}.
        \label{VdC}
    \end{equation}
\end{lemma}
\begin{proof}
    This is an immediate consequence of  van der Corput's lemma. See  \cite[Chapter VIII, Proposition 2 and its Corollary]{stein}.
\end{proof}

A key point of Lemma~\ref{l:VdC} is that the constant $C$ in~\eqref{VdC} depends only on a lower bound for the  $m$-th derivative of $H$.
This allows us to prove the following.

\begin{lemma}
Let $p : \R \to \R$ be a monic polynomial of degree $m \ge 2$.
Let $g \sim N(0,1)$.
Then, for every $\eps > 0$, we have that
\begin{equation} \label{e:IntervalMeasureBound}
    \sup_{z \in \R} \P \big( |p(g) - z| \le \eps \big) \les \eps^{\frac1m}
\end{equation}

\noi
uniformly in $p$.
\end{lemma}

\begin{proof}
Let $\chi \in C_c^\infty(\R)$ be a non-negative function with $\chi = 1$ on $[-1,1]$.
Fix $z \in \R$, and define  $\chi_{\eps,z}(x) = \chi (\frac{x-z}\eps)$, so that $\ind_{[z - \eps, z + \eps]} \le \chi_{\eps,z}$.
We observe that by hypothesis, the function
$H(x) = p(x) -z  $
satisfies
$$ H^{(m)}(x) = m! \ge 1. $$
Then, by Lemma \ref{l:VdC}, we have
\begin{align}
\begin{split}
\P \big( |p(g) - z| \le \eps \big)  &\le \E [\chi_{\eps,z} (p(g))]\\
& = \int_\R \ft {\chi_{\eps,z}} (\xi) \E [e^{i p(g) \cdot \xi}] d\xi\\
& = \eps \int_\R \ft \chi (\eps \xi) e^{-iz\xi} \E [e^{i p(g) \cdot \xi}] d\xi \\
& = \int_\R \ft \chi (\xi) \E [e^{i (-z + p(g)) \frac{\xi}\eps}] d\xi\\
& = \frac{1}{\sqrt{2\pi}}\int_\R \ft \chi (\xi) \cdot \bigg( \int_\R e^{i (-z + p(x)) \frac{\xi}\eps} e^{-\frac{x^2}2} dx \bigg) d\xi\\
& \les\eps^\frac1m \int_\R  \frac{|\ft \chi (\xi)|}{|\xi|^\frac1m }d\xi \\
&\les \eps^\frac1m.
\end{split}
\end{align}
Here we used that $\frac{|\ft \chi (\xi)|}{|\xi|^\frac1m }$ is integrable since $\ft \chi \in L^1(\R)$ (since $\chi$ is smooth), and since $m \ge 2$,
$$\int_{-1}^1 \frac1{|\xi|^\frac1m } d\xi < \infty. $$
Since the implicit constants depend only on $m$ and $\chi$, taking the supremum over $z \in \R$ yields \eqref{e:IntervalMeasureBound}.
\end{proof}

Recall the Gaussian measure $\mu_\dl^+[\ld]$ given in \eqref{mu_pm}. We have the following.

\begin{proposition}[Non-concentration inequality, $d>4$]
\label{PROP:div_h}
Let $Z \in \R, \eta >0$.
For every $\dl < \frac13$ and every $\ld \in \R$, we have that
\begin{equation}
\label{eq:fluctuations_h}
\mu_\dl^+[\ld]
\bigg(\bigg\{ \Big| \frac 14 \int_{B_1\backslash B_\dl} |u|^4 - Z\Big| \le \eta \dl^{4-d} \bigg\}\bigg) \les \eta^{\frac14},
\end{equation}
uniformly in  $Z \in \R$ and $\ld \in \R$.
\end{proposition}

\begin{proof}
Take a radial test function $v \not\equiv 0$ such that $\supp(v) \subset B_3 \backslash B_2$ and
$\| v \|_{\dot H^1} = 1$, and set $\kappa := \|v\|_{L^4}^4 \in (0,\infty)$.
Define $v_\dl (x) = \dl^{1-\frac{d}2} v(x/
\dl)$,
which satisfies
\begin{align}
\label{v_dl}
\| v_\dl\|_{\dot H^1} = 1, \quad \|v_\dl\|_{L^4}^4 = \kappa \dl^{4-d}, \quad  \supp(v_\dl) \subset B_{3\dl} \backslash B_{2\dl}.
\end{align}

\noi
For $\delta < \frac13$, we decompose $X^+_\dl \sim \mu_\dl^+ $ in \eqref{mu+} as
\[
X^+_\dl =  g v_\dl + Y,
\]
where $g = \jb{X^+_\dl,v_\dl}_{\dot H^1} \sim N(0,1)$. Then, for $u \sim \mu_\dl^+[\ld]$,
we write
\[
u = X^+_\dl + \ld h_\dl = g v_\dl + Y + \ld h_\dl,
\]
where $g, Y$ are independent Gaussians.
We write
\begin{align}
\begin{split}
\int |g v_\dl + Y + \ld h_\dl|^4 & = g^4 \int v_\dl^4 + 4g^3 \bigg( \int v_\dl^3 (Y+ \ld h_\dl) \bigg) + 6 g^2 \bigg(\int v_\dl^2  |Y + \ld h_\dl|^2 \bigg)  \\
& \hphantom{X} + 4g \bigg( \int v_\dl  (Y + \ld h_\dl)^3 \bigg) + \int |Y + \ld h_\dl|^4 \\
&  =: \kappa \dl^{4-d} \mathcal P_{\ld,\dl,Y} (g),
\end{split}
\end{align}

\noi
where, by \eqref{v_dl}, $\mathcal P_{\ld,\dl,Y}$ is a monic polynomial of degree $4$, whose coefficients depend on $Y$.
Therefore, recalling that $g, Y$ are independent, we have that
\begin{align*}
    \mu_\dl^+[\ld]
\bigg(\bigg\{ \Big| \frac 14 \int_{B_1\backslash B_\dl} |u|^4 - Z\Big| \le \eta \dl^{4-d} \bigg\}\bigg) &= \E_Y\big[ \P_g \big(\big\{ \big|\mathcal P_{\ld,\dl,Y} (g) - 4\kappa^{-1}\dl^{d-4}Z\big| \le 4\kappa^{-1}\eta\big\} \big)\big] \\
&\les \E_Y\big[(\kappa^{-1}\eta)^\frac14\big] \les \eta^\frac14,
\end{align*}
where we used \eqref{e:IntervalMeasureBound} with $m = 4$ to derive the first inequality, and in the last step that $\kappa > 0$ depends only on the fixed function $v$.
\end{proof}

\begin{remark}

Proposition~\ref{PROP:div_h} shows that, in dimensions $d > 4$, the law of $\|u\|_{L^4(B_1 \setminus B_\delta)}^4$ under $\mu_\delta^+[\lambda]$ does not concentrate on any interval of length $O(\dl^{4-d})$.
This suggests the non-concentration of the limiting measure as $\delta \to 0$.

\end{remark}

\subsubsection{Non-concentration in the critical dimension}

We now move to the case $d=4$.
We note that in this case, while one does actually have \eqref{eq:fluctuations_h}, this is not helpful in proving a non-concentration inequality of the form \eqref{e:nonConcentrationVague}, since we do not have any diverging term (in $\dl$) in the event on the  left-hand side. Nevertheless, up to refining the assumptions of the inequality, we can still obtain a logarithmically diverging term.

\begin{proposition}[Non-concentration inequality, $d=4$]
\label{PROP:div}
Let $Z \in \R$, $0 < \eta \le 1$, and $A \ge 1$. There exists $\dl_0 = \dl_0(A,\eta)$ such that for every $\dl \le \dl_0$ and every $|\ld| \le A\dl^{-1}$, we have that
\begin{equation} \label{eq:fluctuations}
\mu_\dl^+[\ld]\bigg(\bigg\{ \Big| \frac 14 \int_{B_1\backslash B_\dl} |u|^4 - Z\Big| \le \eta |\log \dl|^\frac12 \bigg\}\bigg) \les \eta,
\end{equation}
uniformly in $Z \in \R$ and $\ld$.
\end{proposition}

Proposition \ref{PROP:div} shows that the law of $\|u\|_{L^4(B_1\backslash B_\dl)}^4$ under $\mu_\dl^+[\ld]$ does not concentrate on any interval of length $O(|\log \dl|^{\frac12})$ as $\dl \to 0$. The window $|\log\dl|^{\frac12}$ diverges, which suggests that the non-trivial portion of the measure escapes to infinity.

Before the proof of Proposition \ref{PROP:div}, we need some preparation.

We further decompose $X^+_\dl$, given by \eqref{mu+}, as follows.
Let $u_1 \in C_c^\infty (B_1)$ be a radial test function with $\supp u_1 \subset B_1 \backslash B_{1/2}$ and $\|u_1\|_{\dot H^1} = 1$, and set $\kappa := \| u_1\|_{L^4}^4 \in (0,\infty)$.
Let $J := \lfloor \log_2 (1/\dl) \rfloor$, so that $2^{-J} \ge \dl$ and $J \sim |\log \dl|$.
Then we define
\begin{align} \label{uj}
u_j (r) = 2^{{j-1}} u_1 (2^{j-1} r)
\end{align}

\noi
for $1 \le j \le J$.
Since $d = 4$, we have $\|u_j\|_{\dot H^1} = 1$ and $\|u_j\|_{L^4}^4 = \kappa$, while $\supp u_j \subset B_{2^{-j+1}} \backslash B_{2^{-j}} \subset B_1 \backslash B_\dl$.
Defining $g_j = \jb{X^+_\dl,u_j}_{\dot H^1}$, we can write
\begin{align}\label{X+dec}
X^+_\dl = \sum_{j \le J} g_j u_j + Y,
\end{align}
where $g_j$ are i.i.d.\ standard Gaussians, and $Y$ is orthogonal to ${\rm span}\, (u_j)_{j\le J}$ in $\dot H^1$. Moreover, we claim that $Y$ is independent from $\{g_j\}_{j \le J}$.
To see this, we compute  the correlation between $g_j$ and $Y$. More precisely, since $(Y, (g_j)_{j \le J})$ is a Gaussian vector, independence is equivalent to showing
that for every test function $\varphi$,
we have $\E [g_j \jb{Y,\varphi}_{\dot H^1}] = 0$, for all $j \le J$.
Recalling that $g_j = \jb{X^+_\dl,u_j}_{\dot H^1}$, we have
\[
\begin{split}
\E [g_j \jb{Y,\varphi}_{\dot H^1}] & = \E [\jb{X^+_\dl,u_j}_{\dot H^1} \jb{Y,\varphi}_{\dot H^1}] \\
& = \E [\jb{X^+_\dl,u_j}_{\dot H^1} \jb{X^+_\dl,\varphi}_{\dot H^1}] - \sum_{i\le J}  \E [\jb{X^+_\dl,u_j}_{\dot H^1} g_i \jb{u_i,\varphi}_{\dot H^1}] \\
& = \jb{u_j, \varphi}_{\dot H^1} - \jb{u_j,\varphi}_{\dot H^1} = 0.
\end{split}
\]
 This proves the  claim.

With the above notation, we see that for $u$ distributed as $\mu_\dl^+ [\ld]$, we have
\[
u = X^+_\dl + \ld h_\dl = \sum_{j \le J} g_j u_j + Y + \ld h_\dl ,
\]

\noi
where $h_\dl$ is defined in \eqref{hdelta}.
Moreover, we can write
\begin{align}\label{decom}
\int_{B_1\backslash B_\dl} |u|^4 = \sum_{j \le J} \int_{B_{2^{-j+1}}\backslash B_{2^{-j}}} \Big(g_j u_j + (Y+ \ld h_\dl)\ind_{B_{2^{-j+1}}\backslash B_{2^{-j}}} \Big)^4 + \int_{B_{2^{-J}}\backslash B_\dl} |Y + \ld h_\dl|^4,
\end{align}

\noi
for $u \sim \mu_\dl^+[\ld]$.
For simplicity of notation, let
\begin{align}
\begin{split}
G_j(y) & := \int \Big(g_j u_j + (y + \ld h_\dl)\ind_{B_{2^{-j+1}}\backslash B_{2^{-j}}} \Big)^4 \\
& \hphantom{XX} - \E \bigg[\int \Big(g_j u_j + (y + \ld h_\dl) \ind_{B_{2^{-j+1}}\backslash B_{2^{-j}}} \Big)^4 \bigg]  ,
\end{split}
\label{Gj}
\end{align}

\noi
for $1 \le j \le J$, and
\begin{align}
\label{G0}
G_0 (y) := \int_{B_{2^{-J}}\backslash B_\dl} |y + \ld h_\dl|^4 dx ,
\end{align}

\noi
which is a deterministic function of $y$.

With the above notation, we have the following lower bound.
\begin{lemma}
\label{LEM:bdG}
With $G_j(y)$ defined in \eqref{Gj}, we have
\begin{align}
\label{Gj_l}
\E\big[|G_j (y)|^2\big] \ge 96 \kappa^2,
\end{align}

\noi
provided $j \le J$.
\end{lemma}

\begin{proof}
Let
\[
y_j = (y + \ld h_\dl) \ind_{B_{2^{-j+1}}\backslash B_{2^{-j}}}.
\]
From the definition of $G_j$ in \eqref{Gj} and $u_j$ in \eqref{uj}, we have
\begin{align}
\label{Gj_lower0}
\begin{split}
\E\big[|G_j (y)|^2\big] & =  \E \bigg( \int \big (g_j u_j + y_j \big)^4 \bigg)^2  - \bigg( \E \bigg[\int \big (g_j u_j + y_j \big)^4 \bigg] \bigg)^2\\
& = \E \bigg[ \bigg( \int g_j^4 u_j^4  + 4g_j^3 u_j^3 y_j + 6 g_j^2 u_j^2 y_j^2  + 4g_j u_j y_j^3 + y_j^4 dx\bigg)^2 \bigg] \\
& \hphantom{XXXX} - \bigg( \E \int g_j^4 u_j^4  + 6 g_j^2 u_j^2 (x) y_j^2 + y_j^4 dx \bigg)^2\\
& \ge \E \bigg[ \bigg( \int g_j^4 u_j^4  + 6 g_j^2 u_j^2 (x) y_j^2 + y_j^4 dx\bigg)^2 \bigg] \\
& \hphantom{XXXX} - \bigg( \E \int g_j^4 u_j^4  + 6 g_j^2 u_j^2 y_j^2   + y_j^4 dx \bigg)^2,
\end{split}
\end{align}

\noi
where we used the fact that $\E [g_j^k] = 0$ for odd $k$.
Then we expand the right-hand side of \eqref{Gj_lower0} and use the fact ${\rm Var} [X] = \E[X^2] - (\E X)^2 \ge 0$ to get
\begin{align}
\label{Gj_lower1}
\begin{split}
\textup{RHS of \eqref{Gj_lower0}} &= {\rm Var} \bigg[ \int g_j^4 u_j^4 \bigg] + {\rm Var} \bigg[ 6 \int g_j^2 u_j^2 y_j^2 \bigg] + {\rm Var} \bigg[\int  y_j^4 dx \bigg]  \\
& \hphantom{XX} + 2\E \bigg[ \bigg( \int g_j^4 u_j^4 \bigg) \bigg( \int 6 g_j^2 u_j^2 y_j^2 \bigg) + \bigg( \int g_j^4 u_j^4 \bigg) \bigg( \int y_j^4 \bigg)  \\
& \hphantom{XXXX} + \bigg( \int 6 g_j^2 u_j^2 y_j^2 \bigg) \bigg( \int y_j^4 \bigg) \bigg]\\
& \hphantom{XX} - 2\bigg[ \bigg( \E \int g_j^4 u_j^4 \bigg) \bigg(  \E \int 6 g_j^2 u_j^2 y_j^2 \bigg) + \bigg( \E  \int g_j^4 u_j^4 \bigg) \bigg(  \E \int y_j^4 \bigg)  \\
& \hphantom{XXXX} + \bigg(\E  \int 6 u_j^2 g_j^2 y_j^2 \bigg) \bigg(\E \int y_j^4 \bigg) \bigg].
\end{split}
\end{align}

\noi
We note that
\[
\begin{split}
\E & \bigg[ \bigg( \int g_j^4 u_j^4 \bigg) \bigg( \int 6 g_j^2 u_j^2 y_j^2 \bigg) \bigg] - \bigg( \E \int g_j^4 u_j^4 \bigg) \bigg(  \E \int 6 g_j^2 u_j^2 y_j^2 \bigg) \\
& = 6 \|u_j\|_{L^4}^4 \big(\E [g_j^6] - \E [g_j^4] \cdot \E [g_j^2] \big) \cdot \int u_j^2 y_j^2 \ge 0.
\end{split}
\]

\noi
Similarly, we note that
\[
\E \bigg[  \bigg( \int g_j^4 u_j^4 \bigg) \bigg( \int y_j^4 \bigg)  \bigg]- \bigg( \E  \int g_j^4 u_j^4 \bigg) \bigg(  \E \int y_j^4 \bigg) \ge 0
\]
and
\[
\E \bigg[ \bigg( \int 6 g_j^2 u_j^2 y_j^2 \bigg) \bigg( \int y_j^4 \bigg) \bigg] - \bigg( \E \int 6 g_j^2 u_j^2 y_j^2 \bigg) \bigg( \E \int y_j^4 \bigg) \ge 0 .
\]

\noi
The above computation with \eqref{Gj_lower0} and \eqref{Gj_lower1} yields
\begin{align}
\label{Gj_lower}
\begin{split}
\E\big[|G_j (y)|^2\big] & \ge {\rm Var} \bigg[ \int g_j^4 u_j^4 \bigg] + {\rm Var} \bigg[ 6 \int g_j^2 u_j^2 y_j^2 \bigg] + {\rm Var} \bigg[\int  y_j^4 dx \bigg] \\
& \ge {\rm Var} \bigg[ \int g_j^4 u_j^4 \bigg]  =  \kappa^2 {\rm Var} \big[ g_j^4 \big]  = 96 \kappa^2,
\end{split}
\end{align}

\noi
where we used $\|u_j\|_{L^4}^4 = \kappa$ and ${\rm Var} [g_j^4] = \E[g_j^8] - (\E [g_j^4])^2 = 105 - 9$.
\end{proof}

In what follows, we split the summation in \eqref{decom} into different sets.
For $y \in L^4(B_1\backslash B_\dl)$, we denote by
\begin{align}\label{My}
{\rm M}(y,\dl) :=  \frac{\| y + \ld h_\dl \|_{L^4(B_1\backslash B_\dl)}^4}{|\log \dl|},
\end{align}
and
\begin{align}\label{Iy}
{\rm I}(y,\dl) := \big \{ 1\le j \le J: \|y_j\|_{L^4}^4 \le 4 {\rm M}(y,\dl) \big \},
\end{align}

\noi
where
\[
y_j = (y + \ld h_\dl) \ind_{B_{2^{-j+1}}\backslash B_{2^{-j}}}.
\]

\noi

Since $\sum_{j \le J} \|y_j\|_{L^4}^4 \le \|y + \ld h_\dl\|_{L^4(B_1\backslash B_\dl)}^4 = {\rm M}(y,\dl) |\log \dl|$, Chebyshev's inequality gives $|{\rm I}(y,\dl)^c| \le \frac14 |\log \dl|$, and hence, since $J \ge \frac{|\log \dl|}{\log 2} - 1$, we have $|{\rm I}(y,\dl)| \ge J - \frac14 |\log \dl| \ge \frac16 |\log \dl|$ for $\dl \le e^{-1}$.

Since $(a+b)^4 \le 8 (a^4 + b^4)$, $\|u_j\|_{L^4}^4 = \kappa$, and $\|y_j\|_{L^4}^4 \le 4 {\rm M}(y,\dl)$ for $j \in {\rm I}(y,\dl)$, we deduce, for $j \in {\rm I}(y,\dl)$,
\begin{align*}
\E\big[|G_j (y)|^2\big] & \les \E \big( g_j^4 \|u_j\|_{L^4}^4 + \|y_j \|_{L^4}^4 \big)^2 \les 1 +  {\rm M}(y,\dl)^{8},\\
\E\big[|G_j (y)|^3\big] & \les \E \big( g_j^4 \|u_j\|_{L^4}^4 + \|y_j \|_{L^4}^4 \big)^3 \les 1 +  {\rm M}(y,\dl)^{12},
\end{align*}

\noi
where the implicit constants depend only on $\kappa$.

\noi
In what follows, we denote by
\begin{align}\label{Gbar}
\cj G(y) := \sum_{j\in {\rm I}(y,\dl) } G_j(y),
\end{align}

\noi
where $G_j(y)$ is defined in \eqref{Gj}, ${\rm I}(y,\dl)$ is given in \eqref{Iy}, and $y$ is a function such that $y \in L^4(B_1\backslash B_\dl)$.
Then, we have the following estimate.

\begin{lemma}
\label{LEM:Gbar}
Let $\dl \le e^{-1}$, let $\cj G (y)$ be given by \eqref{Gbar}, and let ${\rm M}(y,\dl)$ be given in \eqref{My}. Then, for every $\eta > 0$, we have
\begin{align}
\label{es:Gbar}
\P \big(  |\cj G(y) - r| \le \eta |\log \dl|^\frac 12 \big ) \les \eta + \exp\bigg(- \frac{ c \eta^2 |\log \dl|}{(1 + {\rm M}(y,\dl))^{24}}\bigg),
\end{align}

\noi
uniformly in $r \in \R$ and $y \in L^4 (B_1 \backslash B_\dl)$, where $c > 0$ and the implicit constant depend only on $u_1$.
\end{lemma}

\begin{proof}
We first claim that there exists a constant $c_0 \in (0,1)$, depending only on $\kappa$, such that, for every $j \in {\rm I}(y,\dl)$,
\begin{align}
\label{Gj_chf}
\big| \E \big[ \exp( i \xi G_j(y)) \big] \big| \le 1 - c_0 \xi^2
\qquad \text{ for } |\xi| \le c_0 \big( 1 + {\rm M}(y,\dl)\big)^{-12}.
\end{align}

\noi
Indeed, since $\E [G_j(y)] = 0$ by \eqref{Gj}, a second order Taylor expansion of the exponential gives
\[
\big| \E \big[ \exp( i \xi G_j(y)) \big] \big|
\le \bigg| 1 - \frac{\xi^2}2 \E \big[ |G_j(y)|^2\big] \bigg| + \frac{|\xi|^3}6 \E \big[ |G_j(y)|^3\big] .
\]

\noi
By Lemma \ref{LEM:bdG} we have $\E [|G_j(y)|^2] \ge 96 \kappa^2$, while the moment bounds above give
$\E [|G_j(y)|^2] \les 1 + {\rm M}(y,\dl)^{8}$ and $\E [|G_j(y)|^3] \les 1 + {\rm M}(y,\dl)^{12}$.
Hence, for $|\xi|$ in the range stated in \eqref{Gj_chf} with $c_0$ small enough, depending only on $\kappa$, the first term on the right-hand side equals $1 - \frac{\xi^2}2 \E [|G_j(y)|^2] \le 1 - 48 \kappa^2 \xi^2$ and the second one is at most $24 \kappa^2 \xi^2$, which proves \eqref{Gj_chf} if moreover $c_0 \le 24 \kappa^2$.

Given any $r \in \R$, by Markov's inequality, \eqref{Gbar}, the independence of $\{g_j\}_{j \le J}$, and the Fourier transform, we have that
\begin{align}
\label{sing1}
\begin{split}
\P \big( & |\cj G(y) - r| \le \eta |\log \dl|^\frac 12 \big )  \le e \, \E\bigg[ \exp\Big(-\frac{(\cj G(y) - r)^2}{2\eta^2 |\log\dl|}\Big) \bigg] \\
& \sim \eta |\log \dl|^\frac12 \int \exp\Big(-\frac{\eta^2 |\log\dl| \xi^2}{2}\Big) e^{-i\xi r} \E\big[\exp(i\xi \cj G(y))\big] d\xi \\
& \le \eta |\log \dl|^\frac12 \int \exp\Big(-\frac{\eta^2 |\log\dl| \xi^2}{2}\Big) \prod_{j\in {\rm I}(y,\dl)} \Big| \E\big[\exp(i\xi G_j(y))\big] \Big| d\xi .
\end{split}
\end{align}

\noi
We split the last integral according to whether $|\xi| \le c_0 (1 + {\rm M}(y,\dl))^{-12}$ or not.
On the first region, \eqref{Gj_chf} together with $|{\rm I}(y,\dl)| \ges |\log \dl|$ gives
\[
\prod_{j\in {\rm I}(y,\dl)} \Big| \E\big[\exp(i\xi G_j(y))\big] \Big| \le \big( 1 - c_0 \xi^2\big)^{|{\rm I}(y,\dl)|} \le \exp\big(- c \, \xi^2 |\log \dl|\big),
\]

\noi
so that, after the change of variables $\xi \mapsto \frac{\xi}{\eta |\log \dl|^{1/2}}$,
\begin{align}
\label{sing01}
\begin{split}
\eta |\log \dl|^\frac12 & \int_{|\xi| \le c_0 (1 + {\rm M}(y,\dl))^{-12}} \exp\Big(-\frac{\eta^2 |\log\dl| \xi^2}{2}\Big) \exp\big(- c\, \xi^2 |\log \dl|\big) d\xi \\
& \le \int_\R \exp\Big(- \frac{c\, \xi^2}{\eta^2}\Big) d\xi \les \eta .
\end{split}
\end{align}

\noi
On the second region we simply bound the product by $1$ and obtain, after the same change of variables,
\begin{align}
\label{sing02}
\begin{split}
\eta |\log \dl|^\frac12 & \int_{|\xi| > c_0 (1 + {\rm M}(y,\dl))^{-12}} \exp\Big(-\frac{\eta^2 |\log\dl| \xi^2}{2}\Big) d\xi \\
& = \int_{|\xi| > c_0 \eta |\log \dl|^\frac12 (1 + {\rm M}(y,\dl))^{-12}} \exp\Big(-\frac{\xi^2}{2}\Big) d\xi
\les \exp\bigg(- \frac{ c \eta^2 |\log \dl|}{(1 + {\rm M}(y,\dl))^{24}}\bigg) .
\end{split}
\end{align}

\noi
Collecting \eqref{sing1}, \eqref{sing01}, and \eqref{sing02}, we obtain
\begin{align}
\label{sing2}
\P \big(  |\cj G(y) - r| \le \eta |\log \dl|^\frac 12 \big ) \les \eta + \exp\bigg(- \frac{ c \eta^2 |\log \dl|}{(1 + {\rm M}(y,\dl))^{24}}\bigg) ,
\end{align}

\noi
which is \eqref{es:Gbar}. Thus, we finish the proof.
\end{proof}

Note that in Lemma \ref{LEM:Gbar}, $y$ is deterministic as far as the probability $\P$ is concerned.
In what follows, we shall consider the estimate with $y$ replaced by the random field $Y$.
From the definition of $Y$ in \eqref{X+dec}, one has that
\[
\begin{split}
\E[{\rm M}(Y,\dl)]  & = \frac1{|\log \dl|} \int_{B_1 \backslash B_\dl } \E \big[(Y   + \ld h_\dl)^4\big] dx \\
&  = \frac1{|\log \dl|} \int_{B_1 \backslash B_\dl } \E \bigg[ \bigg( X^+_\dl - \sum_j g_j u_j + \ld h_\dl\bigg)^4 \bigg] dx .
\end{split}
\]

In particular, we have the following estimate.

\begin{lemma}
\label{LEM:Gbar2}
Let $0 < \eta \le 1$, $A \ge 1$, and $|\ld| \le A \dl^{-1}$. With the above notations, there exists $\dl_0 = \dl_0 (A, \eta) \in (0, e^{-1}]$ such that we have
\begin{align}\label{MY2}
\E\bigg[ \exp\Big(-c \frac{\eta^2 |\log \dl|}{(1 + {\rm M}(Y,\dl))^{24}}\Big)\bigg] \les \eta,
\end{align}

\noi
for $\dl \le \dl_0$, where ${\rm M}(Y,\dl)$ is defined in \eqref{My}.
\end{lemma}

\begin{proof}
By using H\"older inequality and disjointness of the supports of the $u_j$, we note that
\begin{align}
\label{MY}
\begin{split}
\E[{\rm M}(Y,\dl)]
& \les \frac1{|\log \dl|} \int_{B_1 \backslash B_\dl } \bigg( \E \big[(X^+_\dl)^4 \big] +  \E \bigg[ \Big( \sum_j g_j u_j \Big)^4 \bigg] + (\ld h_\dl)^4 \bigg)\\
& = \frac1{|\log \dl|} \int_{B_1 \backslash B_\dl } \bigg( \E \big[(X^+_\dl)^4 \big] +  \sum_{j} \E ( g_j u_j )^4  + (\ld h_\dl)^4 \bigg) \\
& \les \frac{\Lambda (\dl)}{|\log \dl|} + \frac{3 J \kappa}{|\log \dl|} + \ld^4 \frac{\dl^{8} }{|\log \dl|} \int_{\dl }^1 r^{-8} r^3 dr  \\
& \les 1 + \ld^4 \dl^4 \les 1 + A^4,
\end{split}
\end{align}

\noi
where we used \eqref{eq:X+L4}, $\|u_j\|_{L^4}^4 = \kappa$, $J \le \frac{|\log \dl|}{\log 2}$, \eqref{hdelta}, and $|\ld| \le A \dl^{-1}$.

Finally, by using Markov's inequality and \eqref{MY}, together with $1 + A^4 \le 2 A^4$ (recall that $A \ge 1$), we obtain
\begin{align}
\label{MY1}
\begin{split}
\E\bigg[ & \exp\Big(-c \frac{\eta^2 |\log \dl|}{(1 + {\rm M}(Y,\dl))^{24}}\Big)\bigg] \\
& \les  \exp \big(-2^{-24} c \eta^2 |\log \dl| \big)  +  \sum_{N \ge 1\text{ dyadic}} \exp \bigg(-c \frac{\eta^2 |\log \dl|}{(1 + 2N)^{24}}\bigg) \cdot \P \big(\{{\rm M}(Y,\dl) \ge N\} \big) \\
& \les \exp \big(-2^{-24} c \eta^2 |\log \dl| \big)  + \sum_{N \ge 1\text{ dyadic}} \exp \bigg(-c \frac{\eta^2 |\log \dl|}{(1 + 2N)^{24}}\bigg) \cdot A^4 N^{-1} \\
& \les \exp \big(-2^{-24} c \eta^2 |\log \dl| \big)  + \sum_{N \ge 1\text{ dyadic}}  \frac{(1 + 2N)^\frac12}{(\eta^2 |\log \dl|)^{\frac1{48}}} \cdot A^4 N^{-1} \\
&\les A^4 (\eta^2 |\log \dl|)^{-\frac1{48}} ,
\end{split}
\end{align}

\noi
where we used the fact $e^{-x} \les x^{-\frac1{48}}$ for $x > 0$, and that $\sum_{N \ge 1 \text{ dyadic}} N^{-\frac12} < \infty$. By choosing
\begin{align*}
    \dl_0 = \dl_0 (A, \eta) : = \exp\big( - A^{192} \eta^{-50} \big),
\end{align*}

\noi
which satisfies $\dl_0 \le e^{-1}$ since $A \ge 1$ and $\eta \le 1$, we have $A^4 (\eta^2 |\log \dl|)^{-\frac1{48}} \le \eta$ for every $\dl \le \dl_0$, and the proof is concluded by \eqref{MY1}.
\end{proof}

With these preparations, we are ready to prove Proposition \ref{PROP:div}.

\begin{proof}[Proof of Proposition \ref{PROP:div}]
Let $\dl \le \dl_0 (A, \eta)$, with $\dl_0 (A,\eta) \le e^{-1}$ given by Lemma \ref{LEM:Gbar2}, and let $|\ld| \le A \dl^{-1}$.
For $1 \le j \le J$, let
\[
m_j (y) := \E \bigg[\int \Big(g_j u_j + (y + \ld h_\dl) \ind_{B_{2^{-j+1}}\backslash B_{2^{-j}}} \Big)^4 \bigg]
\]

\noi
be the expectation subtracted in \eqref{Gj}, which is a deterministic function of $y$.
For $u = \sum_{j \le J} g_j u_j + Y + \ld h_\dl$ distributed according to $\mu_\dl^+[\ld]$, \eqref{decom}, \eqref{Gj}, and \eqref{G0} give
\[
\frac 14 \int_{B_1\backslash B_\dl} |u|^4 - Z = \frac 14 \bigg( \sum_{j=0}^J G_j(Y) - Z_Y \bigg),
\]
where
$$ Z_Y := 4 Z - \sum_{j=1}^J m_j(Y).$$

\noi
Hence, splitting the sum in $j$ according to whether $j \in {\rm I}(Y,\dl)$ or not (note that $j = 0 \notin {\rm I}(Y,\dl)$), and setting $r_Y := Z_Y - \sum_{j \notin {\rm I}(Y,\dl)} G_j(Y)$, we obtain
\begin{align*}
\mu_\dl^+[\ld] \bigg(\bigg\{ \Big| \frac 14 \int_{B_1\backslash B_\dl} |u|^4 - Z\Big| \le \eta |\log \dl|^\frac12 \bigg\}\bigg)
& = \P\Big( \big| \cj G(Y) - r_Y \big| \le 4\eta |\log \dl|^\frac12 \Big) \\
& = \E\Big[ \P \Big( \big| \cj G(Y) - r_Y \big| \le 4\eta |\log \dl|^\frac12 \, \Big| \, Y, (g_j)_{j \notin {\rm I}(Y,\dl)} \Big) \Big] .
\end{align*}

\noi
Since $Y$ and $\{g_j\}_{1 \le j \le J}$ are independent and the $g_j$ are i.i.d., conditionally on $Y$ and $(g_j)_{j \notin {\rm I}(Y,\dl)}$ the quantity $r_Y$ is fixed, while $\cj G(Y) = \sum_{j \in {\rm I}(Y,\dl)} G_j(Y)$ has the law of $\cj G(y)$ at $y = Y$.
Therefore, applying \eqref{sing2} with $4\eta$ in place of $\eta$ and $r = r_Y$, we get
\[
\mu_\dl^+[\ld] \bigg(\bigg\{ \Big| \frac 14 \int_{B_1\backslash B_\dl} |u|^4 - Z\Big| \le \eta |\log \dl|^\frac12 \bigg\}\bigg)
\les \E\bigg[ \eta + \exp\Big(-16 c \frac{\eta^2 |\log \dl|}{(1 + {\rm M}(Y,\dl))^{24}}\Big) \bigg]  \les \eta,
\]

\noi
where in the last step we used \eqref{MY2}.
Thus, we finish the proof.
\end{proof}

\subsubsection{Proof of Theorem \ref{THM:main_singular}}
Before proceeding to the proof of Theorem \ref{THM:main_singular}, we need the following lemma.

\begin{lemma}
\label{LEM:tail}
Suppose that $\nu$ is a probability measure, and let $f$ be a Borel function such that
$$\int \exp(f) d\nu = 1.$$
Then, for every $r >0$,
\begin{equation} \label{eq:tail1}
\int_{\{f \le -r\}} \exp(f) d\nu \le \exp(-r) ,
\end{equation}
and
\begin{equation} \label{eq:tail2}
\nu(\{ f \ge r\} ) \le \exp(-r) .
\end{equation}
\end{lemma}
\begin{proof}
Since $\nu$ is a probability measure, we clearly have that
$$ \int_{\{f \le -r\}} \exp(f) d\nu \le \int \exp(-r) d\nu \le \exp(-r). $$
Moreover, since $\int \exp (f) d\nu =1$, by Markov's inequality,
$$ \nu(\{ f \ge r\}) \le \exp(-r) \int_{\{f \ge r\}} \exp(f) d\nu \le \exp(-r). $$
Thus, we finish the proof.
\end{proof}

Now, we are ready to prove Theorem \ref{THM:main_singular}.

\begin{proof}[Proof of Theorem \ref{THM:main_singular}]

Let $\rho_0$ be a limit point of $\rho_\eps$ as $\eps \to 0$. To show that $\rho_0$ and $\mu$ are singular, it is enough to find a sequence of sets $A_n$ such that
\begin{align}
\label{singular}
\mu(A_n) \to 1, \qquad \sum_{n \ge 1} \rho_0 (A_n) < \infty .
\end{align}
Indeed, if \eqref{singular} holds, then for the set
\[
A = \limsup_{n \to \infty} A_n
\]
we have $\mu (A) \ge \limsup_{n \to \infty} \mu (A_n) = 1$, while $\rho_0 (A) = 0$ by Borel--Cantelli. Thus $A$ is of full measure for $\mu$ and of zero measure for $\rho_0$, which concludes the proof of Theorem \ref{THM:main_singular}.

We define sets $A_n$ as follows,
$$
A_n := \Big\{ u: F(u,\dl_n) \le - \frac1n |\log \dl_n|^\frac12\Big\} \cap \{u: |u(\dl_n)| \le n \dl_n^{1-\frac d2}  \}, $$
where
\begin{align}\label{Fsing}
F(u,\dl_n) = - \frac 14 \int_{B_1\backslash B_{\dl_n}} |u|^4 dx - \log Z_{\dl_n}^+[u(\dl_n)] ,
\end{align}
and $\dl_n \downarrow 0$ is to be chosen later. Note that we have
$$A_n^c  = \bigg( \Big\{ F(u,\dl_n) > - \frac1n |\log \dl_n|^\frac12\Big\} \cap \{|u(\dl_n)| \le n \dl_n^{1-\frac d2}  \} \bigg) \cup \{ |u(\dl_n)| > n \dl_n^{1-\frac d2} \}.$$

We first check the first part of \eqref{singular}, i.e., for the Gaussian measure $\mu$, we would like to prove
\begin{align} \label{sin1}
\lim_{n \to \infty} \mu(A_n) = 1,
\end{align}

\noi
provided $\dl_n \downarrow 0$ fast enough.
By \eqref{GFFX_e} and \eqref{eq:law_R}, we have
\begin{align*}
 \mu(A_n^c)
\le &\ \mu\big(\{ |u(\dl_n)| > n \dl_n^{1-\frac d2} \}\big) \\
&\ + \mu\Big(\Big\{ |F(u,\dl_n)| \le \frac1n |\log \dl_n|^\frac12\Big\} \cap \{|u(\dl_n)| \le n \dl_n^{1-\frac d2}  \}\Big) \\
&\ + \mu\Big(\Big\{ F(u,\dl_n) >  \frac1n |\log \dl_n|^\frac12\Big\} \cap \{|u(\dl_n)| \le n \dl_n^{1-\frac d2}  \}\Big) \\
\les &\ \frac{1}{n} + \int_{|\ld| \le n \dl_n^{1-\frac d2}} \mu_{\dl_n}^+[\ld]\Big(\Big\{ \big| F(u,\dl_n) \big| \le \frac1n |\log \dl_n|^\frac12 \Big\}\Big) d\mu_{\dl_n}^\partial(\ld) \\
&\ + \int_{|\ld| \le n \dl_n^{1-\frac d2}} \mu_{\dl_n}^+[\ld]\Big(\Big\{ F(u,\dl_n) > \frac1n |\log \dl_n|^\frac12 \Big\}\Big) d\mu_{\dl_n}^\partial(\ld).
\end{align*}

\noi
We first claim that
\begin{align}
\label{dlstar}
\int_{|\ld| \le n \dl_n^{1-\frac d2}} \mu_{\dl_n}^+[\ld]\Big(\Big\{ \big| F(u,\dl_n) \big| \le \frac1n |\log \dl_n|^\frac12 \Big\}\Big) d\mu_{\dl_n}^\partial(\ld)  \les \frac1{n},
\end{align}

\noi
provided $\dl_n \le \dl_* (n) \ll 1$.
When $d = 4$ this is Proposition \ref{PROP:div} applied with $\eta = \frac1n$ and $A = n$, so that $\dl_*(n) = \dl_0 (n, \frac1n)$ suffices.
When $d > 4$, we instead apply Proposition \ref{PROP:div_h} with
$\eta = \eta_n := \frac{|\log \dl_n|^{\frac12}}{n \dl_n^{4-d}}$, which bounds the left-hand side of \eqref{dlstar} by $\eta_n^{\frac14}$. Since $d > 4$ we have $\dl_n^{d-4} |\log \dl_n|^{\frac12} \to 0$ as $\dl_n \downarrow 0$, so that $\eta_n \le n^{-4}$, and hence $\eta_n^{\frac14} \le \frac1n$, as soon as $\dl_n \le \dl_*(n)$ for a suitable $\dl_*(n) \ll 1$.
We also note from \eqref{Fsing} and \eqref{meas_cut} that
\[
\begin{split}
\int & \exp \big ( F(u,\dl_n)\big) d \mu_{\dl_n}^+[\ld] \\
& = \int \exp \bigg( - \frac 14 \int_{B_1\backslash B_{\dl_n}} |u|^4 - \log Z_{\dl_n}^+[\ld] \bigg) d \mu_{\dl_n}^+[\ld] \\
& = \frac1{Z_{\dl_n}^+[\ld]} \int \exp \bigg( - \frac 14 \int_{B_1\backslash B_{\dl_n}} |u|^4  \bigg) d \mu_{\dl_n}^+[\ld] \\
& = 1,
\end{split}
\]

\noi
which together with \eqref{eq:tail2} in Lemma \ref{LEM:tail} implies that
\[
\mu_{\dl_n}^+[\ld]\Big(\Big\{ F(u,\dl_n) > \frac1n |\log \dl_n|^\frac12 \Big\}\Big) \le \exp \bigg(-  \frac1n |\log \dl_n|^\frac12 \bigg) \le e^{-n},
\]

\noi uniformly in $\ld$, as long as  $\dl_n \le e^{-n^4}$.
 Choosing $\dl_n \le \min \big( \dl_*(n), e^{-n^4}\big)$ and collecting the above, we arrive at
\begin{align}
\label{An_est}
\mu (A_n) = 1 - \mu (A_n^c) =  1 + o (1),
\end{align}

\noi
as $n \to \infty$. Thus, we finish the proof of \eqref{sin1}.

We now move to the second part of \eqref{singular}.
By Proposition \ref{PROP:Markov0}, we have
$$(R_{\dl_n})_{\#} \rho_0 = \int \rho_{\dl_n}^+[\ld] \, \nu_{0,\dl_n}(d\lambda), $$
where $\nu_{0,\dl_n} = \Law(u(\dl_n))$ with $u$ distributed according to $\rho_0$.
Therefore, by \eqref{eq:rhoMarkov0},
\begin{align*}
  \ \rho_0(A_n)
& = \rho_0 \Big(\Big\{ F(u,\dl_n)  \le - \frac1n |\log \dl_n|^\frac12\Big\} \cap \{|u(\dl_n)| \le n \dl_n^{1-\frac d2}  \}\Big) \\
& = \int_{|\ld| \le n \dl_n^{1-\frac d2}} \rho_{\dl_n}^+[\ld]\Big(\Big\{F(u,\dl_n) \le - \frac1n |\log \dl_n|^\frac12\Big\} \Big) \nu_{0,\dl_n} (d \lambda) .
\end{align*}

\noi
We note here from the definition of $\rho_{\dl_n}^+[\ld]$ in \eqref{meas_cut}  that
\[
\begin{split}
\rho_{\dl_n}^+[\ld] & \Big(\Big\{F(u,\dl_n) \le - \frac1n |\log \dl_n|^\frac12\Big\} \Big)  \\
& = \frac1{Z_{\dl_n}^+[\ld]} \int_{F(u,\dl_n) \le - \frac1n |\log \dl_n|^\frac12} \exp \bigg( - \frac 14 \int_{B_1\backslash B_{\dl_n}} |u|^4  \bigg) d \mu_{\dl_n}^+[\ld]  \\
& = \int_{F(u,\dl_n) \le - \frac1n |\log \dl_n|^\frac12} \exp \big( F(u,\dl_n) \big) d \mu_{\dl_n}^+[\ld] \\
& \le \exp\Big(- \frac1n |\log \dl_n|^\frac12\Big),
\end{split}
\]

\noi
where again we used Lemma \ref{LEM:tail}.
Furthermore, we have
\begin{align}\label{sin2}
\begin{split}
\rho_0 (A_n) & \le \int_{|\ld| \le n \dl_n^{1-\frac d2}} \exp\Big(- \frac1n |\log \dl_n|^\frac12\Big) \nu_{0,\dl_n} (d \lambda) \\
& \le \exp\Big(- \frac1n |\log \dl_n|^\frac12\Big)  \le e^{-n},
\end{split}
\end{align}

\noi
where we used the fact that $\nu_{0,\dl_n}$ is a probability measure, and that $\dl_n \le e^{-n^4}$.
Finally, by \eqref{sin2} we conclude that
\[
\sum_{n \ge 1} \rho_0 (A_n) \le \sum_{n \ge 1} e^{-n} < \infty ,
\]

\noi
which proves the second part of \eqref{singular}.

Thus, we finish the proof of the theorem.
\end{proof}

\end{document}